\documentclass[a4paper,11pt,reqno]{amsbook}
\usepackage{etex}
\usepackage{amssymb,amsmath,mathrsfs,graphics,graphicx,mathptm,float}
\usepackage[T1]{fontenc}
\usepackage[english, francais]{babel}
\usepackage[all]{xy}

\theoremstyle{plain}
\newtheorem{thm}{Th\'eor\`eme}[chapter]
\newtheorem{pro}[thm]{Proposition}
\newtheorem{lem}[thm]{Lemme}
\newtheorem{cor}[thm]{Corollaire}

\theoremstyle{definition}

\newtheorem*{defis}{D\'efinitions}
\newtheorem{eg}[thm]{Exemple}
\newtheorem{egs}[thm]{Exemples}
\newtheorem{rem}[thm]{Remarque}

\def\transp #1{\vphantom{#1}^{\mathrm t}\! {#1}}
\def\og{\leavevmode\raise.3ex\hbox{$\scriptscriptstyle\langle\!\langle$~}}
\def\fg{\leavevmode\raise.3ex\hbox{~$\!\scriptscriptstyle\,\rangle\!\rangle$}}
\def\noteB#1#2{{\begin{small}#1\end{small}} \hfill {\begin{small}\pageref{#2}\end{small}} \\}
\def\noteBB#1#2#3{{\begin{small}#1\end{small}} \hfill {\begin{small}\pageref{#2},\,\pageref{#3}\end{small}} \hfill \\}

\addtocounter{section}{0}             
\numberwithin{equation}{chapter}       

\begin{document}
\selectlanguage{french}

\frontmatter

\thispagestyle{empty}

\vspace*{3.5cm}

\begin{flushright}
\textbf{Julie \textsc{D\'eserti}}
\end{flushright}

\vspace{3cm}

\noindent\hrulefill

\vspace{2cm}

\begin{center}
\begin{LARGE}
\begin{flushleft}
\textbf{AUTOMORPHISMES D'ENTROPIE}\\
\textbf{POSITIVE, LE CAS DES SURFACES}\\
\textbf{RATIONNELLES}
\end{flushleft}
\end{LARGE}
\end{center}

\vspace{2cm}

\noindent\hrulefill

\newpage

\noindent Julie \textsc{D\'eserti}

\noindent Institut de Math\'ematiques de Jussieu, Universit\'e Paris $7,$ Projet G\'eom\'etrie et Dynamique, Site Chevaleret, Case $7012,$ $75205$ Paris Cedex 13, France.

\noindent deserti@math.jussieu.fr

\vspace{6cm}

\noindent\hrulefill

\medskip


\medskip

\noindent\hrulefill

\newpage

\selectlanguage{french}

\begin{abstract}
Si $\mathcal{Z}$ est une surface complexe compacte munie d'un automorphisme d'entropie positive, un th\'eor\`eme de \textsc{Cantat} assure qu'ou bien la dimension de \textsc{Kodaira} de $\mathcal{Z}$ est nulle et dans ce cas $f$ est conjugu\'e \`a un automorphisme de l'unique mod\`ele minimal de $\mathcal{Z}$ qui doit \^etre un tore, une surface K$3$ ou une surface d'\textsc{Enriques}; ou bien la surface $\mathcal{Z}$ est rationnelle non minimale et $f$ est birationnellement conjugu\'ee \`a une transformation birationnelle du plan. Nous nous int\'eressons aux r\'esultats obtenus dans ce dernier cas. Apr\`es quelques rappels sur les transformations birationnelles, quelques rappels de g\'eom\'etrie alg\'ebrique et de dynamique, on \'evoquera les travaux de \textsc{Bedford} et \textsc{Kim} (\cite{BK1, BK2, BK3, BK4}) mais aussi ceux de \textsc{McMullen} (\cite{Mc}) et plus r\'ecemment celui de \textsc{Grivaux} and the author (\cite{DG}).

\noindent{\it Classification math\'ematique par sujets (2010). --- 14E07, 32H50, 37F10, 37B40, 37F50.}

\bigskip\bigskip

\selectlanguage{english}
\noindent \textsc{Abstract} Let $\mathcal{Z}$ be a complex compact surface which carries an automorphism $f$ of positive topological entropy. By \textsc{Cantat} either the \textsc{Kodaira} dimension of $\mathcal{Z}$ is zero and $f$ is conjugated to an automorphism on the unique minimal model of $\mathcal{Z}$ which is either a torus, or a K$3$ surface, or an \textsc{Enriques} surface; or $\mathcal{Z}$ is a non-minimal rational surface and $f$ is conjugated to a birational map of the complex projective plane. We deal with results obtained in this last case. After some recalls on birational maps, algebraic geometry and dynamic, we will speak about \textsc{Bedford} and \textsc{Kim}'s works (\cite{BK1, BK2, BK3, BK4}), but also about \textsc{McMullen} (\cite{Mc}) and more recently about \cite{DG}. 


\end{abstract}

\selectlanguage{french}

\mainmatter

\chapter*{Introduction}

\noindent Si $f$ est un automorphisme sur une surface complexe compacte $\mathcal{Z}$ d'entropie topologique positive, alors ou bien la dimension de \textsc{Kodaira} de $\mathcal{Z}$ est nulle et dans ce cas $f$ est conjugu\'e \`a un automorphisme de l'unique mod\`ele minimal de $\mathcal{Z}$ qui doit \^etre un tore, une surface~K$3$ ou une surface d'\textsc{Enriques}; ou bien la surface $\mathcal{Z}$ est rationnelle non minimale et $f$ est birationnellement conjugu\'ee \`a une transformation birationnelle du plan (\cite{Can1}). Le cas des surfaces K$3$ a \'et\'e en particulier \'etudi\'e par \cite{Can2, Mc2, Og, Si, Wa}. Un des premiers exemples donn\'e dans le contexte des surfaces rationnelles est d\^u \`a \textsc{Coble} (\cite{Co}); nous en reparlerons au Chapitre \ref{chapdyn}. Donnons un autre exemple bien connu: soient $\Lambda=\mathbb{Z}[\mathrm{i}]$ et $E=\mathbb{C}/\Lambda.$ Le groupe $\mathrm{SL}_2(\Lambda)$ agit lin\'eairement sur~$\mathbb{C}^2$ et pr\'eserve le r\'eseau $\Lambda\times\Lambda;$ par suite tout \'el\'ement $A$ de $\mathrm{SL}_2(\Lambda)$ induit un automorphisme $f_A$ sur~$E\times E$ qui commute avec $\iota(x,y)=(\mathrm{i}x,\mathrm{i}y).$ L'automorphisme $f_A$ se rel\`eve en un automorphisme $\widetilde{f_A}$ sur la d\'esingularis\'ee de~$(E\times E)/\iota,$ qui est une surface de \textsc{Kummer}. Cette surface est rationnelle et l'entropie de~$\widetilde{f_A}$ est positive d\`es que le module de l'une des valeurs propres de $A$ est strictement plus grand que $1.$

\medskip

\noindent On va s'int\'eresser aux surfaces obtenues en \'eclatant le plan projectif complexe en un nombre fini de points; ceci est justifi\'e par le th\'eor\`eme de \textsc{Nagata} (\cite{Na}, Theorem $5$): soient $\mathcal{Z}$ une surface rationnelle et $f$ un automorphisme sur $\mathcal{Z}$ tel que $f_*$ soit d'ordre infini; alors il existe une suite d'applications holomorphes~$\pi_{j+1}\colon \mathcal{Z}_{j+1}\to \mathcal{Z}_j$ telles que $\mathcal{Z}_1=\mathbb{P}^2(\mathbb{C}),$ $\mathcal{Z}_{N+1}=\mathcal{Z}$ et $\pi_{j+1}$ soit l'\'eclatement de $p_j\in\mathcal{Z}_j.$ De telles surfaces sont appel\'ees {\it surfaces rationnelles basiques}. N\'eanmoins une surface obtenue \`a partir de $\mathbb{P}^2(\mathbb{C})$ apr\`es des \'eclatements g\'en\'eriques n'a pas d'automorphisme non trivial (\cite{Hi, Ko}). 

\medskip

\noindent Dans \cite{Mc} \textsc{McMullen}, en s'appuyant sur des travaux de \textsc{Nagata} et \textsc{Harbourne}, donne un analogue du th\'eor\`eme de \textsc{Torelli} pour les surfaces K$3:$ il construit des automorphismes sur certaines surfaces rationnelles en prescrivant l'action de l'automorphisme sur les groupes de cohomologie de la surface. Ces surfaces sont des surfaces rationnelles poss\'edant, \`a facteur multiplicatif pr\`es, une unique $2$-forme m\'eromorphe $\Omega$ qui ne s'annule pas. Si $f$ est un automorphisme sur $\mathcal{Z}$ obtenu via cette construction, $f^*\Omega$ est proportionnelle \`a $\Omega$ et $f$ pr\'eserve le lieu des p\^oles de $\Omega.$ Notons que lorsqu'on projette $\mathcal{Z}$ sur le plan projectif complexe,~$f$ induit une transformation birationnelle qui pr\'eserve une cubique. 

\medskip

\noindent Dans \cite{BK1, BK2, BK3} le point de vue est quelque peu diff\'erent. Les auteurs consid\`erent des transformations birationnelles de $\mathbb{P}^2(\mathbb{C})$ et ajustent les coefficients de celles-ci afin de trouver une suite finie d'\'eclatements $\pi\colon \mathcal{Z}\to\mathbb{P}^2 (\mathbb{C})$ telle que l'application indui\-te~$f_\mathcal{Z}=\pi^{-1}f\pi$ soit un automorphisme de $\mathcal{Z}.$ Alors ils calculent le premier degr\'e dynamique et \'etablissent si la famille est triviale ou non. Certains de leurs travaux sont inspir\'es par \cite{HV, HV2, Ta1, Ta2, Ta3}. Notons que la construction de McMullen ne permet pas d'obtenir tous les automorphismes construits par \textsc{Bedford} et \textsc{Kim}. Ces derniers produisent en particuliers des exemples ne pr\'eservant aucune courbe mais aussi des familles continues. Ils montrent des propri\'et\'es de nature dynamique comme par exemple la coexistence de domaines de rotation de rang $1$ et $2.$

\medskip

\noindent Dans \cite{DG} les auteurs utilisent la th\'eorie g\'en\'erale des d\'eformations des vari\'et\'es complexes pour d\'ecrire explicitement les petites d\'eformations des surfaces rationnelles. Cela leur permet de donner un crit\`ere simple pour compter le nombre de param\`etres d'une d\'eformation d'une surface rationnelle basique donn\'ee. Ils \'etudient ensuite une famille de transformations birationnelles $(\Phi_n)_{n \geq 2};$ ils construisent, pour tout $n,$ deux points infiniment proches $\widehat{P}_1$ et $\widehat{P}_2$ de $\mathbb{P}^2(\mathbb{C})$ ayant la propri\'et\'e suivante: $\Phi_n$ induit un isomorphisme entre $\mathbb{P}^2(\mathbb{C})$ \'eclat\'e en $\widehat{P}_1$ et~$\mathbb{P}^2( \mathbb{C})$ \'eclat\'e en $\widehat{P}_2.$ Ensuite ils donnent des conditions g\'en\'erales portant sur $\Phi_n$ permettant de trouver des automorphismes $\varphi$ de $\mathbb{P}^2(\mathbb{C})$ tels que $\varphi \, \Phi_n$ soit un automorphisme de $\mathbb{P}^2(\mathbb{C})$ \'eclat\'e en $\widehat{P}_1,$ $\varphi(\widehat{P}_2),$ $(\varphi \, \Phi_n) \, \varphi(\widehat{P}_2),$ $\ldots,$ $(\varphi \, \Phi_n)^k \, \varphi(\widehat{P}_2)=\widehat{P}_1.$ Ils appliquent cette m\'ethode \`a d'autres transformations que $\Phi_n.$ Cette d\'emarche n'\'etant pas propre \`a $\Phi_n$ ils l'appliquent \`a d'autres transformations.

\bigskip

\noindent Dans un premier chapitre nous rappelons des propri\'et\'es des automorphismes polynomiaux de $\mathbb{C}^2$ et des transformations birationnelles du plan projectif complexe. Dans le second nous donnons quelques r\'esultats de g\'eom\'etrie birationnelle comme le th\'eor\`eme de factorisation de \textsc{Zariski}, la notion de dimension de \textsc{Kodaira}, la description des surfaces complexes compactes k\"{a}hl\'eriennes. Ensuite nous rappelons des \'enonc\'es de lin\'earisation, des propri\'et\'es des automorphismes d'entropie positive, la notion de composantes de \textsc{Fatou} etc. Le chapitre~\ref{mcmullen} est consacr\'e au travail de \textsc{McMullen} (\cite{Mc}), les chapitres \ref{chapbedkim1} et \ref{chapbedkim2} aux r\'esultats de \textsc{Bedford} et \textsc{Kim} (\cite{BK2, BK3, BK4}) et le dernier \`a \cite{DG}. 

\subsection*{Remerciements} 
Je tiens \`a remercier D. \textsc{Cerveau} pour sa g\'en\'erosit\'e, ses encouragements et son enthousiasme permanents. Merci \`a P. \textsc{Sad} pour son invitation au sud de l'\'equateur, les s\'eminaires bis... et \`a J. \textsc{Grivaux} pour sa pr\'ecieuse aide. Merci \`a l'IMPA pour son accueil formidable et au CNRS pour la d\'el\'egation qui m'a permis de r\'ealiser ce projet.

\tableofcontents

\chapter{Groupe des automorphismes polynomiaux de $\mathbb{C}^2$ et groupe de \textsc{Cremona}}

\section{Groupe des automorphismes polynomiaux de $\mathbb{C}^2$}

\noindent Un {\it automorphisme polynomial de $\mathbb{C}^2$}\label{ind1} est une 
application bijective de la forme suivante
\begin{align*}
& \mathbb{C}^2\to\mathbb{C}^2,\,\,(x,y)\mapsto(f_1(x,y),f_2(x,y)), 
&& f_i\in\mathbb{C}[x,y].
\end{align*}

\noindent L'ensemble de ces automorphismes forme un groupe, 
appel\'e {\it groupe des automorphismes polynomiaux de~$\mathbb{C}^2,$}\label{ind2}
que nous noterons $\mathrm{Aut}[\mathbb{C}^2].$

\noindent Introduisons deux sous-groupes naturels de $\mathrm{Aut}[
\mathbb{C}^2]$
$$\mathrm{E}=\{(x,y)\mapsto(\alpha x+P(y),\beta y+\gamma)\,\vert
\,\alpha,\,\beta,\,\gamma\in\mathbb{C},\,\alpha\beta\not=0,\,
P\in\mathbb{C}[y]\},$$
$$\mathrm{A}=\{(x,y)\mapsto(a_1x+b_1y+c_1,a_2x+b_2y+c_2)\,\vert\,
a_i,\,b_i,\,c_i\in\mathbb{C},\,a_1b_2-a_2b_1\not=0\}.$$
On dit que $\mathrm{A}$ est le {\it groupe des automorphismes affines}\label{ind3}
et $\mathrm{E}$ le {\it groupe des automorphismes \'el\'ementaires}\label{ind4};
ce sont les exactement les automorphismes qui pr\'eservent le 
feuilletage \og horizontal\fg\, $\mathrm{d}y=0.$

\noindent Soit $\mathrm{S}=\mathrm{A}\cap\mathrm{E}$ le {\it groupe des
automorphismes affines triangulaires}\label{ind5} 
$$\mathrm{S}=\{(x,y)\mapsto(a_1x+b_1y+c_1,b_2y+c_2)\,\vert\,a_i,\,
b_i,\,c_i\in\mathbb{C},\,a_1b_2\not=0\}.$$

\noindent Le groupe $\mathrm{Aut}[\mathbb{C}^2]$ a une structure
de produit amalgam\'e.

\begin{thm}[\cite{Ju}]
{\sl Le groupe $\mathrm{Aut}[\mathbb{C}^2]$ est le produit 
amalgam\'e des sous-groupes~$\mathrm{A}$ et $\mathrm{E}$ 
le long de leur intersection $\mathrm{S}.$ 

\noindent Autrement dit tout
\'el\'ement $\varphi$ de $\mathrm{Aut}[\mathbb{C}^2]\setminus
\mathrm{S}$ s'\'ecrit $(a_1)e_1\ldots a_n(e_n)$ avec $a_i$
dans $\mathrm{A}\setminus\mathrm{E},$ $e_i$ dans $\mathrm{E}
\setminus\mathrm{A};$ de plus, cette \'ecriture est unique 
modulo les relations naturelles
\begin{align*}
& a_ie_i=(a_is)(s^{-1}e_i),&& e_{i-1}a_i=(e_{i-1}s')(s'^{-1}a_i), 
&& s,\,s'\in\mathrm{S}.
\end{align*} }
\end{thm}

\begin{rem}
Hormis dans certains cas (cas r\'esonants) un automorphisme \'el\'ementaire
est conjugu\'e \`a un automorphisme affine. Cependant il existe dans 
$\mathrm{Aut}[\mathbb{C}^2]$ des \'el\'ements non conjugu\'es \`a des 
automorphismes affines; c'est par exemple le cas des {\it automorphismes de 
\textsc{H\'enon} g\'en\'eralis\'es}\label{ind6}
\begin{align*}
& (y,P(y)-\delta x), && \delta\in\mathbb{C}^*,
\,P\in\mathbb{C}[y],\,\deg P\geq 2, 
\end{align*}
\noindent qui s'\'ecrivent aussi $(y,x)(-\delta x +P(y),y)$ avec $(y,x)$
dans $\mathrm{A}\setminus\mathrm{E}$ et $(-\delta x +P(y),y)$
dans $\mathrm{E}\setminus\mathrm{A}.$
\end{rem}

\noindent \`A l'aide du th\'eor\`eme de \textsc{Jung}, \textsc{Friedland} et \textsc{Milnor}
ont d\'emontr\'e l'\'enonc\'e suivant.

\begin{pro}[\cite{FM}]
{\sl Soit $f$ un automorphisme polynomial de $\mathbb{C}^2.$ On
a l'alternative
\smallskip
\begin{itemize}
\item $f$ est conjugu\'e \`a un \'el\'ement de $\mathrm{E};$
\smallskip
\item $f$ est conjugu\'e \`a une compos\'ee d'applications de
\textsc{H\'enon} g\'en\'eralis\'ees, {\it i.e.} $$f=\varphi g_1\ldots
g_n\varphi^{-1}$$ o\`u $\varphi$ est dans $\mathrm{Aut}[\mathbb{C}^2]$
et les $g_i$ des applications de \textsc{H\'enon} g\'en\'eralis\'ees. 
\end{itemize}}
\end{pro}

\noindent Dans le premier cas $f$ est dit {\it de type \'el\'ementaire}\label{ind7},
dans le second {\it de type \textsc{H\'enon}}\label{ind8}; $\mathcal{H}$ d\'esigne le 
semi-groupe des compos\'ees d'applications de \textsc{H\'enon} g\'en\'eralis\'ees.

\begin{rem}
Les automorphismes affines sont de type \'el\'ementaire (triangulation
des matrices).
\end{rem}

\noindent Nous allons voir que cette dichotomie type \'el\'ementaire/type
\textsc{H\'enon} se poursuit:

\begin{itemize}
\item Tout \'el\'ement de type \'el\'ementaire pr\'eserve une fibration
rationnelle alors qu'un automorphisme de type \textsc{H\'enon} n'en pr\'eserve pas.

\item Les automorphismes de type \'el\'ementaire ont tous un 
centralisateur non d\'enombrable (en fait presque tous se plongent
dans un groupe \`a un param\`etre de $\mathrm{Aut}[\mathbb{C}^2]$) 
alors que le centralisateur d'un automorphisme de type \textsc{H\'enon} est 
d\'enombrable, souvent m\^eme r\'eduit \`a ses it\'er\'es (\cite{La}).

\item Un automorphisme de type \textsc{H\'enon} poss\`ede une infinit\'e de points 
p\'eriodiques hyperboliques.
\end{itemize}

\noindent Le {\it degr\'e alg\'ebrique}\label{ind9} d'un automorphisme polynomial
\begin{align*}
&f\colon\mathbb{C}^2\to\mathbb{C}^2, &&(x,y)\mapsto(f_1(x,y),f_2(x,y)),
&& f_i\in\mathbb{C}[x,y]
\end{align*}

\noindent est d\'efini par $\deg f=\max(\deg f_1,\deg f_2).$ Ce degr\'e 
n'est pas un invariant dynamique (en g\'en\'eral si $f,$ $g$ sont 
dans~$\mathrm{Aut}[\mathbb{C}^2],$ alors $\deg(fgf^{-1})\not=\deg g$); on
introduit donc un second degr\'e, le degr\'e dynamique (\cite{Di, FM}),
d\'efini par $$\lambda(f)=\lim_{n\to +\infty}(\deg f^n)^{1/n}.$$ On
peut v\'erifier que celui-ci est un invariant dynamique (en fait il
existe deux constantes strictement positives~$\alpha$ et $\beta$ 
telles que $$\alpha\deg g^n\leq\deg(fg^nf^{-1})\leq\beta\deg g^n$$
pour tout entier $n$).

\noindent Si $f$ est un automorphisme \'el\'ementaire, on a $\deg f^n=\deg f$ d'o\`u $\lambda(f)=1.$ 
Au contraire, pour un \'el\'ement~$f$ de $\mathcal{H},$ on a $\lambda(f)=\deg f;$
plus pr\'ecis\'ement si $f=f_1\ldots f_n,$ avec~$f_i$ application
de \textsc{H\'enon} g\'en\'eralis\'ee, on a $\deg f=\displaystyle\prod_i\deg f_i\geq 2$
(\emph{voir} \cite{FM}). Par suite
\smallskip
\begin{itemize}
\item $f$ est de type \'el\'ementaire si et seulement si 
$\lambda(f)=1;$
\smallskip
\item $f$ est de type \textsc{H\'enon} si et seulement si $\lambda(f)>1.$
\end{itemize}
\smallskip

\noindent La plupart des propri\'et\'es qui pr\'ec\`edent ont 
\'et\'e d\'emontr\'ees en s'appuyant sur la structure de produit
amalgam\'e de~$\mathrm{Aut}[\mathbb{C}^2].$ Plus pr\'ecis\'ement
puisque $\mathrm{Aut}[\mathbb{C}^2]=\mathrm{A}\ast_\mathrm{S}
\mathrm{E}$ la th\'eorie de \textsc{Bass}-\textsc{Serre} assure que $\mathrm{Aut}
[\mathbb{C}^2]$ agit de mani\`ere non triviale par translation
\`a gauche sur un arbre~$\mathcal{T}$ (\emph{voir} \cite{Se}).
L'arbre $\mathcal{T}$ est d\'efini comme suit: l'ensemble des
sommets est l'union disjointe des classes \`a gauche 
$(\mathrm{Aut}[\mathbb{C}^2])/\mathrm{A}$ et $(\mathrm{Aut}
[\mathbb{C}^2])/\mathrm{E}$ et celui des ar\^etes l'union 
disjointe des classes \`a gauche $(\mathrm{Aut}[\mathbb{C}^2])/
\mathrm{S}.$ Pour tout automorphisme polynomial $f$ l'ar\^ete
$f\mathrm{S}$ relie les sommets $f\mathrm{A}$ et $f\mathrm{E}.$
On h\'erite ainsi d'une action fid\`ele de $\mathrm{Aut}
[\mathbb{C}^2]$ dans le groupe des isom\'etries de l'arbre
$\mathcal{T}.$ L'\'etude de cette action permet de d\'emontrer
de nombreuses propri\'et\'es sur~$\mathrm{Aut}[\mathbb{C}^2].$

\section{Groupe des transformations birationnelles}

\noindent Une {\it transformation rationnelle}\label{ind10} de $\mathbb{P}^2
(\mathbb{C})$ dans lui-m\^eme est une transformation de la 
forme
\begin{align*}
&\mathbb{P}^2(\mathbb{C})\dashrightarrow\mathbb{P}^2(\mathbb{C}), 
&& (x:y:z)\mapsto(f_0(x,y,z):f_1(x,y,z):f_2(x,y,z))
\end{align*}

\noindent o\`u les $f_i$ sont des polyn\^omes homog\`enes de 
m\^eme degr\'e sans facteur commun.

\noindent Une {\it transformation birationnelle}\label{ind11} du plan projectif complexe
dans lui-m\^eme est une transformation rationnelle qui admet un inverse 
rationnel.

\begin{egs}
\begin{itemize}
\item Un automorphisme de $\mathbb{P}^2(\mathbb{C}),$ {\it i.e.} un 
\'el\'ement de $\mathrm{PGL}_3(\mathbb{C}),$ est une transformation
birationnelle.\smallskip

\item Tout automorphisme polynomial de $\mathbb{C}^2$ se prolonge en 
une transformation birationnelle du plan projectif complexe; par 
exemple on a dans le cas des automorphismes \'el\'ementaires
\begin{align*}
&\mathbb{P}^2(\mathbb{C})\dashrightarrow\mathbb{P}^2(\mathbb{C}),
&&(x:y:z)\mapsto(\alpha xz^{n-1}+z^nP\left(\frac{y}{z}\right):\beta yz^{n-1}+\gamma
z^n:z^n)
\end{align*}
\noindent o\`u $\alpha,$ $\beta,$ $\gamma$ sont des \'el\'ements de 
$\mathbb{C}$ tels que $\alpha\beta\not=0$ et $P$ un polyn\^ome de 
$\mathbb{C}[y]$ de degr\'e~$n,$ alors que dans le cas des automorphismes
de \textsc{H\'enon} g\'en\'eralis\'e on a
\begin{align*}
&\mathbb{P}^2(\mathbb{C})\dashrightarrow\mathbb{P}^2(\mathbb{C}),
&&(x:y:z)\mapsto(yz^{n-1}:z^nP\left(\frac{y}{z}\right)-\delta xz^{n-1}:
z^n),
\end{align*}

\noindent o\`u $\delta$ d\'esigne un complexe non nul et $P$ un 
polyn\^ome de degr\'e sup\'erieur ou \'egal \`a $2.$\smallskip

\item La transformation 
\begin{align*}
& \sigma\colon\mathbb{P}^2(\mathbb{C})\dashrightarrow\mathbb{P}^2(\mathbb{C}),
&& (x:y:z)\mapsto(yz:xz:xy)
\end{align*}
\noindent est rationnelle; elle s'\'ecrit aussi $\left(\frac{1}{x}:\frac{1}{y}:
\frac{1}{z}\right).$ On 
remarque en particulier que c'est une involution, elle est donc birationnelle.  
\end{itemize}
\end{egs}

\noindent On d\'esigne par $\mathrm{Bir}(\mathbb{P}^2)$
le {\it groupe des transformations birationnelles du plan projectif 
dans lui-m\^eme} encore appel\'e {\it groupe de \textsc{Cremona}}\label{ind12}; les \'el\'ements 
de $\mathrm{Bir}(\mathbb{P}^2)$ s'appellent aussi
{\it transformations de \textsc{Cremona}}\label{ind13}.

\noindent Le groupe $\mathrm{Aut}(\mathbb{P}^2(\mathbb{C}))$ des 
automorphismes de $\mathbb{P}^2(\mathbb{C})$ est un sous-groupe de 
$\mathrm{Bir}(\mathbb{P}^2),$ de m\^eme que $\mathrm{Aut}[\mathbb{C}^2].$

\noindent Une transformation birationnelle
\begin{align*}
&f\colon\mathbb{P}^2(\mathbb{C})\dashrightarrow\mathbb{P}^2(\mathbb{C}), 
&& (x:y:z)\mapsto(f_0(x,y,z):f_1(x,y,z):f_2(x,y,z))
\end{align*}

\noindent a un {\it ensemble d'ind\'etermination}\label{ind14} $\mathrm{Ind}\,f$
d\'efini par les z\'eros communs des $f_i,$ on dit aussi que ce sont les
{\it points \'eclat\'es}\label{ind15} par~$f.$ Elle poss\`ede aussi un {\it 
ensemble exceptionnel}\label{ind16} $\mathrm{Exc}\,f$ dont les composantes
sont les {\it courbes contract\'ees}\label{ind17} par $f;$ c'est \og l'image\fg\, par $f^{-1}$ de $\mathrm{Ind}\,f^{-1}.$ Lorsque $f$ est birationnelle, $\mathrm{Exc}\,f$ est le lieu des z\'eros de $\det\,\mathrm{jac}\,f.$ 

\begin{egs}
\begin{itemize}
\item Si $f$ est un automorphisme de $\mathbb{P}^2(\mathbb{C}),$
alors $f$ est du type
$$(x:y:z)\mapsto(a_0x+b_0y+c_0z:a_1x+b_1y+c_1z:a_2x+b_2y+c_2z)$$
\noindent avec $\left\vert\begin{array}{ccc}
a_0 & b_0 & c_0\\
a_1 & b_1 & c_1\\
a_2 & b_2 & c_2
\end{array}
\right\vert\not=0$ et $\mathrm{Ind}\,f=\mathrm{Exc}\,f=\emptyset.$
\smallskip

\item Soit $f\colon\mathbb{P}^2(\mathbb{C})\dashrightarrow
\mathbb{P}^2(\mathbb{C})$ le prolongement
d'un automorphisme \'el\'ementaire \`a $\mathbb{P}^2(\mathbb{C})$
$$(x:y:z)\mapsto(\alpha xz^{n-1}+z^n
P\left(\frac{y}{z}\right):\beta yz^{n-1}+\gamma z^n:z^n),$$
\noindent avec $\alpha,\,\beta,\,\gamma\in\mathbb{C},\,\alpha
\beta\not=0,\,P\in\mathbb{C}[y],\,\deg P=n,$ alors 
\begin{align*}
&\mathrm{Exc}\,f=\{z=0\} &&\text{ et } && \mathrm{Ind}\, f=\{(1:0:0)\}.
\end{align*}
De m\^eme le prolongement d'un automorphisme de \textsc{H\'enon} g\'en\'eralis\'e \`a $\mathbb{P}^2(\mathbb{C})$ contracte une unique droite et \'eclate un seul
point situ\'e sur cette droite. 
\smallskip

\item Consid\'erons l'involution $\sigma\colon\mathbb{P}^2(\mathbb{C})
\dashrightarrow\mathbb{P}^2(\mathbb{C}),$ $(x:y:z)\mapsto(yz:xz:xy).$
On remarque que $$\mathrm{Ind}\,\sigma=\{(1:0:0),\,(0:1:0),\,(0:0:1)\}$$
et $\mathrm{Exc}\,\sigma=\{x=0\}\cup\{y=0\}\cup\{z=0\}.$ Plus pr\'ecis\'ement
$\sigma$ contracte la droite d'\'equation $x=0$ (resp. $y=0,$ resp $z=0$)
sur le point $(1:0:0)$ (resp. $(0:1:0),$ resp. $(0:0:1)$). Puisque 
$\sigma$ est une involution~$\sigma$ \'eclate le point $(1:0:0)$ 
(resp. $(0:1:0),$ resp. $(0:0:1)$) sur la droite $x=0$ (resp. $y=0,$
resp. $z=0$).
\end{itemize}
\end{egs}

\noindent L'application d'{\it \'eclatement}\label{ind26} en un point est 
un exemple de transformation birationnelle. Soient~$\mathcal{Z}$ une surface et $p$
un point de $\mathcal{Z}.$ Il existe une surface $\widetilde{\mathcal{Z}}$ et un 
morphisme $\pi\colon\widetilde{\mathcal{Z}}\to\mathcal{Z}$ tels que
\begin{itemize}
\item $\mathrm{E}=\pi^{-1}(p)$ soit isomorphe \`a $\mathbb{P}^1(\mathbb{C});$

\item la restriction $\pi_{\vert\widetilde{\mathcal{Z}}\setminus\mathrm{E}}\colon \widetilde{\mathcal{Z}}\setminus \mathrm{E}\to\mathcal{Z}\setminus\{p\}$ de $\pi$ \`a $\widetilde{\mathcal{Z}}\setminus\mathrm{E}$ soit un isomorphisme.
\end{itemize}

\noindent \`A isomorphisme pr\`es $\widetilde{\mathcal{Z}}$ et $\pi$ sont uniques. On dit que $\pi$ est l'application d'\'eclatement en~$p$ et $\widetilde{\mathcal{Z}}$ est l'\'eclat\'e de $\mathcal{Z}$ en $p.$ On dit aussi qu'on passe de $\mathcal{Z}$ \`a $\widetilde{\mathcal{Z}}$ en \'eclatant le point $p$ et qu'on passe de $\widetilde{\mathcal{Z}}$ \`a $\mathcal{Z}$ en contractant la courbe $\mathrm{E}.$ La courbe rationnelle $\mathrm{E}$ est appel\'ee {\it diviseur exceptionnel}\label{ind26b}.

\noindent Une pr\'esentation concr\`ete de l'\'eclatement 
de $\mathbb{C}^2$ en $0$ est la suivante: on 
consid\`ere $$\Gamma=\{((x,y),[\xi:\eta])\in
\mathbb{C}^2\times\mathbb{P}^1(\mathbb{C})\,\vert\,
x\eta=y\xi\}$$ avec la projection $\pi$ sur le premier 
facteur; le couple $(\Gamma,\pi)$ est l'\'eclatement de 
$\mathbb{C}^2$ en $0.$ Notons que $\pi^{-1}\colon
\mathbb{C}^2\setminus\{0\}\to\Gamma$ est l'application 
d\'efinie par $\pi^{-1}(x,y)=((x,y),[x:y]);$ on peut 
aussi \'ecrire $$\pi^{-1}(x,y)=((x,y),[1:y/x])=((x,y),
[x/y:1]).$$ Ces deux repr\'esentations permettent de 
d\'efinir des cartes $\mathcal{U}'=\mathbb{C}^2_{s,
\eta}$ et $\mathcal{U}''=\mathbb{C}^2_{\xi,t}$ o\`u 
les coordonn\'ees sont d\'efinies \`a l'aide de $\pi$
\begin{align*}
&\pi'\colon(s,\eta)\mapsto(s,s\eta)=(x,y), &&
\pi''\colon(\xi,t)\mapsto(\xi t,t)=(x,y).
\end{align*}

\noindent Par suite $\Gamma=\mathcal{U}'\cup\mathcal{U}''$
et $\mathrm{E}'=\mathrm{E}\cap\mathcal{U}'=\{s=0\}$ est \'equivalent \`a
$\mathbb{C}.$ Dans l'\'eclat\'e
la transform\'ee stricte (\emph{voir} Chapitre \ref{ga}) de $\{x=0\}$ (resp. $\{y=0\}$)
dans $\Gamma$ est donn\'ee par $\{\xi=0\}$ (resp.~$\{\eta=0\}$);
il s'en suit qu'on utilise le syst\`eme de coordonn\'ees 
$(\xi,t)$ (resp. $(s,\eta)$) si on veut travailler au 
voisinage de~$\{x=0\}$ (resp. $\{y=0\}$).

\noindent Il existe un th\'eor\`eme de g\'en\'eration pour $\mathrm{Bir}
(\mathbb{P}^2).$

\begin{thm}[\cite{Ca, No, No2, No3}]\label{nono}
{\sl Le groupe de \textsc{Cremona} est engendr\'e par~$\sigma$ et $\mathrm{Aut}(\mathbb{P}^2
(\mathbb{C})).$}
\end{thm}

\noindent On peut reformuler cet \'enonc\'e d'une autre fa\c{c}on; avant 
remarquons que les transformations birationnelles d'une surface $\mathcal{Z}$
dans elle-m\^eme forment un groupe que l'on note $\mathrm{Bir}(\mathcal{Z}).$

\begin{thm}[\cite{Is}]
{\sl Consid\'erons le groupe $\mathrm{Bir}(\mathbb{P}^1(\mathbb{C})
\times\mathbb{P}^1(\mathbb{C}))$ des transformations birationnelles 
de $\mathbb{P}^1(\mathbb{C})
\times\mathbb{P}^1(\mathbb{C})$ dans lui-m\^eme, $\pi\colon\mathbb{P}^1(\mathbb{C})
\times\mathbb{P}^1(\mathbb{C})\to\mathbb{P}^1(\mathbb{C})$ la 
projection sur le premier facteur et $\tau\colon(u,v)\mapsto(v,u).$ 
Le grou\-pe~$\mathrm{Bir}(\mathbb{P}^1(\mathbb{C})\times\mathbb{P}^1
(\mathbb{C}))$ est engendr\'e par $\tau$ et le groupe des transformations
qui pr\'eservent la fibration $\pi.$}
\end{thm}

\noindent Contrairement \`a $\mathrm{Aut}[\mathbb{C}^2]$
le groupe $\mathrm{Bir}(\mathbb{P}^2)$ n'est pas le 
produit amalgam\'e de deux groupes le long de leur intersection
(\cite{Wr}). Les techniques utilis\'ees pour d\'ecrire ces deux 
groupes sont donc essentiellement diff\'erentes. 

\noindent Il y a une notion de premier degr\'e dynamique
sur $\mathrm{Bir}(\mathbb{P}^2)$ qui prolonge la notion 
d\'efinie sur~$\mathrm{Aut}[\mathbb{C}^2].$ Si $f\colon\mathbb{P}^2
(\mathbb{C})\dashrightarrow\mathbb{P}^2(\mathbb{C}),$ $(x:y:z)
\mapsto(f_0(x,y,z):f_1(x,y,z):f_2(x,y,z))$ d\'esigne une transformation
birationnelle, le {\it degr\'e alg\'ebrique}\label{ind18} de $f$ est par d\'efinition 
\'egal au degr\'e des~$f_i.$ Quant au {\it degr\'e dynamique}\label{ind19} il est donn\'e
par la m\^eme d\'efinition, \`a savoir par $$\lambda(f)=\displaystyle\lim_{n
\to +\infty}(\deg f^n)^{1/n}.$$

\noindent Soit $f$ dans $\mathrm{Bir}(\mathbb{P}^2);$ la 
transformation $f$ est dite {\it alg\'ebriquement stable}\label{ind20} s'il n'existe
pas de courbe $\mathcal{C}$ dans~$\mathbb{P}^2(\mathbb{C})$
telle que $f^k(\mathcal{C})$ appartienne \`a $\mathrm{Ind}\,f$
pour un certain entier $k\geq 0.$ Autrement dit une transformation
est alg\'ebriquement stable si la situation suivante n'arrive pas

\begin{figure}[H]
\begin{center}
\input{as.pstex_t}
\end{center}
\end{figure}

\noindent Les deux conditions qui suivent sont \'equivalentes (\cite{DiFa})
\smallskip
\begin{itemize}
\item il n'existe pas de courbe $\mathcal{C}$ dans $\mathbb{P}^2
(\mathbb{C})$ telle que $f^k(\mathcal{C})$ appartienne \`a 
$\mathrm{Ind}\,f$ pour un certain entier $k\geq 0;$

\item pour tout entier $n$ on a $\deg f^n=(\deg f)^n.$
\end{itemize}

\begin{egs}
\begin{itemize}
\item Un automorphisme de $\mathbb{P}^2(\mathbb{C})$ est alg\'ebriquement stable.

\item L'involution $\sigma\colon\mathbb{P}^2(\mathbb{C})
\dashrightarrow\mathbb{P}^2(\mathbb{C}),$ $(x:y:z)
\mapsto(yz:xz:xy)$ n'est pas alg\'ebriquement stable:
si $\mathcal{C}$ appartient \`a $\mathrm{Exc}\,\sigma$
alors $\sigma(\mathcal{C})$ est un point de $\mathrm{Ind}\,\sigma;$
par ailleurs $\deg\sigma^2=1$ alors que $(\deg\sigma)^2=4.$

\item Si $A$ est un \'el\'ement \og g\'en\'erique\fg\, de $\mathrm{Aut}(\mathbb{P}^2(\mathbb{C})),$ alors $A\sigma$ est alg\'ebriquement stable (\cite{CeDe}). 
\end{itemize}
\end{egs}

\noindent D'apr\`es \textsc{Diller} et \textsc{Favre} toute transformation birationnelle sur 
une surface complexe compacte est alg\'ebriquement stable modulo conjugaison birationnelle.

\begin{pro}[\cite{DiFa}]
{\sl Soit $f\colon \mathcal{Z}\dashrightarrow \mathcal{Z}$ une transformation birationnelle sur
une surface complexe compacte $\mathcal{Z}.$ Il existe une suite d'\'eclatements 
$\pi\colon\mathcal{X}\to\mathcal{Z}$ telle que~$\pi^{-1}f\pi$ soit alg\'ebriquement stable.}
\end{pro}

\begin{proof}[{\sl Id\'ee de d\'emonstration}]
Raisonnons par l'absurde, {\it i.e.} supposons que $f$ ne soit pas alg\'ebriquement 
stable; il existe alors une cour\-be~$\mathcal{C}$ et un 
entier $k$ tels que
\begin{figure}[H]
\begin{center}
\input{as4.pstex_t}
\end{center}
\end{figure}

\noindent L'id\'ee mise en \oe uvre par \textsc{Diller} et 
\textsc{Favre} est la suivante: quitte \`a \'eclater les $p_i$
l'image de~$\mathcal{C}$ est, pour $i=1,$ $\ldots,$ 
$k,$ une courbe. En renouvelant un nombre fini de fois 
cette op\'eration on obtient le r\'esultat annonc\'e.
\end{proof}

\noindent Lorsque $f$ est une transformation de \textsc{Cremona}
 alg\'ebriquement stable, on a $\lambda(f)=\deg f.$

\bigskip

\noindent Nous allons voir que pour $\mathrm{Bir}(\mathbb{P}^2)$
il n'y a pas d'analogue \`a la dichotomie automorphisme de 
\textsc{H\'enon}/automorphisme \'el\'ementaire. Pour ce faire introduisons
la famille de transformations birationnelles
$(f_{\alpha,\beta})$ donn\'ee par 
\begin{align*}
&\mathbb{P}^2(\mathbb{C})\dashrightarrow\mathbb{P}^2(\mathbb{C}),
&&(x:y:z)\mapsto((\alpha x+y)z:\beta y(x+z):z(x+z)), &&\alpha,\,
\beta\in\mathbb{C}^*
\end{align*}

\noindent soit dans la carte affine $z=1$
$$f_{\alpha,\beta}(x,y)=\left(\frac{\alpha x+y}{x+1},\beta y\right).$$

\begin{thm}[\cite{De}]\label{linearisation}
{\sl Le premier degr\'e dynamique de $f_{\alpha,\beta}$ vaut $1;$
plus pr\'ecis\'e\-ment~$\deg f_{\alpha,\beta}^n\sim n.$

\noindent Supposons que $\alpha$ et $\beta$ soient de module $1$
et g\'en\'eriques. Si $g$ commute \`a $f_{\alpha,\beta},$
alors $g$ est une puissance de $f_{\alpha,\beta};$ en particulier,
le centralisateur de $f_{\alpha,\beta}$ est d\'enombrable.

\noindent Les \'el\'ements $f_{\alpha,\beta}^2$ poss\`edent 
deux points fixes $m_1,$ $m_2$ et
\smallskip 
\begin{itemize}
\item il existe un voisinage $\mathcal{V}_1$ de $m_1$ sur lequel
$f_{\alpha,\beta}$ est conjugu\'e \`a $(\alpha x,\beta y);$
en particulier l'adh\'erence de l'orbite d'un point de 
$\mathcal{V}_1$ sous l'action de $f_{\alpha,\beta}$ est un tore
de dimension~$2;$\smallskip

\item il existe un voisinage $\mathcal{V}_2$ de $m_2$ tel que 
la dynamique de $f_{\alpha,\beta}^2$ soit localement lin\'eaire
sur $\mathcal{V}_2;$ l'adh\'erence de l'orbite g\'en\'erique d'un 
point de $\mathcal{V}_2$ sous l'action de $f_{\alpha,\beta}^2$
est un cercle.
\end{itemize}
}
\end{thm}

\noindent Supposons que $|\beta|=1;$ dans la carte affine
$(x,y)$ les $f_{\alpha,\beta}$ laissent 
les $3$-vari\'et\'es $|y|=$ cte invariantes. Les orbites pr\'esent\'ees sont born\'ees dans un
$\mathbb{R}^2\times\mathbb{S}^1.$ La dynamique se passe
essentiellement en dimension $3,$ diff\'erentes projections
permettent d'avoir une bonne repr\'esentation de l'orbite d'un
point. Dans la carte affine $z=1$ on d\'esigne par $p_1$ et
$p_2$ les deux projections standards.
Les figures que nous proposons
sont des repr\'esentations (en perspective) des projections
suivantes.
\begin{itemize}
\item On consid\`ere d'abord l'ensemble
\begin{align*}
\Omega_1(m,\alpha,\beta)=\{(p_1(f_{\alpha,\beta}^n(m)),\mathrm{Im} (p_2(f_{
\alpha, \beta}^n(m))))\,\vert\, n=1..30000\};
\end{align*} 
\noindent cet ensemble est contenu dans le produit de $\mathbb{R}^2$ par un
intervalle. L'orbite d'un point sous l'action de
$f_{\alpha,\beta}$ est comprim\'ee par le rev\^etement
double $(x,\rho e^{\mathrm{i}\theta})\to(x,\rho\sin\theta).$

\item De m\^eme on introduit
\begin{align*}
\Omega_2(m,\alpha,\beta)=\{(\mathrm{Re}(p_1(f_{\alpha,
\beta}^n(m))),p_2(f_{\alpha,\beta}^n(m)))\,\vert
\, n=1..30000\}
\end{align*} 
\noindent qui est inclus dans un cylindre $\mathbb{R}\times\mathbb{S}^1;$
cette seconde projection montre comment \og d\'ecomprimer\fg\hspace{1mm}
$\Omega_1$ pour avoir l'allure de l'orbite.
\end{itemize}

\bigskip

\noindent On suppose que $\alpha=\exp(2\mathrm{i}\sqrt{3})$ et $\beta=\exp(2\mathrm{i}\sqrt{2})$ et on note $\Omega_i(m)$ au lieu de $\Omega_i(m,\alpha,\beta).$

\noindent Les figures qui suivent illustrent tr\`es bien le Th\'eor\`eme \ref{linearisation}.

\begin{tabular}{ll}
\hspace{10mm}\begin{picture}(0,0)%
\includegraphics{des1scan.pstex}%
\end{picture}%
\setlength{\unitlength}{3947sp}%
\begingroup\makeatletter\ifx\SetFigFont\undefined%
\gdef\SetFigFont#1#2#3#4#5{%
  \reset@font\fontsize{#1}{#2pt}%
  \fontfamily{#3}\fontseries{#4}\fontshape{#5}%
  \selectfont}%
\fi\endgroup%
\begin{picture}(1200,1200)(1201,-1561)
\end{picture}%
\hspace{18mm} &\hspace{18mm}
\begin{picture}(0,0)%
\includegraphics{des2scan.pstex}%
\end{picture}%
\setlength{\unitlength}{3947sp}%
\begingroup\makeatletter\ifx\SetFigFont\undefined%
\gdef\SetFigFont#1#2#3#4#5{%
  \reset@font\fontsize{#1}{#2pt}%
  \fontfamily{#3}\fontseries{#4}\fontshape{#5}%
  \selectfont}%
\fi\endgroup%
\begin{picture}(1200,1200)(1201,-1561)
\end{picture}%

\hspace{5mm}\\
\hspace{10mm}$\Omega_1(10^{-4}\mathrm{i},10^{-4} \mathrm{i})$\hspace{18mm} &\hspace{18mm}
$\Omega_2(10^{-4}\mathrm{i},10^{-4} \mathrm{i})$
\hspace{5mm}\\
\end{tabular}

\vspace{4mm}

\noindent Il s'agit de l'orbite d'un point situ\'e dans la zone de
lin\'earisation du point $(0:0:1);$ on observe que
l'adh\'erence d'une orbite est un tore.

\vspace{8mm}

\begin{tabular}{ll}
\hspace{10mm}\begin{picture}(0,0)%
\includegraphics{des3scan.pstex}%
\end{picture}%
\setlength{\unitlength}{3947sp}%
\begingroup\makeatletter\ifx\SetFigFont\undefined%
\gdef\SetFigFont#1#2#3#4#5{%
  \reset@font\fontsize{#1}{#2pt}%
  \fontfamily{#3}\fontseries{#4}\fontshape{#5}%
  \selectfont}%
\fi\endgroup%
\begin{picture}(1200,1200)(1201,-1561)
\end{picture}%
\hspace{12mm} & \hspace{16mm}
\begin{picture}(0,0)%
\includegraphics{des4scan.pstex}%
\end{picture}%
\setlength{\unitlength}{3947sp}%
\begingroup\makeatletter\ifx\SetFigFont\undefined%
\gdef\SetFigFont#1#2#3#4#5{%
  \reset@font\fontsize{#1}{#2pt}%
  \fontfamily{#3}\fontseries{#4}\fontshape{#5}%
  \selectfont}%
\fi\endgroup%
\begin{picture}(1200,1200)(1201,-1561)
\end{picture}%

\hspace{5mm}\\
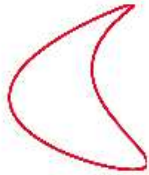
$\Omega_1(10000+10^{-4}\mathrm{i},10000+10^{-4}
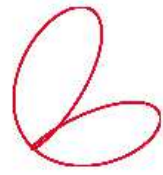
\mathrm{i})$\hspace{8mm} & \hspace{8mm}$ \Omega_2(10000+10^{-4}\mathrm{i},10000+10^{-4}
\mathrm{i})$
\hspace{5mm}
\end{tabular}

\vspace{4mm}

\noindent C'est \og l'orbite\fg\, sous l'action de $f_{\alpha,\beta}^2$ d'un point dans la zone de
lin\'earisation du point $(0:1:0);$ ici l'adh\'erence d'une \og orbite\fg\,
est un cercle topologique. Les singularit\'es
sont des art\'efacts de projection.

\begin{rem}
La droite $z=0$ est contract\'ee par $f_{\alpha,\beta}$ sur 
$(0:1:0)$ lui-m\^eme \'eclat\'e sur~$z=0:$ la 
transforma\-tion~$f_{\alpha,\beta}$ n'est pas alg\'ebriquement stable; c'est
pour cette raison que l'on a consid\'er\'e $f_{\alpha,\beta}^2$
et non~$f_{\alpha,\beta}.$
\end{rem}

\noindent La th\'eorie ne permet pas de pr\'evoir ce qui se
passe \`a l'ext\'erieur de ces deux zones de lin\'earisation. Entre les deux zones de lin\'earisation $\mathcal{V}_1$ et
$\mathcal{V}_2$ les exp\'eriences sugg\`erent une dynamique chaotique
comme le montrent les dessins qui suivent. 

\bigskip

\begin{tabular}{cc}
\hspace{5mm}\begin{picture}(0,0)%
\includegraphics{des9scan.pstex}%
\end{picture}%
\setlength{\unitlength}{3947sp}%
\begingroup\makeatletter\ifx\SetFigFont\undefined%
\gdef\SetFigFont#1#2#3#4#5{%
  \reset@font\fontsize{#1}{#2pt}%
  \fontfamily{#3}\fontseries{#4}\fontshape{#5}%
  \selectfont}%
\fi\endgroup%
\begin{picture}(1200,1200)(1201,-1561)
\end{picture}%
\hspace{9mm} & \hspace{9mm}
\begin{picture}(0,0)%
\includegraphics{des10scan.pstex}%
\end{picture}%
\setlength{\unitlength}{3947sp}%
\begingroup\makeatletter\ifx\SetFigFont\undefined%
\gdef\SetFigFont#1#2#3#4#5{%
  \reset@font\fontsize{#1}{#2pt}%
  \fontfamily{#3}\fontseries{#4}\fontshape{#5}%
  \selectfont}%
\fi\endgroup%
\begin{picture}(1200,1200)(1201,-1561)
\end{picture}%

\hspace{5mm}\\
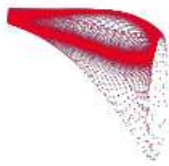
\hspace{8mm}$\Omega_1(0.4+10^{-4}\mathrm{i},0.4+10^{-4} \mathrm{i})$\hspace{9mm} & \hspace{9mm}
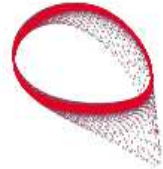
$\Omega_2(0.4+10^{-4}\mathrm{i},0.4+10^{-4} \mathrm{i})$
\hspace{5mm}
\end{tabular}

\vspace{4mm}

\noindent On constate une d\'eformation des tores invariants.

\vspace{10mm}

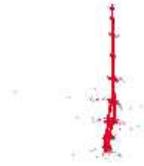
\begin{tabular}{cc}
\hspace{5mm} \begin{picture}(0,0)%
\includegraphics{des5scan.pstex}%
\end{picture}%
\setlength{\unitlength}{3947sp}%
\begingroup\makeatletter\ifx\SetFigFont\undefined%
\gdef\SetFigFont#1#2#3#4#5{%
  \reset@font\fontsize{#1}{#2pt}%
  \fontfamily{#3}\fontseries{#4}\fontshape{#5}%
  \selectfont}%
\fi\endgroup%
\begin{picture}(1200,1200)(1201,-1561)
\end{picture}%
 \hspace{9mm} &
\hspace{9mm}\begin{picture}(0,0)%
\includegraphics{des6scan.pstex}%
\end{picture}%
\setlength{\unitlength}{3947sp}%
\begingroup\makeatletter\ifx\SetFigFont\undefined%
\gdef\SetFigFont#1#2#3#4#5{%
  \reset@font\fontsize{#1}{#2pt}%
  \fontfamily{#3}\fontseries{#4}\fontshape{#5}%
  \selectfont}%
\fi\endgroup%
\begin{picture}(1200,1200)(1201,-1561)
\end{picture}%

\hspace{5mm}\\
& \\
\hspace{5mm} $\Omega_1(0.9+10^{-4}\mathrm{i},0.9 +10^{-4}
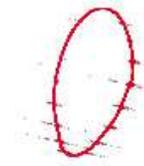
\mathrm{i})$\hspace{9mm} &
\hspace{9mm}$\Omega_2(0.9+10^{-4}\mathrm{i},0.9 +10^{-4}
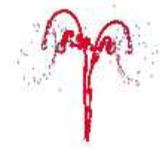
\mathrm{i})$
\hspace{5mm}\\
& \\
\hspace*{5mm} \begin{picture}(0,0)%
\includegraphics{des7scan.pstex}%
\end{picture}%
\setlength{\unitlength}{3947sp}%
\begingroup\makeatletter\ifx\SetFigFont\undefined%
\gdef\SetFigFont#1#2#3#4#5{%
  \reset@font\fontsize{#1}{#2pt}%
  \fontfamily{#3}\fontseries{#4}\fontshape{#5}%
  \selectfont}%
\fi\endgroup%
\begin{picture}(1200,1200)(1201,-1561)
\end{picture}%
 \hspace{9mm} &
\hspace{9mm} \begin{picture}(0,0)%
\includegraphics{des8scan.pstex}%
\end{picture}%
\setlength{\unitlength}{3947sp}%
\begingroup\makeatletter\ifx\SetFigFont\undefined%
\gdef\SetFigFont#1#2#3#4#5{%
  \reset@font\fontsize{#1}{#2pt}%
  \fontfamily{#3}\fontseries{#4}\fontshape{#5}%
  \selectfont}%
\fi\endgroup%
\begin{picture}(1200,1200)(1201,-1561)
\end{picture}%

\hspace{5mm}\\
& \\
\hspace{5mm} $\Omega_1(1+10^{-4}\mathrm{i},1 +10^{-4}\mathrm{i})$\hspace{9mm} &
\hspace{9mm}$\Omega_2(1+10^{-4}\mathrm{i},1 +10^{-4}\mathrm{i})$
\hspace{5mm}\\
& \\
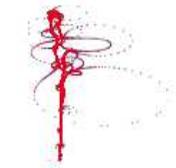
\hspace*{5mm}\begin{picture}(0,0)%
\includegraphics{etrange6scan.pstex}%
\end{picture}%
\setlength{\unitlength}{3947sp}%
\begingroup\makeatletter\ifx\SetFigFont\undefined%
\gdef\SetFigFont#1#2#3#4#5{%
  \reset@font\fontsize{#1}{#2pt}%
  \fontfamily{#3}\fontseries{#4}\fontshape{#5}%
  \selectfont}%
\fi\endgroup%
\begin{picture}(1200,1200)(1201,-1561)
\end{picture}%

\hspace{9mm} &\hspace{9mm} \begin{picture}(0,0)%
\includegraphics{f15scan.pstex}%
\end{picture}%
\setlength{\unitlength}{3947sp}%
\begingroup\makeatletter\ifx\SetFigFont\undefined%
\gdef\SetFigFont#1#2#3#4#5{%
  \reset@font\fontsize{#1}{#2pt}%
  \fontfamily{#3}\fontseries{#4}\fontshape{#5}%
  \selectfont}%
\fi\endgroup%
\begin{picture}(1200,1200)(1201,-1561)
\end{picture}%

\hspace{5mm}\\
& \\
\hspace{5mm} $\Omega_1(1.08+10^{-4}\mathrm{i},1.08+10^{-4}\mathrm{i})$\hspace{9mm} &
\hspace{9mm}$\Omega_2(1.08+10^{-4}\mathrm{i},1.08+10^{-4}\mathrm{i})$
\hspace{5mm}\\
\end{tabular}

\bigskip

\noindent Manifestement les tores invariants ont disparu. Toutefois la figure 
semble s'organiser autour d'une courbe ferm\'ee.

\medskip

\noindent Si, contrairement au cadre des automorphismes polynomiaux
de $\mathbb{C}^2,$ on n'a pas d'\'equivalence entre premier degr\'e
dynamique strictement sup\'erieur \`a $1$ et centralisateur non 
d\'enombrable on a n\'eanmoins une implication. Plus pr\'ecis\'ement
on a l'\'enonc\'e suivant.

\begin{thm}[\cite{Can}]
{\sl Soit $f$ une transformation birationnelle du plan projectif
complexe dont le premier degr\'e dynamique $\lambda(f)$ est
strictement plus grand que $1.$ Si $g$ est un \'el\'ement
de $\mathrm{Bir}(\mathbb{P}^2)$ qui commute avec $f,$ il existe
deux entiers $m$ dans $\mathbb{N}^*$ et $n$ dans $\mathbb{Z}$
tels que~$g^m=f^n.$}
\end{thm}

\noindent \'Etant donn\'ee une transformation birationnelle $f,$ \textsc{Diller} et \textsc{Favre} ont \'etudi\'e le comportement de $(\deg f^n)_n$ et donn\'e des propri\'et\'es dans chacun des cas possibles. 

\begin{thm}[\cite{DiFa}]\label{DiFa}
{\sl Une transformation birationnelle de $\mathbb{P}^2(\mathbb{C})$
satisfait, \`a conjugaison birationnelle pr\`es, une et une seule
des propri\'et\'es suivantes.
\smallskip
\begin{itemize}
\item La suite $(\deg f^n)_n$ est born\'ee, $f$ est un automorphisme
d'une surface rationnelle $\mathcal{Z}$ et un it\'er\'e de $f$ 
appartient \`a $\mathrm{Aut}^0(\mathcal{Z}),$ la composante neutre de
$\mathrm{Aut}(\mathcal{Z}).$\smallskip

\item $\deg f^n\sim n$ alors $f$ n'est pas un automorphisme et 
$f$ pr\'eserve une unique fibration qui est rationnelle. \smallskip

\item $\deg f^n\sim n^2$ auquel cas $f$ est un automorphisme 
pr\'eservant une unique fibration qui est elliptique.\smallskip

\item $\deg f^n\sim\alpha^n,$ $\alpha>1.$\smallskip
\end{itemize}

\noindent Dans les trois premi\`eres \'eventualit\'e on a $\lambda(f)=1,$ dans le 
derni\`ere $\alpha=\lambda(f)>1.$
}
\end{thm}

\begin{egs}
\begin{itemize}
\item Si $f$ est un automorphisme de $\mathbb{P}^2(\mathbb{C}),$ alors $\deg f^n=1$ pour tout $n.$ Toute transformation birationnelle $f$ d'ordre fini satisfait: $(\deg f^n)_n$ est born\'ee.

\smallskip 

\item La transformation $f=(xz:xy:z^2)$ v\'erifie $\deg f^n
\sim n$ (se placer dans la carte $z=1$).

\smallskip

\item Le prolongement $f$ d'un automorphisme de \textsc{H\'enon} g\'en\'eralis\'e
$(y,P(y)-\delta x),$ avec $P$ dans $\mathbb{C}[y]$ tel que $\deg P
\geq 2$ et $\delta$ dans $\mathbb{C}^*,$ a une croissance des 
degr\'es exponentielle, {\it i.e.} $\deg f^n=(\deg P)^n.$

\smallskip

%

\item Soit $f_M$ l'\'el\'ement de $\mathrm{Bir}(\mathbb{P}^2)$
d\'efini dans la carte $z=1$ par 
\begin{align*}
& f_M=(x^ay^b,x^cy^d)&&\text{avec} && M=\left[\begin{array}{cc}
a & b\\
c & d
\end{array}
\right]\in\mathrm{SL}_2(\mathbb{Z}).
\end{align*}

\noindent On a l'alternative\smallskip
\begin{itemize}
\item ou bien $\vert\text{tr}\,M\vert\leq 2$ et $\lambda(f_M)=1;$\smallskip

\item ou bien $M$ a pour valeurs propres r\'eelles $\lambda$ et $\lambda^{-1}$ avec 
$\lambda^{-1}<1<\lambda$ et $\lambda(f_M)=\lambda.$\smallskip
\end{itemize}

\noindent Par suite $(\deg f_M^n)_n$ est born\'ee (resp. $\deg f_M^n\sim n,$
resp. $\deg f_M^n\sim n^2,$ resp. $\deg f_M^n\sim~\lambda(f_M)^n$) si et 
seulement si $M$ est elliptique (resp. parabolique, resp. parabolique, 
resp. hyperbolique).
\end{itemize}
\end{egs}

\noindent Une {\it vari\'et\'e k\"{a}hl\'erienne}\label{ind22} $M$ est une vari\'et\'e hermitienne $M$ ({\it i.e.} une vari\'et\'e complexe munie d'une m\'etrique hermitienne $h$) telle que la 2-forme $\omega = - \mathrm{im}\,h$ soit ferm\'ee.

\begin{egs}
\begin{itemize}
\item Puisque toute sous-vari\'et\'e d'une vari\'et\'e k\"{a}hl\'erienne est k\"{a}hl\'erienne et que la m\'etrique de Fubini-Study sur $\mathbb{P}^n(\mathbb{C})$ est k\"{a}hl\'erienne toute vari\'et\'e alg\'ebrique projective est k\"{a}hl\'erienne.

\item Soit $\Lambda$ un r\'eseau dans $\mathbb{C}^n;$ le tore $\mathbb{C}^n/\Lambda$ est une vari\'et\'e k\"{a}hl\'erienne.

\item Toute surface de Riemann est k\"{a}hl\'erienne.
\end{itemize}
\end{egs}

\noindent Les notions de premier degr\'e dynamique, de stabilit\'e alg\'ebrique
et le Th\'eor\`eme \ref{DiFa} se g\'en\'eralisent au contexte suivant: les transformations
bim\'eromorphes sur des surfaces compactes k\"{a}hl\'eriennes. Une transformation rationnelle $f$ sur une surface complexe compacte k\"{a}hl\'erienne $\mathcal{Z}$ d\'efinit un op\'erateur lin\'eaire $f^*$ sur les groupes de cohomologie de $\mathcal{Z}$ pr\'eservant la d\'ecomposition de \textsc{Hodge} $\mathrm{H}^k(\mathcal{Z},\mathbb{C})=\displaystyle \bigoplus_{i+j=k}\mathrm{H}^{i,j}(\mathcal{Z}, \mathbb{C}).$ On d\'efinit le {\it premier degr\'e dynamique}\label{ind23} de $f$ par
\begin{align*}
\lambda(f)=\lim_{n\to +\infty}\vert\vert (f^n)^*\vert\vert^{1/n}
\end{align*}

\noindent o\`u $\vert\vert\cdot\vert\vert$ d\'esigne une norme sur $\mathrm{End}(\mathrm{H}^{1,1} (\mathcal{Z},\mathbb{R}));$ toujours d'apr\`es \cite{DiFa} cette quantit\'e est un invariant birationnel. Elle co\"{i}ncide sur le plan projectif complexe avec la notion d\'ej\`a introduite.

\chapter{Un peu de g\'eom\'etrie alg\'ebrique}\label{ga}

\section{Quelques d\'efinitions et propri\'et\'es}

\noindent Soient $\mathcal{X}$ et $\mathcal{Z}$ deux surfaces complexes compactes. Une application m\'eromorphe $f\colon\mathcal{X}\dashrightarrow\mathcal{Z}$ est d\'efinie par son graphe $\Gamma(f)\subset\mathcal{X}\times\mathcal{Z};$ ce graphe est une sous-vari\'et\'e irr\'eductible pour laquelle la projection $\pi_1\colon\Gamma(f)\to\mathcal{X}$ sur le premier facteur est une {\it modification propre}\label{ind26z}, {\it i.e.} une application holomorphe surjective propre dont la fibre g\'en\'erique est un point (\cite{Fi}). Le {\it lieu d'ind\'etermination}\label{ind26w} de $f$ est l'ensemble fini de points o\`u $\pi_1$ n'admet pas d'inverse local. L'application~$f$ est {\it dominante}\label{ind26y} si la seconde projection $\pi_2\colon\Gamma(f)\to\mathcal{Z}$ est surjective. L'{\it ensemble critique}\label{ind26x} de~$f$ est l'image par $\pi_1$ du lieu critique de $\pi_2;$ on le notera $\mathcal{C}(f).$ Notons~$\mathrm{Exc}\,\pi_2$ l'ensemble des points o\`u $\pi_2$ n'est pas une application finie; on d\'efinit l'{\it ensemble exceptionnel}\label{ind26t} de~$f$ par~$\mathrm{Exc}\,f=\pi_1(\mathrm{Exc}\,\pi_2)\subset\mathcal{C}(f).$ Soient $f\colon\mathcal{X}\dashrightarrow \mathcal{Y},$ $g\colon\mathcal{Y}\dashrightarrow\mathcal{Z}$ deux applications m\'eromorphes dominantes; le graphe $\Gamma(g\circ f)$ de $g\circ f$ est l'adh\'erence de $$\{(x, g(f(x))) \in\mathcal{X}\times\mathcal{Z}\,\vert\, x\not \in\mathrm{Ind}\,f,\,f(x)\not\in\mathrm{Ind}\, g\};$$ ce sous-ensemble co\"{i}ncide avec $$\Gamma(g)\circ\Gamma(f)=\{(x,z)\in\mathcal{X}\times \mathcal{Z}\,\vert\,\exists\, y\in\mathcal{Y}, \,(x,y)\in\Gamma(f),\,(y,z)\in\Gamma(g)\}$$ si et seulement si ce dernier est irr\'eductible.

\begin{pro}[\cite{DiFa}]
{\sl Soient $f\colon\mathcal{X}\dashrightarrow \mathcal{Y},$ $g\colon\mathcal{Y}\dashrightarrow\mathcal{Z}$ deux applications m\'eromor\-phes dominantes entre deux surfaces complexes compactes; $\Gamma(g)\circ\Gamma(f)$ est irr\'eductible si et seulement s'il n'existe pas de composante $\mathcal{V}$ dans $\mathrm{Exc}\, f$ telle que $f(\mathcal{V})$ appartienne \`a $\mathrm{Ind}\, g.$}
\end{pro} 

\noindent La transformation $f$ est {\it alg\'ebriquement stable}\label{ind26r} si et seulement si $\Gamma(f) \circ\Gamma(f^n)=\Gamma(f\circ f^n)$ pour tout $n$ dans $\mathbb{N}.$

\bigskip

\noindent Soit $f\colon\mathcal{X}\dashrightarrow\mathcal{Z}$ une application m\'eromorphe dominante entre deux surfaces complexes compactes, $\Gamma$ une d\'esingularisation de son graphe et $\pi_1,$ $\pi_2$ les deux projections naturelles. Une forme lisse $\alpha\in \mathcal{C}^\infty_{p,q}(\mathcal{Z})$ de bidegr\'e $(p,q)$ peut \^etre tir\'ee en arri\`ere comme une forme lisse~$\pi_2^*\alpha\in\mathcal{C}^\infty_{p,q}(\Gamma)$ et pouss\'ee en avant comme un courant. Si $\{\alpha\}\in\mathrm{H}^{p,q}(\mathcal{Z})$ est la classe de Dolbeault d'une forme lisse $\alpha$ on d\'efinit $f^*\{\alpha\}$ par $f^*\{\alpha\}:=\{\pi_{1*}\pi_2^*\alpha\}\in \mathrm{H}^{p,q}(\mathcal{X})$ d'o\`u l'application lin\'eaire $f^*\colon\mathrm{H}^{p,q}(\mathcal{Z})\to\mathrm{H}^{p,q}(\mathcal{X}).$ De la m\^eme fa\c{c}on on peut d\'efinir $$f_*=\pi_{2*}\pi_1^*\colon\mathrm{H}^{p,q}(\mathcal{X})\to\mathrm{H}^{p,q}(\mathcal{Z}).$$

\bigskip

\noindent Soient $\mathcal{Z}$ une surface et $\pi\colon\widetilde{\mathcal{Z}}\to\mathcal{Z}$ l'application d'\'eclatement d'un point $p$ de $\mathcal{Z}.$ Si $\mathcal{C}$ est une courbe
de $\mathcal{Z}$ passant par $p$ on d\'esigne par $\widetilde{\mathcal{C}}=\overline{\pi^{-1} (\mathcal{C}\setminus\{p\})}$ la {\it transform\'ee stricte}\label{ind27} de~$\mathcal{C};$ la {\it transform\'ee totale}\label{ind27b} de $\mathcal{C}$ est le diviseur $\pi^*\mathcal{C}.$ Remarquons que si $\mathcal{C}$ est lisse en $p$ on a~$$\pi^*\mathcal{C}=\widetilde{\mathcal{C}}+\mathrm{E}.$$

\noindent Un {\it diviseur}\label{ind28} $E$ sur $\mathcal{Z}$ est une 
combinaison lin\'eaire du type $\sum c_iC_i,$ les $C_i$ 
d\'esignant des courbes irr\'eductibles \'eventuellement singuli\`eres de $\mathcal{Z}$ et les $c_i$ des entiers relatifs. Les diviseurs $D'$ et~$D''$ sont {\it lin\'eairement 
\'equivalents}\label{ind29}, $D'\sim D'',$ si $D'-D''$ est le diviseur d'une fonction 
rationnelle~$R,$ autrement dit si $D'-D''$ co\"incide avec le
lieu des z\'eros (avec multiplicit\'e) de $R$ priv\'e de 
l'ensemble des p\^oles de $R.$

\noindent Le {\it groupe de \textsc{Picard}}\label{ind30} de $\mathcal{Z},$ not\'e $\mathrm{Pic}
(\mathcal{Z}),$ est le groupe quotient des diviseurs modulo \'equivalence lin\'eaire.
On a: $\mathrm{Pic}(\mathbb{P}^2(\mathbb{C}))\simeq
\mathbb{Z}.$ Soient $\mathrm{H}$ la classe d'une droite dans 
$\mathrm{Pic}(\mathbb{P}^2(\mathbb{C}))$
et~$\mathcal{C}$ la classe d'une courbe de degr\'e $d;$ on 
constate que $\mathcal{C}\sim d\mathrm{H}.$

\noindent La surface $\mathcal{Z}$ est munie d'une forme d'intersection: si $D$ et $D'$ sont deux courbes distinctes,~$D\cdot D'$ correspond au nombre de points d'intersection de ces courbes compt\'es avec multiplicit\'e; on remarque que dans ce cas $D\cdot D'$ est positif ou nul. On peut \'etendre naturellement cette d\'efinition pour donner un sens \`a l'intersection de deux diviseurs quelconques; un cas particulier est l'auto-intersection, on d\'esigne alors $D\cdot D$ par $D^2.$ Le nombre d'intersection v\'erifie les propri\'et\'es suivantes
\begin{itemize}
\item si $D'\sim D''$ alors $D\cdot D'=D\cdot D'';$

\item $\pi^*D\cdot\pi^*D'=D\cdot D';$

\item $\mathrm{E}\cdot\pi^*D=0;$

\item $\mathrm{E}^2=-1;$

\item $\widetilde{\mathcal{C}}^2=\mathcal{C}^2-1.$
\end{itemize}


\bigskip

\noindent Soient $\pi\colon\mathrm{Bl}_p\mathbb{P}^2
\to\mathbb{P}^2(\mathbb{C})$ l'\'eclatement
de $\mathbb{P}^2(\mathbb{C})$ au point $p$ et $\mathrm{E}=\pi^{-1}(p)$ la 
fibre exceptionnelle. Notons $\mathrm{Rat}(\mathrm{S})$ 
l'ensemble des fonctions rationnelles sur $\mathrm{S};$
on a: $$\mathrm{Rat}(\mathrm{Bl}_p\mathbb{P}^2)=\pi^*
(\mathrm{Rat}(\mathbb{P}^2)).$$ Le groupe $\mathrm{Pic}
(\mathrm{Bl}_p\mathbb{P}^2)$ est engendr\'e par $\mathrm{E}$
et $\widetilde{\mathrm{L}}$ o\`u $\mathrm{L}=\{\ell=0\}$ 
d\'esigne une droite ne passant pas par $p$ et $\widetilde{\mathrm{L}}=\pi^*
\mathrm{L}=\{\ell\circ\pi=0\}$ le pull-back de $\mathrm{L}.$

\noindent Plus g\'en\'eralement soient $\mathcal{Z}$ une vari\'et\'e 
complexe obtenue en \'eclatant le plan projectif complexe 
en $N$ points distincts $p_1,$ $\ldots,$ $p_N$ et $\pi$ la composition de ces diff\'erents \'eclatements. Posons $\mathrm{H}_\mathcal{Z}=
\pi^*\mathrm{H},$ $\mathrm{H}$ d\'esignant une droite. Si $\mathrm{H}$
ne contient aucun des $p_i,$ alors $\mathrm{H}_\mathcal{Z}$ est 
repr\'esent\'e par la transform\'ee stricte 
$\widetilde{\mathrm{H}}$ de $\mathrm{H}.$ Notons 
$\mathrm{E}_j=\pi^{-1}(p_j)$ la classe dans $\mathrm{Pic}
(\mathcal{Z})$ de la fibre exceptionnelle obtenue en \'eclatant 
$p_i.$ En g\'en\'eral si $\mathrm{H}$ est une droite de 
$\mathbb{P}^2(\mathbb{C}),$ alors $\mathrm{H}_\mathcal{Z}=
\widetilde{\mathrm{H}}+\displaystyle\sum_{j\vert p_j\in\mathrm{H}}
\mathrm{E}_j.$ Dans ce cas on peut d\'ecrire $\mathrm{Pic}(\mathcal{Z}).$

\begin{thm}
{\sl Soient $p_1,$ $\ldots,$ $p_N$ des points distincts
de $\mathbb{P}^2(\mathbb{C}).$ Notons $\pi\colon
\mathcal{Z}\to\mathbb{P}^2(\mathbb{C})$ la suite d'\'eclatements
des $p_i.$ Si $\mathrm{E}_j=\pi^{-1}(p_j)$ d\'esignent 
les fibres exceptionnelles et $\mathrm{H}$ une droite 
g\'en\'erique de $\mathbb{P}^2(\mathbb{C}),$ alors
$$\mathrm{Pic}(\mathcal{Z})=\mathbb{Z}\mathrm{E}_1\oplus\ldots
\oplus\mathbb{Z}\mathrm{E}_N\oplus\mathbb{Z}\pi^* 
\mathrm{H}.$$}
\end{thm}

\bigskip

\noindent Rappelons ce que sont les {\it surfaces de \textsc{Hirzebruch}}\label{ind31}
$\mathrm{F}_n.$ Posons $\mathrm{F}_0=\mathbb{P}^1(\mathbb{C})
\times\mathbb{P}^1(\mathbb{C}).$ La surface $\mathrm{F}_1$
est obtenue en \'eclatant le plan projectif complexe en 
$(1:0:0);$ cette surface est un compactifi\'e de $\mathbb{C}^2$ naturellement muni d'une fibration rationnelle correspondant aux droi\-tes~$y=$ cte. Le diviseur \`a l'infini
est constitu\'e de deux courbes rationnelles s'intersectant 
transversalement en un point. On a \smallskip
\begin{itemize}
\item la transform\'ee stricte de la droite \`a l'infini 
dans $\mathbb{P}^2(\mathbb{C})$ qui est une droite;\smallskip

\item le diviseur exceptionnel de l'\'eclatement qui est 
une section pour la fibration.
\end{itemize}

\noindent Plus g\'en\'eralement $\mathrm{F}_n$ est, pour 
tout $n\geq 1,$ un compactifi\'e de $\mathbb{C}^2$ 
muni d'une fibration rationnelle telle que le diviseur
\`a l'infini soit constitu\'e de deux courbes rationnelles
transverses, une fibre $f$ et une section $s_n$ d'auto-intersection
$-n$ 
\begin{align*}
&\mathrm{F}_n=\mathbb{P}_{\mathbb{P}^1(\mathbb{C})}
(\mathcal{O}_{\mathbb{P}^1(\mathbb{C})}(1)\oplus
\mathcal{O}_{\mathbb{P}^1(\mathbb{C})}(n)),&& n\geq 2.
\end{align*}

\noindent Consid\'erons la surface $\mathrm{F}_n.$ On note $p$
l'intersection de $s_n$ et $f$ une fibre, $\pi_1$ l'\'eclatement de 
$\mathrm{F}_n$ en $p$ et $\pi_2$ la contraction de la 
transform\'ee stricte $\widetilde{f}$ de $f.$ On passe de 
$\mathrm{F}_n$ \`a $\mathrm{F}_{n+1}$ via $\pi_2\pi_1^{-1}$

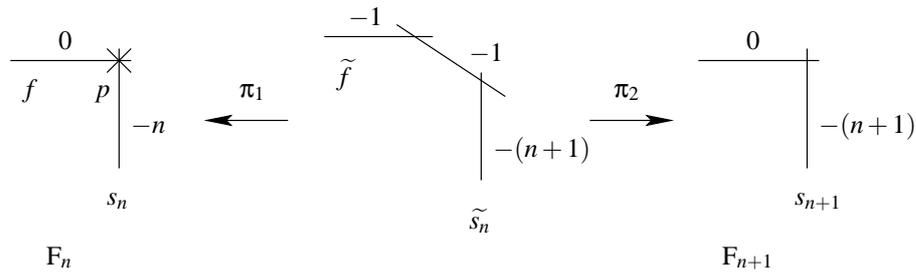
\begin{figure}[H]
\begin{center}
\input{el1.pstex_t}
\end{center}
\end{figure}

\noindent On peut aussi passer de $\mathrm{F}_{n+1}$ \`a 
$\mathrm{F}_n$ via $\pi_2\pi_1^{-1}$ o\`u 
$\pi_1$ est l'\'eclatement de $\mathrm{F}_{n+1}$ en un 
point $p$ de la fibre $f$ qui n'appartient pas \`a $s_{n+1}$
et $\pi_2$ la contraction de la transform\'ee stricte
$\widetilde{f}$ de $f$

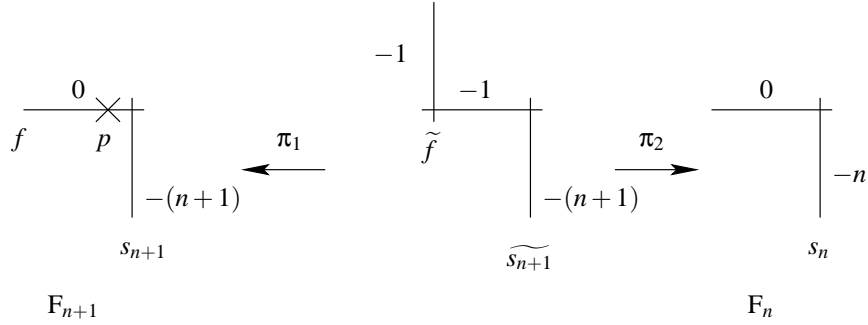
\begin{figure}[H]
\begin{center}
\input{el2.pstex_t}
\end{center}
\end{figure}

\noindent Dans ces deux \'eventualit\'es on dit que 
$\pi_2\pi_1^{-1}$ est une {\it transformation \'el\'ementaire}\label{ind32}
en $p.$

\section{Th\'eor\`eme de factorisation de \textsc{Zariski}}

\noindent Le th\'eor\`eme de factorisation de \textsc{Zariski}
assure que toute transformation de \textsc{Cremona} s'\'ecrit au moyen 
d'\'eclatements; ce r\'esultat est en fait valable pour toute
transformation birationnelle d'une surface projective lisse dans une 
autre. Avant de l'\'enoncer rappelons que si~$\mathcal{X}$ est une vari\'et\'e 
irr\'eductible et $\mathcal{Y}$ une vari\'et\'e, une 
\textit{transformation rationnelle}
$f\colon \mathcal{X}\dashrightarrow~\mathcal{Y}$ est un 
morphisme d'un ouvert $\mathcal{U}$ de $\mathcal{X}$ dans $\mathcal{Y}$ qui n'est pas la
restriction d'un mor\-phisme~$\widetilde{\mathcal{U}}\to~\mathcal{Y}$ avec 
$\mathcal{U}\subsetneq\widetilde{\mathcal{U}}.$

\begin{thm}[\textsc{Zariski}, 1944]\label{Zariski}
{\sl Soient $\mathcal{X},$ $\mathcal{Y}$ deux surfaces projectives lisses et $f$ une 
transformation birationnelle entre $\mathcal{X}$ et $\mathcal{Y}.$ Il existe une 
surface projective lisse $\mathcal{Z}$ et deux suites d'\'eclatements
\begin{align*}
& \pi_1\colon \mathcal{Z}\to \mathcal{X}, &&\pi_2\colon \mathcal{Z}\to \mathcal{Y}
\end{align*}
\noindent telles que 
$f=\pi_2\pi^{-1}_1$
$$\xymatrix{& \mathcal{Z}\ar[dl]_{\pi_1}\ar[dr]^{\pi_2} &\\
\mathcal{X}\ar@{-->}[rr]_f & & \mathcal{Y}}$$}
\end{thm}

\begin{eg}
\noindent L'involution $\sigma$ se d\'ecompose en deux suites d'\'eclatements 

\begin{figure}[H]
\begin{center}
\input{decomp.pstex_t}
\end{center}
\end{figure}

\noindent avec
\begin{align*}
 & A=(1:0:0), && B=(0:1:0), && C=(0:0:1),
\end{align*}
\noindent $E_A$ (resp. $E_B,$ resp. $E_C$) le diviseur exceptionnel obtenu en \'eclatant
$A$ (resp. $B,$ resp. $C$) et~$\widetilde{L}_{AB}$ (resp. $\widetilde{L}_{AC},$ 
resp. $\widetilde{L}_{BC}$) la transform\'ee stricte de $L_{AB}$ (resp. 
$L_{AC},$ resp. $L_{BC}$).
\end{eg}

\noindent On trouve une d\'emonstration du Th\'eor\`eme \ref{Zariski} dans \cite{Be2};
elle se d\'ecompose en deux \'etapes:
\begin{itemize}
\item tout d'abord on montre qu'une transformation rationnelle d'une surface
dans $\mathbb{P}^n(\mathbb{C})$ s'\'ecrit~$\phi\pi^{-1}$ o\`u $\pi$ d\'esigne une suite
d'\'eclatements et $\phi$ un morphisme;

\item puis on \'etablit qu'un morphisme entre deux surfaces est la
compos\'ee d'un isomorphisme et d'une suite d'\'eclatements.
\end{itemize}

\noindent Avant tout rappelons que le diviseur exceptionnel $\mathrm{E}$ obtenu en \'eclatant un 
point $m$ d'une surface $\mathcal{Z}$ est appel\'e
\textit{premier voisinage infinit\'esimal} de
$m$ et les points de $\mathrm{E}$ sont dits \textit{infiniment
proches} de $m.$ Le \textit{$k$-i\`eme
voisinage infinit\'esimal} de $m$ est l'ensemble des
points contenus dans le premier voisinage d'un certain point du
$(k-1)$-i\`eme voisinage infinit\'esimal de $m.$ 

\noindent Si $f\colon\mathbb{P}^2(\mathbb{C})\dashrightarrow\mathbb{P}^2(\mathbb{C}),$
$(x:y:z)\mapsto(f_0(x,y,z):f_1(x,y,z):f_2(x,y,z))$
d\'esigne une transformation birationnelle, le \textit{r\'eseau
homalo\"idal} associ\'e \`a $f$ est le syst\`eme de cour\-bes~$\mathscr{H}_f$ 
d\'efini par
\begin{align*}
&\alpha_0f_0+\alpha_1f_1+\alpha_2f_2=0,&& (\alpha_0:\alpha_1:
\alpha_2)\in\mathbb{P}^2(\mathbb{C});
\end{align*}
c'est l'image r\'eciproque par $f$ du r\'eseau de droites
$\alpha_0x+\alpha_1y+\alpha_2z=0.$ Chaque
courbe du r\'eseau $\mathscr{H}_f$ est donc rationnelle. Les points
base de $\mathscr{H}_f$ sont les points par lesquels
passent toutes les courbes du r\'eseau; on les appelle aussi
\textit{points base} de $f.$ Ils peuvent \^etre dans~$\mathbb{P}^2(\mathbb{C})$ ou infiniment proches de $\mathbb{P}^2(\mathbb{C});$
d\`es que l'un de ces points n'est pas dans $\mathbb{P}^2(\mathbb{C}),$
{\it i.e.} n'est pas {\it propre}\label{ind36}, il appartient \`a un
$k$-i\`eme voisinage infinit\'esimal d'un point base. 

\section{Matrices caract\'eristiques}\label{matcar}

\noindent Soit $f$ une transformation birationnelle de $\mathbb{P}^2(\mathbb{C})$ dans lui-m\^eme de degr\'e $\nu.$ D'apr\`es le Th\'eor\`eme \ref{Zariski} il existe~$\pi$ et $\eta$ deux suites d'\'eclatements telles que 
$$\xymatrix{& \mathcal{Z}\ar[dl]_{\pi}\ar[dr]^{\eta} &\\
\mathbb{P}^2(\mathbb{C})\ar@{-->}[rr]_f & & \mathbb{P}^2(\mathbb{C}) }$$
On peut \'ecrire $\pi$ sous la forme $$\pi\colon\mathcal{Z}=\mathcal{Z}_k\stackrel{\pi_k}{\to} \mathcal{Z}_{k-1}\stackrel{\pi_{k-1}}{\to}\ldots\stackrel{\pi_2}{\to}\mathcal{Z}_1\stackrel{\pi_1}{\to} \mathcal{Z}_0=\mathbb{P}^2(\mathbb{C})$$ o\`u $\pi_i$ est l'\'eclatement d'un point $p_{i-1}$ de $\mathcal{Z}_{i-1};$ de m\^eme, pour $j=1,$ $\ldots,$ $k,$ il existe $\eta_j\colon\mathcal{X}_j \to~\mathcal{X}_{j-1}$ \'eclatement du point $p'_{j-1}$ de $\mathcal{X}_{j-1}$ tels que $$\eta\colon\mathcal{Z} =\mathcal{X}_k\stackrel{\eta_k}{\to}\mathcal{X}_{k-1} \stackrel{\eta_{k-1}}{\to}\ldots\stackrel{\eta_2}{\to}\mathcal{X}_1\stackrel{\eta_1}{\to} \mathcal{X}_0=\mathbb{P}^2(\mathbb{C}).$$ Notons $\mathrm{E}_1,$ $\ldots,$ $\mathrm{E}_k$ (resp. $\mathrm{E}'_1,$ $\ldots,$ $\mathrm{E}'_k$) les diviseurs exceptionnels obtenus en \'ecla\-tant~$p_1,\,\ldots,\,p_k$ (resp. $p'_1,$ $\ldots,$ $p'_k$). Une {\it r\'esolution ordonn\'ee}\label{ind37} de $f$ est une d\'ecomposition $f=\eta\pi^{-1}$ o\`u $\eta$ et~$\pi$ sont des suites d'\'eclatements ordonn\'ees. La donn\'ee d'une r\'esolution ordonn\'ee de~$f$ d\'efinit deux bases de $\mathrm{Pic}(\mathcal{Z})$
\begin{itemize}
\item $\mathcal{B}=\{e_0=\pi^*\mathrm{H},\, e_1=\mathrm{E}_1,\,\ldots,\, e_k=\mathrm{E}_k\};$

\item $\mathcal{B}'=\{e'_0=\eta^*\mathrm{H},\, e'_1=\mathrm{E}'_1,\,\ldots,\, e'_k=\mathrm{E}'_k\}.$
\end{itemize}

\noindent On peut \'ecrire les $e'_i$ de la fa\c{c}on suivante
\begin{align*}
&e'_0=\nu e_0-\sum_{i=1}^k m_ie_i, && e'_j=\nu_je_0-\sum_{i=1}^k m_{ij}e_i,\, j\geq 1.
\end{align*}

\noindent La matrice de changement de bases $$M=\left[\begin{array}{cccc} \nu&\nu_1&\ldots&\nu_k\\ -m_1 & -m_{11}&\ldots&-m_{1k}\\
\vdots&\vdots& &\vdots\\ -m_k&-m_{k1}&\ldots&-m_{kk}\end{array}\right]$$ est appel\'ee {\it matrice caract\'eristique}\label{ind38} de $f.$ La premi\`ere colonne de la matrice $M$ est le vecteur $(\nu,-m_1,\ldots,-m_k),$ vecteur caract\'eristique de $f.$ Les autres colonnes $(\nu_i,-m_{1i},\ldots,-m_{ki})$ d\'ecrivent le \og comportement de $\mathrm{E}'_i$\fg: si $\nu_j>0,$ alors $\pi(\mathrm{E}'_j)$ est une courbe dans $\mathbb{P}^2(\mathbb{C})$ de degr\'e~$\nu_j$ passant par les points base $x_\ell$ de $f$ avec multiplicit\'e $m_{\ell j}.$  

\noindent Notons que si $\mathcal{Z}$ domine $\mathbb{P}^2(\mathbb{C})$ on peut identifier $\mathrm{H}^2(\mathcal{Z},\mathbb{Z})$ et $\mathrm{Pic}(\mathcal{Z})$ ce que nous ferons par la suite.

\section{Dimension de \textsc{Kodaira}}

\noindent Soit $\mathcal{M}$ une vari\'et\'e complexe compacte. Si $L$ est un fibr\'e en droites holomorphe sur $\mathcal{M},$ on d\'esigne par $\mathrm{H}^0(\mathcal{M},L)$ le $\mathbb{C}$-espace vectoriel  constitu\'e des sections holomorphes globales de $L.$ La dimension de $\mathrm{H}^0( \mathcal{M},L)$ est finie. Soient $x$ un point de $\mathcal{M}$ et $L_x$ la fibre de~$L$ en ce point; l'\'evaluation des sections de $L$ au point $x$ d\'etermine une application lin\'eai\-re~$\Theta_{L_x}\colon\mathrm{H}^0(\mathcal{M},L)\to L_x.$ Lorsque toutes les sections globales de $L$ s'annulent en $x,$ l'application $\Theta_{L_x}$ est identiquement nulle et $x$ est un point base de $L.$ \'Etant donn\'e un isomorphisme de $L_x$ avec $\mathbb{C},$ l'application $\Theta_{L_x}$ s'interpr\`ete comme une forme lin\'eaire; elle ne d\'epend du choix de l'isomorphisme que par un facteur multiplicatif. Ainsi pour tout point $x$ de $\mathcal{M}$ qui n'est pas un point base de $L$ on obtient un \'el\'ement $[\Theta_{L_x}]$  de $\mathbb{P}( \mathrm{H}^0(\mathcal{M},L)^*).$ Pour les fibr\'es en droites poss\'edant au moins une section non nulle on a une application m\'eromorphe $$\Theta_L\colon\mathcal{M}\dashrightarrow\mathbb{P} (\mathrm{H}^0(\mathcal{M},L)^*)$$ telle que $\mathrm{Ind}\,\Theta_L$ soit contenu dans l'ensemble des points base de $L.$ On peut r\'ep\'eter cette construction en rempla\c{c}ant $L$ par $L^{\otimes \ell}$ o\`u $\ell$ d\'esigne un entier. On d\'efinit la {\it dimension de \textsc{Kodaira}-\textsc{Iitaka}}\label{ind888} de $L$ par $$\mathrm{kod}\,(\mathcal{M},L)=\max_{\ell>0}\dim_\mathbb{C}( \Theta_{L^{\otimes\ell}}(\mathcal{M}))$$ avec la convention suivante: si aucune puissance strictement positive de $L$ ne poss\`ede de section non nulle, on pose $\mathrm{kod}\,( \mathcal{M},L)=-\infty.$ La {\it dimension de \textsc{Kodaira}}\label{ind88} de $\mathcal{M},$ not\'ee $\mathrm{kod}\,\mathcal{M},$ est la dimension de \textsc{Kodaira}-\textsc{Iitaka} du fibr\'e canonique $\mathrm{K}_\mathcal{M}=\det(\mathrm{T}^*\mathcal{M})$ de $\mathcal{M}.$

\medskip

\noindent Soit $\mathcal{Z}$ une surface complexe compacte k\"{a}hl\'erienne; $\mathrm{kod}\, \mathcal{Z}$ appartient \`a $\{-\infty,\,0,\,1,\,2\}.$ Lorsque la dimension de \textsc{Kodaira} de $\mathcal{Z}$ est positive ou nulle, le mod\`ele minimal $\mathcal{Z}'$ de $\mathcal{Z},$ {\it i.e.} la surface obtenue en contractant les \'eventuelles courbes d'auto-intersection $-1$ sur $\mathcal{Z},$ est unique; de plus, les groupes $\mathrm{Bir}(\mathcal{Z}'),$ $\mathrm{Bir}(\mathcal{Z})$ et $\mathrm{Aut}(\mathcal{Z})$ co\"{i}ncident (\cite{BHPV}). Avant de caract\'eriser de telles surfaces $\mathcal{Z}$ rappelons quelques d\'efinitions.

\medskip

\noindent Une {\it surface K$3$}\label{ind33} est une surface $\mathcal{Z}$ 
complexe, compacte, simplement connexe, \`a fibr\'e canonique
trivial. En particulier il existe une $2$-forme holomorphe
$\omega$ sur $\mathcal{Z}$ qui ne s'annule pas; $\omega$ est 
unique \`a multiplication pr\`es par un scalaire.

\begin{egs}
\begin{itemize}
\item Toute surface quartique lisse dans $\mathbb{P}^3(\mathbb{C})$
est une surface K$3.$

\smallskip
\item Tout rev\^etement double de $\mathbb{P}^2(\mathbb{C})$
ramifi\'e le long d'une courbe sextique lisse est une surface
K$3.$
\end{itemize}
\end{egs}

\noindent Soit $\mathcal{Z}$ une surface K$3$ munie d'une 
involution holomorphe $\iota.$ Lorsque $\iota$ est sans point 
fixe, le quotient est ce qu'on appelle une {\it surface
de \textsc{Enriques}}\label{ind34}; sinon il s'agit d'une surface rationnelle.

\noindent Une {\it surface minimale}\label{ind35} est une surface ne 
contenant aucune courbe rationnelle lisse d'auto-intersection
$-1.$ Les surfaces rationnelles minimales sont $\mathbb{P}^2
(\mathbb{C}),$ $\mathbb{P}^1(\mathbb{C})\times\mathbb{P}^1
(\mathbb{C})$ et $\mathrm{F}_n,$ $n\geq 2$ (\emph{voir} \cite{Be2}). 

\noindent Nous pouvons maintenant passer \`a la description des surfaces complexes compactes k\"{a}hl\'eriennes.

\begin{itemize}
\item Lorsque $\mathrm{kod}\,\mathcal{Z}=-\infty$ il y a deux familles: la premi\`ere constitu\'ee des surfaces birationnelles \`a $\mathbb{P}^2(\mathbb{C})$ (par exemple $\mathbb{P}^2(\mathbb{C}),$ $\mathbb{P}^1(\mathbb{C})\times\mathbb{P}^1(\mathbb{C}),$ les surfaces de \textsc{Hirzebruch}); la seconde form\'ee des surfaces r\'egl\'ees non rationnelles (par exemple les fibr\'es en $\mathbb{P}^1(\mathbb{C})$ au dessus d'une courbe de genre $g\geq 1$). 

\item Lorsque $\mathrm{kod}\,\mathcal{Z}=0$ il y a aussi deux familles: celle constitu\'ee des tores $\mathbb{C}^2/\Gamma$ et des surfaces hyperelliptiques (surfaces rev\^etues par des tores \`a l'aide d'un morphisme \'etale); et enfin celle form\'ee des surfaces K$3$ et surfaces de \textsc{Enriques}.

\item Lorsque $\mathrm{kod}\,\mathcal{Z}=1,$ alors $\mathcal{Z}$ est une surface elliptique.

\item Lorsque $\mathrm{kod}\,\mathcal{Z}=2,$ alors $\mathcal{Z}$ est une surface de type g\'en\'eral (exemple: toute hypersurface lisse de $\mathbb{P}^3(\mathbb{C})$ de degr\'e sup\'erieur ou \'egal \`a $5$).
\end{itemize}

\chapter{Un peu de dynamique}\label{chapdyn}

\section{Lin\'earisation}

\noindent Commen\c{c}ons par quelques rappels de lin\'earisation en dimension $1.$ Soit 
\begin{align*}
&f(z)=\alpha z+a_2z^2+a_3z^3+ \ldots, && \alpha=\mathrm{e}^{2\mathrm{i}\theta}, && \theta \in\mathbb{R}\setminus \mathbb{Q}.
\end{align*}

\noindent On cherche $h(z)=z+b_2z^2+\ldots$ tel que $fh(z)=h(\alpha z).$ Formellement
$$b_2=\frac{a_2}{\alpha^2-\alpha},\,\,\,\ldots,\,\,\, b_n=\frac{a_n+Q_n}{\alpha^n-\alpha}\,\,\, \text{ o\`u } Q_n\in\mathbb{Z}[a_i,\,\,i\leq n-1,\,\, b_i,\,\, i\leq n];$$ on dit que $f$ est {\it formellement lin\'earisable}\label{ind45}.

\begin{thm}[\textsc{Cremer}]
{\sl Si $\liminf\vert\alpha^q-\alpha\vert^{1/q}=0,$ il existe un germe analytique $f$ non lin\'earisable.

\noindent Si $\liminf\vert\alpha^q-\alpha\vert^{\frac{1}{\alpha^q}}=0$ aucun germe polynomial $f(z)=\alpha z+a_2z^2+\ldots+z^d$ n'est lin\'earisable.}
\end{thm}

\begin{thm}[\textsc{Siegel}]
{\sl S'il existe deux constantes $c$ et $M$ strictement positives telles que $\vert\alpha^q -\alpha\vert\geq\frac{c}{q^M},$ tout germe $f(z)=\alpha z+a_2z^2+\ldots$ est localement lin\'earisable.}
\end{thm}

\noindent Passons maintenant au cas de deux variables. Soit $f(x,y)=(\alpha x,\beta y)+$ termes d'ordre sup\'erieur avec $\alpha,$~$\beta$ de module $1,$ non racines de l'unit\'e. On parle de {\it r\'esonance}\label{ind46} lorsqu'on a des relations de la forme $\alpha=\alpha^a \beta^b$ ou~$\beta=~\alpha^a\beta^b$ o\`u $a,$ $b$ d\'esignent des entiers positifs tels que~$a+b\geq 2.$ Un {\it mon\^ome r\'esonant}\label{ind47} est un mon\^ome de la for\-me~$x^ay^b.$ Il y a une obstruction formelle \`a la lin\'earisation: la pr\'esence de mon\^omes r\'esonants. On dit que $\alpha$ et $\beta$ sont {\it multiplicativement ind\'ependants}\label{ind48} si l'unique solution de $\alpha^a \beta^b=1$ avec $a,$ $b$ dans $\mathbb{Z}$ est la solution nulle $(0,0).$ On dit que $\alpha$ et $\beta$ sont {\it simultan\'ement diophantiens}\label{ind49} s'il existe deux constantes strictement positives $c$ et $M$ telles que 
\begin{align*}
&\min\Big(\vert\alpha^a\beta^b-\alpha\vert,\,\vert\alpha^a\beta^b-\beta\vert\Big)\geq \frac{c}{\vert a+b \vert^M}&& \forall a,\, b\in\mathbb{N},\, a+b\geq 2. 
\end{align*}

\begin{thm}
{\sl Si $\alpha$ et $\beta$ sont simultan\'ement diophantiens, $f$ est lin\'earisable.

\noindent Si $\alpha$ et $\beta$ sont alg\'ebriques et multiplicativement ind\'ependants, ils sont simultan\'ement diophantiens.}
\end{thm}

\noindent On renvoie \`a \cite{H} pour de plus amples d\'etails sur les questions de lin\'earisation.

\section{O\`u chercher des automorphismes d'entropie positive ?}

\subsection{Entropie topologique}

\noindent Soient $X$ un espace m\'etrique compact et $f\colon X\to X$ une transformation continue. Soit $\varepsilon$ un nombre r\'eel strictement positif. Pour tout entier positif $n,$ on d\'esigne par $N(n,\varepsilon)$ le cardinal minimum d'une partie $X_n$ de $X$ telle que pour tout point $y$ de~$X$ il existe $x$ dans $X$ tel que 
\begin{align*}
&\mathrm{dist}\,(f^j(x),f^j(y))\leq \varepsilon, &&\forall\,\,\,\,\, 0\leq j\leq n.
\end{align*}

\noindent On introduit la quantit\'e $h_{\text{top}}(f,\varepsilon)$ d\'efinie par $$h_{\text{top}}(f,\varepsilon)=\limsup_{n\to +\infty}\left(\frac{1}{n}\log(N(n,\varepsilon)) \right).$$ L'{\it entropie topologique}\label{ind33b} de $f$ est donn\'ee par $$h_{\text{top}}(f)=\lim_{\varepsilon\to 0}h_{\text{top}}(f,\varepsilon);$$ cette d\'efinition ne d\'epend pas du choix de la distance.

\noindent Soit $f$ une transformation $\mathcal{C}^\infty$ sur une vari\'et\'e compacte $\mathcal{M},$ on a l'in\'egalit\'e $$h_{\text{top}}(f)\geq\log r(f^*),$$ {\it i.e.} l'entropie topologique est minor\'ee par le logarithme du rayon spectral de l'application lin\'eaire induite par $f$ sur $\mathrm{H}^*(\mathcal{M},\mathbb{R}),$ somme directe des groupes de cohomologie de $\mathcal{M}.$ Remarquons que la majoration $h_{\text{top}}(f)\geq\log r(f^*)$ est encore valable dans le cas m\'eromorphe (\cite{DS}). Lorsque $\mathcal{M}$ est k\"{a}hl\'erienne on a un r\'esultat plus pr\'ecis; avant de l'\'enonc\'e introduisons la notation suivante: pour tout entier $p$ compris entre $0$ et $\dim_\mathbb{C}\mathcal{M}$ on d\'esigne par~$\lambda_p(f)$ le rayon spectral de l'application $f^*$ agissant sur le groupe de cohomologie de Dol\-beault~$\mathrm{H}^{p,p}(\mathcal{M},\mathbb{R}).$ 

\begin{thm}[\cite{Gr1, Gr2, Yo}]
{\sl Soit $f$ une transformation holomorphe sur une vari\'et\'e complexe compacte k\"{a}hl\'erienne $\mathcal{M}.$ On a l'\'egalit\'e $$h_{\text{top}}(f)=\max_{0\leq p \leq \dim_\mathbb{C}\mathcal{M}}\log\lambda_p(f).$$}
\end{thm}

\begin{rem}
Le rayon spectral de $f^*$ est strictement plus grand que $1$ si et seulement si l'un des $\lambda_p(f)$ l'est et, en fait, si et seulement si $\lambda(f)=\lambda_1(f)$ l'est. En d'autres termes pour savoir si $f$ est d'entropie positive il suffit de calculer la croissance de $(f^n)^*\{\alpha\}$ o\`u $\{\alpha\}$ est une forme de \textsc{K\"{a}hler}.
\end{rem}

\begin{egs}
\begin{itemize}
\item Si $\mathcal{M}$ est une vari\'et\'e compacte k\"{a}hl\'erienne, tout \'el\'ement de la composante connexe de l'identit\'e de $\mathrm{Aut}\,\mathcal{M}$ est d'entropie nulle.

\item L'entropie topologique d'un endomorphisme holomorphe $f$ de l'espace projectif est \'egale au logarithme du degr\'e topologique de $f.$

\item Contrairement \`a un automorphisme \'el\'ementaire, un automorphisme de \textsc{H\'enon} est d'entropie positive.
\end{itemize}
\end{egs}

\subsection{Un th\'eor\`eme de \textsc{Cantat}}

\noindent L'\'enonc\'e suivant r\'epond \`a la question: o\`u chercher des automorphismes de surfaces complexes compactes d'entropie positive ?

\begin{thm}[\cite{Can1}]
{\sl Soit $\mathcal{Z}$ une surface complexe compacte. Supposons que $\mathcal{Z}$ poss\`ede un automorphisme $f$ dont l'entropie topologique est positive. Alors
\begin{itemize}
\item ou bien $\mathrm{kod}\,\mathcal{Z}=0$ et $f$ est conjugu\'e \`a un automorphisme de l'unique mod\`ele minimal de $\mathcal{Z}$ qui est un tore, une surface K$3$ ou une surface de \textsc{Enriques}; 

\item ou bien $\mathcal{Z}$ est rationnelle auquel cas $\mathcal{Z}$ s'obtient en \'eclatant $\mathbb{P}^2(\mathbb{C})$ en au moins $10$ points (\'eventuellement proches) et $f$ est birationnellement conjugu\'e \`a un \'el\'ement de $\mathrm{Bir}(\mathbb{P}^2).$
\end{itemize}

\noindent En particulier $\mathcal{Z}$ est une surface k\"{a}hl\'erienne.}
\end{thm}

\begin{egs}
\begin{itemize}
\item Commen\c{c}ons par un exemple d\^u \`a \textsc{Coble}. Consid\'erons une courbe sextique $\mathcal{C}$ dans $\mathbb{P}^2(\mathbb{C})$ avec $10$ points doubles. Supposons que $\mathcal{C}$ soit choisie g\'en\'erique. Soit $\mathcal{Z}$ la surface obtenue en \'eclatant ces $10$ points; on l'appelle {\it surface de \textsc{Coble}}\label{ind100}. La surface $\mathcal{Z}$ est le quotient d'une surface K$3$ par une involution ayant des points fixes et les automorphismes de $\mathcal{Z}$ proviennent de ceux de la surface K$3$ (\emph{voir} \cite{Co}). 

\noindent Le groupe $\mathrm{Aut}(\mathcal{Z})$ s'identifie \`a un sous-groupe discret et de covolume fini de $\mathrm{SO}(1,9).$ Par suite c'est un groupe infini d\'enombrable contenant un groupe libre non ab\'elien. On peut construire des automorphismes sur $\mathcal{Z}$ de la fa\c{c}on suivante. Soit $p$ un point double de $\mathcal{C}.$ L'ensemble des courbes de degr\'e $6$ passant par les $9$ autres points doubles constitue un pinceau de courbes de genre $1.$ Si on \'eclate ces $9$ points on obtient une surface~$\mathcal{Z}_p$ sur laquelle les transform\'ees strictes de ces courbes forment les fibres d'une fibration elliptique singuli\`ere $\pi_p\colon\mathcal{Z}_p\to\mathbb{P}^1 (\mathbb{C}).$ Remarquons que chacun des diviseurs exceptionnels coupe la fibre g\'en\'erique de $\pi_p$ en deux points. Consid\'erons l'un de ces diviseurs que nous noterons $\mathrm{E}_0.$ Soient $\mathrm{E}$ un autre diviseur et $f$ l'automorphisme de $\mathcal{Z}_p$ d\'efini par
\begin{itemize}
\item si $m$ est un point d'une fibre lisse $\mathrm{F}$ de $\pi_p,$ l'image $f(m)$ de $m$ est l'unique point de $\mathrm{F}$ tel que $f(m)+\mathrm{E}\cap\mathrm{F}=m+\mathrm{E}_0\cap\mathrm{F};$

\item cette transformation s'\'etend aux fibres singuli\`eres.
\end{itemize}

\noindent On remarque que $\pi_p\circ f=\pi_p.$ En faisant varier le choix de $\mathrm{E}_0$ on obtient huit automorphismes. Pour un choix g\'en\'erique de $\mathcal{C}$ ces automorphismes engendrent un groupe isomorphe \`a $\mathbb{Z}^8.$ Chacun de ces automorphismes fixe la fibre singuli\`ere obtenue par transform\'ee stricte de $\mathcal{C}$ et fixe le point singulier $p$ de cette fibre. Quitte \`a \'eclater ce point les automorphismes construits se rel\`event \`a la surface de \textsc{Coble} $\mathcal{Z}.$ En faisant varier $p$ parmi les singularit\'es de $\mathcal{C}$ on obtient $10$ copies de $\mathbb{Z}^8$ dans $\mathrm{Aut}(\mathcal{Z});$ pour un choix g\'en\'erique de $\mathcal{C}$ il n'y a pas de relation non triviale, autrement dit le produit libre de $10$ copies de $\mathbb{Z}^8$ se plonge dans $\mathrm{Aut}(\mathcal{Z}).$  

\smallskip
\item On a l'\'enonc\'e suivant, dit th\'eor\`eme de \textsc{Torelli}.
\begin{thm}[\cite{BHPV}]
{\sl Soit $\mathcal{Z}$ une surface K$3.$ Le morphisme 
\begin{align*}
&\mathrm{Aut}(\mathcal{Z})\to\mathrm{GL}(\mathrm{H}^2(\mathcal{Z},\mathbb{Z})), && f\mapsto f^*
\end{align*}

\noindent est injectif. 

\noindent R\'eciproquement supposons que $F$ soit un \'el\'ement de $\mathrm{GL}(\mathrm{H}^2(\mathcal{Z}, \mathbb{Z}))$ qui pr\'eserve la forme d'intersection sur $\mathrm{H}^2(\mathcal{Z},\mathbb{Z})$ ainsi que la d\'ecomposition de \textsc{Hodge} de $\mathrm{H}^2(\mathcal{Z},\mathbb{C})$ et le c\^one de \textsc{K\"{a}hler} de $\mathrm{H}^2(\mathcal{Z},\mathbb{R});$ alors il existe un automorphisme $f$ sur $\mathcal{Z}$ tel que~$f^*=~F.$}
\end{thm}

\smallskip
\item Lorsque $\mathcal{Z}$ est un tore $\mathbb{C}^2/\Gamma$ l'\'enonc\'e est le m\^eme modulo le fait suivant: le noyau de 
\begin{align*}
&\mathrm{Aut}(\mathcal{Z})\to\mathrm{GL}(\mathrm{H}^2(\mathcal{Z},\mathbb{Z})), && f\mapsto f^*
\end{align*}

\noindent est engendr\'e par les translations et l'involution $(x,y)\mapsto(-x,-y).$
\end{itemize}
\end{egs}

\noindent Le cas des surfaces K$3$ a \'et\'e \'etudi\'e par 
\textsc{Cantat} et \textsc{McMullen} (\cite{Can2, Mc2, Si, Wa}) ; celui des surfaces rationnelles est celui
qui produit le plus de tels automorphismes (\cite{Mc, BK1, BK2}) 
d'o\`u l'int\'er\^et qu'il suscite. N\'eanmoins on ne sait pas
d\'eterminer l'ensemble des surfaces rationnelles admettant
des automorphismes non triviaux et \'etant donn\'ee une 
surface on est incapable de d\'ecrire ses automorphismes.

\subsection{Cas des surfaces rationelles non minimales}

\noindent Jusqu'\`a aujourd'hui on s'est essentiellement int\'eress\'es
aux surfaces obtenues en \'eclatant $\mathbb{P}^2(\mathbb{C})$
un nombre fini de points; ceci est justifi\'e par le th\'eor\`eme
suivant d\^u \`a \textsc{Nagata}.

\begin{thm}[\cite{Na}]\label{nagata}
{\sl Soient $\mathcal{Z}$ une surface rationnelle et $f$ un automorphisme sur $\mathcal{Z}$ tel que $f_*$ soit d'ordre infini. Il existe une suite d'applications d'\'eclatement $\pi_{j+1}\colon\mathcal{Z}_{j+1}\to\mathcal{Z}_j$ d'un point $p_j$ de $\mathcal{Z}_j$ telles que $\mathcal{Z}_1=\mathbb{P}^2(\mathbb{C}),$ $\mathcal{Z}_{N+1}=\mathcal{Z}.$}
\end{thm}

\noindent N\'eanmoins remarquons qu'une surface obtenue en \'eclatant $\mathbb{P}^2(\mathbb{C})$ en des points g\'en\'eriques n'a pas d'automorphisme non trivial (\cite{Hi, Ko}). De plus on a l'\'enonc\'e suivant que l'on peut par exemple trouver dans \cite{Di2}.

\begin{pro}
{\sl Soient $\mathcal{Z}$ une surface obtenue en \'eclatant $\mathbb{P}^2(\mathbb{C})$ en $n\leq 9$ points et~$f$ un automorphisme sur $\mathcal{Z}.$ L'entropie topologique $\log\lambda(f)$ de $f$ est nulle. De plus, si $n$ est inf\'erieur ou \'egal \`a $8,$ un it\'er\'e de~$f$ est birationnellement conjugu\'e \`a un automorphisme du plan projectif complexe.}
\end{pro}

\begin{proof}[{\sl D\'emonstration}]
Raisonnons par l'absurde: supposons que $f$ soit d'entropie positive. D'apr\`es \cite{Can1} il existe~$\theta$ dans $\mathrm{H}^2(\mathcal{Z},\mathbb{R})$ d'auto-intersection nulle et tel que $f^*\theta =\lambda(f)\theta.$ Si $\mathrm{K}_\mathcal{Z}$ d\'esigne la classe du diviseur canonique sur $\mathcal{Z},$ on a $f_*\mathrm{K}_\mathcal{Z}=f^*\mathrm{K}_\mathcal{Z}=\mathrm{K}_\mathcal{Z}.$ On constate que 
$$\lambda(f)^{-1}\langle\theta,\mathrm{K}_\mathcal{Z}\rangle=\langle f^*\theta,\mathrm{K}_\mathcal{Z}\rangle=\langle \theta,f_*\mathrm{K}_\mathcal{Z}\rangle=\langle\theta,\mathrm{K}_\mathcal{Z}\rangle$$ d'o\`u $\langle\theta, \mathrm{K}_\mathcal{Z}\rangle=0.$ Comme la forme d'intersection sur $\mathcal{Z}$ a une unique valeur propre positive et comme $\mathrm{K}_\mathcal{Z}^2\geq 0$ pour $n\leq 9,$ la classe $\theta$ s'\'ecrit $c\mathrm{K}_\mathcal{Z}$ pour un certain $c<0.$ Mais alors~$f^*\theta=\theta\not=\lambda(f)\theta:$ contradiction.

\noindent Supposons que $n$ soit inf\'erieur ou \'egal \`a $8$ alors $\mathrm{K}_\mathcal{Z}^2>0.$ Il s'en suit que la forme d'intersection est strictement n\'egative sur le compl\'ement orthogonal $\mathrm{H}\subset\mathrm{H}^2(\mathcal{Z},\mathbb{R})$ de $\mathrm{K}_\mathcal{Z}.$ Comme $\mathrm{H}$ est de dimension finie et invariant par $f^*$ et comme $f^*$ pr\'eserve $\mathrm{H}^2(\mathcal{Z}, \mathbb{Z}),$ la restriction de $f^*$ \`a~$\mathrm{H}$ est d'ordre fini. Il en r\'esulte l'existence d'un entier $k$ tel que $f^{k*}$ soit trivial. En particulier $f^k$ pr\'eserve les diviseurs exceptionnels obtenus en \'eclatant les $n$ points de $\mathbb{P}^2(\mathbb{C});$ par suite $f^k$ appartient, \`a conjugaison pr\`es, \`a $\mathrm{Aut}(\mathbb{P}^2(\mathbb{C})).$
\end{proof}

\section{Propri\'et\'es d'un automorphisme d'entropie positive}

\noindent Soit $f$ un automorphisme sur une surface k\"{a}hl\'erienne d'entropie positive. Comme l'indique le r\'esultat qui suit les valeurs propres de $f^*$ v\'erifient des propri\'et\'es particuli\`eres.

\begin{thm}\label{lambda}
{\sl Soit $f$ un automorphisme sur une surface k\"{a}hl\'erienne tel que $\lambda(f)>1.$ Le premier degr\'e dynamique $\lambda(f)$ est une valeur propre de $f^*$ avec multiplicit\'e $1$ et c'est l'unique valeur propre de module strictement sup\'erieur \`a~$1.$

\noindent Si $\eta$ est une valeur propre de $f^*,$ alors soit $\eta$ appartient \`a $\{\lambda(f), 1/\lambda(f)\},$ soit $\vert\eta\vert=1.$}
\end{thm}

\noindent Soit $f$ un automorphisme sur une surface k\"{a}hl\'erienne tel que $\lambda(f)>1.$ Notons $\chi_f$ le polyn\^ome caract\'eristique de $f^*.$ C'est un polyn\^ome monique dont le terme constant, qui est le d\'eterminant de $f^*,$ vaut $\pm 1.$ D\'esignons par $\psi_f$ le polyn\^ome minimal de $\lambda(f).$ D'apr\`es le Th\'eor\`eme \ref{lambda} les racines de $\chi_f,$ et celles de $\psi_f,$ except\'ees $\lambda(f)$ et $1/\lambda(f),$ sont sur le cercle unit\'e. Un tel $\psi_f$ est un {\it polyn\^ome de \textsc{Salem}}\label{ind39} et $\lambda(f)$ est un {\it nombre de \textsc{Salem}}\label{ind25}. Le plus petit nombre de \textsc{Salem} connu est la racine $\lambda_{\text{Lehmer}}$ du polyn\^ome de \textsc{Lehmer} $$L(t)=t^{10}+t^9-t_7-t^6-t^5-t^4-t^3+t+1$$ dont nous reparlerons au Chapitre \ref{mcmullen}.

\section{Trois involutions}

\noindent Le groupe $\mathrm{PGL}_3(\mathbb{C})\times\mathrm{PGL}_3(\mathbb{C})$ agit sur $\mathrm{Bir}(\mathbb{P}^2)$ de la fa\c{c}on suivante
\begin{align*}
& \mathrm{PGL}_3(\mathbb{C})\times\mathrm{Bir}(\mathbb{P}^2) \times\mathrm{PGL}_3(\mathbb{C})\to \mathrm{Bir}(\mathbb{P}^2)
&&(A,f,B)\mapsto AfB^{-1}.
\end{align*}

\noindent Il est naturel de s'int\'eresser aux transformations birationnelles quadratiques (\cite{CeDe}); en effet toute transformation birationnelle du plan projectif complexe s'\'ecrit comme un produit de transformations quadratiques (Th\'eor\`eme \ref{nono}). Introduisons les trois involutions birationnelles quadratiques suivantes qui jouent un r\^ole particulier
\begin{align*}
& \sigma=(yz:xz:xy), &&\rho=(xy:z^2:yz),&&\tau=(x^2:xy:y^2-xz). 
\end{align*}

\noindent Soit $f$ une transformation birationnelle de degr\'e $2;$ alors $f$ est de l'une des formes suivantes
\begin{align*}
& A\sigma B, &&A\rho B, && A\tau B, && A,\, B\in\mathrm{PGL}_3(\mathbb{C}).
\end{align*}

\noindent Les transformations $A\sigma B,$ avec $A,$ $B$ automorphismes de $\mathbb{P}^2(\mathbb{C}),$ sont les transformations birationnelles quadratiques g\'en\'eriques; en effet  l'adh\'erence de l'ensemble des transformations birationnelles quadratiques est une vari\'et\'e irr\'eductible et si $\mathcal{O}(f)$ d\'esigne l'orbite de $f$ sous l'action de~$\mathrm{PGL}_3(\mathbb{C})\times\mathrm{PGL}_3(\mathbb{C})$ on peut v\'erifier que (\cite{CeDe}) 
\begin{align*}
&\dim\mathcal{O}(\sigma)=14, && \dim\mathcal{O}(\rho)=13, &&\dim\mathcal{O}(\tau) =12.
\end{align*}

\subsection{\'Etude de $\sigma$}\label{sigma}

Les points d'ind\'etermination de $\sigma$ sont $P=(1:0:0),$ $Q=(0:1:0)$ et~$R=(0:0:1);$ quant au lieu exceptionnel, il est form\'e des droites $\Delta=\{x=0\},$ $\Delta'=\{y=0\}$ et $\Delta''=\{z=0\}.$ 

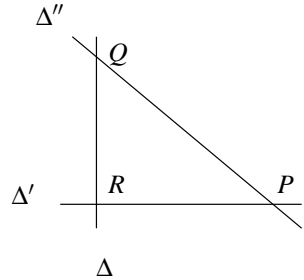
\begin{figure}[H]
\begin{center}
\input{sigma1.pstex_t}
\end{center}
\end{figure}

\noindent Commen\c{c}ons par \'eclater le point $P;$ notons $\mathrm{E}$ le diviseur exceptionnel et $\mathcal{D}_1$ la transform\'ee stricte de $\mathcal{D}:$

\begin{align*}
&\left\{\begin{array}{ll} y=u_1\\ z=u_1v_1\end{array}\right. && \begin{array}{ll} \mathrm{E}=\{u_1 =0\}\\ \Delta''_1=\{v_1=0\}\end{array} && \hspace{2cm} &&\left\{\begin{array}{ll} y=r_1s_1\\ z=s_1\end{array}\right. && \begin{array}{ll} \mathrm{E}=\{s_1 =0\}\\ \Delta'_1=\{r_1=0\}\end{array}
\end{align*}

\begin{figure}[H]
\begin{center}
\input{sigma2.pstex_t}
\end{center}
\end{figure}

\noindent D'une part $$(u_1,v_1)\to(u_1,u_1v_1)_{(y,z)}\to(u_1v_1:v_1:1)=\left(\frac{1}{u_1}, \frac{1}{u_1v_1}\right)_{(y,z)}\to \left(\frac{1}{u_1},\frac{1}{v_1}\right)_{(u_1,v_1)};$$ d'autre part $$(r_1,s_1)\to(r_1s_1,s_1)_{(y,z)}\to(r_1s_1:1:r_1)=\left(\frac{1}{r_1s_1}, \frac{1}{s_1} \right)_{(y,z)} \to\left(\frac{1}{r_1},\frac{1}{s_1}\right)_{(r_1,s_1)}.$$ Par suite $\mathrm{E}$ est envoy\'e sur $\Delta_1;$ comme $\sigma$ est une involution $\Delta_1$ est envoy\'e sur $\mathrm{E}.$

\noindent \'Eclatons le point $Q_1;$ cette fois nous d\'esignerons par $\mathrm{F}$ le diviseur exceptionnel et par $\mathcal{D}_2$ la transform\'ee stricte de $\mathcal{D}_1:$

\begin{align*}
&\left\{\begin{array}{ll} x=u_2\\ z=u_2v_2\end{array}\right. && \begin{array}{ll} \mathrm{F}=\{u_2 =0\}\\ \Delta''_2=\{v_2=0\}\end{array} && \hspace{2cm} &&\left\{\begin{array}{ll} x=r_2s_2\\ z=s_2\end{array}\right. && \begin{array}{ll} \mathrm{E}=\{s_2 =0\}\\ \Delta_2=\{r_2=0\}\end{array}
\end{align*}

\begin{figure}[H]
\begin{center}
\input{sigma3.pstex_t}
\end{center}
\end{figure}

\noindent On a $$(u_2,v_2)\to(u_2,u_2v_2)_{(x,z)}\to(v_2:u_2v_2:1)=\left(\frac{1}{u_2},\frac{1}{u_2v_2}\right)_{(x,z)}\to\left(\frac{1}{u_2},\frac{1}{v_2}\right)_{(u_2,v_2)}$$ et $$(r_2,s_2)\to(r_2s_2,s_2)_{( x,z)}\to(1:r_2s_2:r_2)=\left(\frac{1}{r_2s_2},\frac{1}{s_2}\right)_{(x,z)}\to\left(\frac{1}{r_2},\frac{1}{s_2}\right)_{(r_2,s_2)}.$$

\noindent Il s'en suit que $\mathrm{F}\to\Delta'_2$ d'o\`u $\Delta'_2\to\mathrm{F}.$

\noindent Enfin \'eclatons le point $R_2;$ si $\mathrm{G}$ est le diviseur exceptionnel on a 
\begin{align*}
&\left\{\begin{array}{ll} x=u_3\\ y=u_3v_3\end{array}\right. && \begin{array}{ll} \mathrm{G}=\{u_3 =0\}\\ \Delta''_3=\{v_3=0\}\end{array} && \hspace{2cm} &&\left\{\begin{array}{ll} x=r_3s_3\\ z=s_3\end{array}\right. && \begin{array}{ll} \mathrm{E}=\{s_3 =0\}\\ \Delta_2=\{r_3=0\}\end{array}
\end{align*}

\begin{figure}[H]
\begin{center}
\input{sigma4.pstex_t}
\end{center}
\end{figure}

\noindent On constate que $$(u_3,v_3)\to(u_3,u_3v_3)_{(x,y)}\to(v_3:1:u_3v_3)=\left(\frac{1}{u_3}, \frac{1}{u_3v_3}\right)_{(x,y)} \to\left(\frac{1}{u_3},\frac{1}{v_3}\right)_{(u_3,v_3)}$$ et $$(r_3 ,s_3)\to(r_3s_3,s_3)_{(x,y)}\to(1:r_3:r_3s_3)=\left(\frac{1}{r_3s_3},\frac{1}{s_3}\right)_{(x,y)}\to \left(\frac{1}{r_3},\frac{1}{s_3}\right)_{(r_3,s_3)}.$$ Il en r\'esulte que $\mathrm{G}\to \Delta'_3$ et $\Delta'_3\to\mathrm{G}.$ Il n'y a plus de point d'ind\'etermination, plus de courbe contract\'ee; autrement dit $\sigma$ est conjugu\'e \`a un automorphisme de $\mathbb{P}^2( \mathbb{C})$ \'eclat\'e en $P,$ $Q_1$ et $R_2.$

\bigskip

\noindent Pla\c{c}ons-nous dans la base $\{\mathrm{H},\,\mathrm{E},\,\mathrm{F},\,\mathrm{G}\}.$ Au premier \'eclatement $\Delta$ et $\mathrm{E}$ sont \'echang\'es; le point \'eclat\'e \'etant l'intersection de $\Delta'$ et $\Delta''$ on a $\Delta\to\Delta+\mathrm{F}+\mathrm{G}.$ Par suite $\sigma^*\mathrm{E}=\mathrm{H}-\mathrm{F}- \mathrm{G}.$ De m\^eme $\sigma^*\mathrm{F}=\mathrm{H}-\mathrm{E}-\mathrm{G}$ et $\sigma^*\mathrm{G}=\mathrm{H}- \mathrm{E}-\mathrm{F}.$ Reste \`a d\'eterminer $\sigma^*\mathrm{}.$ L'image d'une droite g\'en\'erique par $\sigma$ est une conique d'o\`u $\sigma^*\mathrm{H}=2\mathrm{H}+m_1\mathrm{E}+m_2\mathrm{F}+m_3\mathrm{G}.$ Soit $\mathrm{L}$ une droite g\'en\'erique d'\'equation $\ell=0$ o\`u $\ell=a_0x+a_1y+a_2z.$ Un calcul montre que $$(u_1,v_1)\to(u_1,u_1v_1)_{(y,z)}\to(u_1^2v_1:u_1v_1:u_1)\to u_1(a_0v_2+a_1u_2v_2+a_2)$$ s'annule \`a l'ordre $1$ sur $\mathrm{E}=\{u_1 =0\}$ d'o\`u $m_1=0.$ De m\^eme on constate que $$(u_2,v_2)\to(u_2,u_2v_2)_{(x,z)}\to(u_2v_2:u_2^2v_2:u_2)\to u_2(a_0v_2+a_1u_2v_2+a_2),$$ resp. $$(u_3,v_3)\to(u_3,u_3v_3)_{(x,y)}\to(u_3v_3:u_3:u_3^2v_3)\to u_3(a_0v_3+a_1+a_2u_3v_3)$$ s'annule \`a l'ordre $1$ sur $\mathrm{F}=\{u_2=0\},$ resp. $\mathrm{G}=\{u_3=0\}$ d'o\`u $m_2=-1,$ resp. $m_3=-1.$ Par suite $\sigma^*\mathrm{H}=2\mathrm{H}-\mathrm{E}-\mathrm{F}-\mathrm{G}$ et la matrice caract\'eristique de $\sigma$ dans la base $\{\mathrm{H},\,\mathrm{E},\,\mathrm{F},\,\mathrm{G}\}$ est $$M_\sigma=\left[\begin{array}{cccc} 2 & 1 &1 & 1\\ -1 & 0 & -1 & -1 \\ -1 & -1 & 0 & -1 \\ -1 & -1 & -1 & 0 \end{array}\right].$$ 

\subsection{\'Etude de $\rho$}\label{rho}

\`A partir des plongements naturels 
\begin{align*}
&\mathbb{C}^2\subset\mathbb{P}^2(\mathbb{C}),\,(x:y:1)\mapsto(x:y:z),&&
\mathbb{C}\times\mathbb{C}\subset\mathbb{P}^1(\mathbb{C})\times\mathbb{P}^1(\mathbb{C}),\, (x,y)\mapsto((x:1),(y:1))
\end{align*}

\noindent on obtient l'identification suivante
$$\mathbb{P}^2(\mathbb{C})\longleftrightarrow\mathbb{P}^1(\mathbb{C})\times\mathbb{P}^1(\mathbb{C}),\,(x:y:z)\longleftrightarrow((x:z),(y:z)).$$

\noindent On a donc
$$\xymatrix {
(x:y:z)\ar@{<->}[r]\ar[d]_{\rho} & ((x:z),(y:z))\ar[d]^{\rho}\\
 (xy:z^2:yz)\ar@{<->}[r] & ((xy:yz),(z^2:yz))=((x:z),(z:y))}$$
 puisque $((x:z),(z:y))$ est holomorphe sur $\mathbb{P}^1(\mathbb{C})\times\mathbb{P}^1( \mathbb{C}),$ la transformation $\rho$ est un automorphisme sur~$\mathbb{P}^2( \mathbb{C})$ \'eclat\'e en un nombre fini de points. D\'etaillons ceci.

\smallskip

\noindent On remarque que 
\begin{align*}
&\mathrm{Ind}\,\rho=\{P,\,Q\}, &&
\mathrm{Exc}\,\rho=\Delta\cup\Delta'
\end{align*}

\noindent o\`u

\begin{align*}
&P=(0:1:0),\, Q=(1:0:0), && \Delta=\{y=0\},\,
\Delta'=\{z=0\}. 
\end{align*}

\begin{figure}[H]
\begin{center}
\input{rho1.pstex_t}
\end{center}
\end{figure}

\noindent Posons:

\begin{align*}
&\left\{\begin{array}{ll} x=r_1s_1\\ z=s_1\end{array}\right. && \begin{array}{ll} \mathrm{E}=\{s_1 =0\}\\ \end{array} && \hspace{2cm} &&\left\{\begin{array}{ll} x=u_1\\ z=u_1v_1\end{array}\right. && \begin{array}{ll} \mathrm{E}=\{u_1=0\}\\ \Delta'=\{v_1=0\}\end{array}
\end{align*}

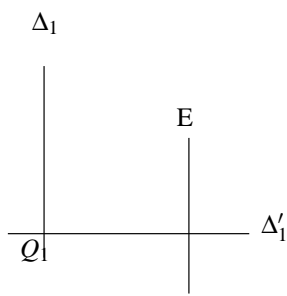
\begin{figure}[H]
\begin{center}
\input{rho2.pstex_t}
\end{center}
\end{figure}

\noindent On a $(r_1,s_1)\to(r_1s_1,s_1)_{(x,z)}\to(r_1:s_1:1);$ par suite $\mathrm{E}$ est envoy\'e sur $\Delta.$ Puisque $\rho$ est une involution $\Delta\to\mathrm{E}.$

\noindent Par ailleurs $(u_1,v_1)\to(u_1,u_1v_1)_{(x,z)}\to(1:u_1v_1^2:v_1).$ 

\noindent \'Eclatons $Q_1,$ posons
\begin{align*}
&\left\{\begin{array}{ll} y=r_2s_2\\ z=s_2\end{array}\right. && \begin{array}{ll} \mathrm{F}=\{s_2 =0\}\\ \Delta_2=\{r_2=0\}\end{array} && \hspace{2cm} &&\left\{\begin{array}{ll} y=u_2\\ z=u_2v_2\end{array}\right. && \begin{array}{ll} \mathrm{F}=\{u_2 =0\}\\ \Delta'_2=\{v_2=0\}\end{array}
\end{align*}

\begin{figure}[H]
\begin{center}
\input{rho3.pstex_t}
\end{center}
\end{figure}

\noindent On constate que $(u_2,v_2)\to(u_2,u_2v_2)_{(y,z)}\to(u_2v_2^2,u_2v_2)_{(y,z)}\to (v_2, u_2v_2)_{(r_2,s_2)};$ ainsi $\mathrm{F}\to\mathrm{F}$ et $\Delta'_2\to Q_2.$

\noindent On a aussi: $(r_2,s_2)\to(r_2s_2,s_2)_{(y,z)}\to(r_2:s_2:r_2s_2).$ 

\noindent Pour finir \'eclatons $Q_2,$ posons
\begin{align*}
&\left\{\begin{array}{ll} r_2=r_3s_3\\ s_2=s_3\end{array}\right. && \begin{array}{ll} \mathrm{G}=\{s_3 =0\}\\ \Delta_3=\{r_3=0\}\end{array} && \hspace{2cm} &&\left\{\begin{array}{ll} r_2=u_3\\ s_2=u_3v_3\end{array}\right. && \begin{array}{ll} \mathrm{G}=\{u_3=0\}\\ \mathrm{F}_1=\{v_3=0\}\end{array}
\end{align*}

\begin{figure}[H]
\begin{center}
\input{rho4.pstex_t}
\end{center}
\end{figure}

\noindent On a: $(u_3,v_3)\to(u_3,u_3v_3)_{(r_2,s_2)}\to(u_3^2v_3,u_3v_3)_{(y,z)}\to(1:v_3:u_3v_3)$ et $$(r_3,s_3)\to(r_3s_3,s_3)_{(r_2,s_2)}\to(r_3s_3^2,s_3)_{(y,z)}\to(r_3:1:r_3s_3).$$ Par suite il n'y a plus de point d'ind\'etermination, $\mathrm{G}$ est envoy\'e sur $\Delta'_3$ et puisque $\rho$ est une involu\-tion,~$\Delta'_3$ est envoy\'e sur $\mathrm{G}.$ 

\noindent Par cons\'equent $\rho$ est conjugu\'e \`a un automorphisme de $\mathrm{Bl}_{P,Q_1,Q_2}\mathbb{P}^2.$ 

\bigskip

\noindent On travaille dans la base $\{\mathrm{H},\,\mathrm{E},\,\mathrm{F},\,\mathrm{G}\}.$ Au premier \'eclatement $\mathrm{E}$ est envoy\'e sur $\Delta;$ au second $\mathrm{F}$ est fix\'e et $\Delta$ est contract\'e sur~$\mathrm{F}\cap\Delta.$ Enfin au dernier $\mathrm{G}$ est envoy\'e sur $\Delta'.$ Le second \'eclatement \'etant l'\'eclatement d'un point sur $\Delta$ on a $\Delta\to\Delta + \mathrm{F};$ puisque le dernier consiste \`a \'ecla\-ter~$\mathrm{F}\cap\Delta$ on a $\Delta \to\Delta + \mathrm{F}\to\Delta+\mathrm{F}+2\mathrm{G}.$ Puisque les deux premiers points \'eclat\'es sont sur $\Delta'$ on a~$\Delta'\to\Delta'+\mathrm{E}+\mathrm{F};$ pour finir on \'eclate un point de $\mathrm{F}\cap\Delta$ d'o\`u $\Delta'\to\Delta' +\mathrm{E}+\mathrm{F}\to\Delta'+\mathrm{E}+\mathrm{F}+\mathrm{G}.$

\noindent Par suite
\begin{align*}
&\rho^*\mathrm{E}=\mathrm{H}-\mathrm{F}-2\mathrm{G},&&\rho^*\mathrm{F}=\mathrm{F}, &&
\rho^*\mathrm{G}=\mathrm{H}-\mathrm{E}-\mathrm{F}-\mathrm{G}.
\end{align*}

\noindent Reste \`a d\'eterminer $\rho^*\mathrm{H}.$ L'image d'une droite g\'en\'erique par $\rho$ \'etant une conique $\rho^*\mathrm{H}$ s'\'ecrit aussi $2\mathrm{H}+m_0\mathrm{E}+m_1\mathrm{F}+m_2 \mathrm{G}.$ D\'eterminons les $m_i.$ Soit $\mathrm{L}$ une droite g\'en\'erique; elle a pour \'equation $\ell=0$ avec $\ell=a_0x+a_1y +a_2z.$ Un calcul montre que
$$(r_1,s_1)\to(r_1s_1,s_1)_{(x,z)}\to(r_1s_1:s_1^2:s_1)\to s_1(a_0r_1+a_1s_1+a_2)$$ s'annule \`a l'ordre $1$ sur $\mathrm{E}=\{s_1=0\}$ d'o\`u $m_0=1.$ On constate que 
$$(r_2,s_2)\to(r_2s_2,s_2)_{(y,z)}\to(r_2s_2:s_2^2:r_2s_2^2)\to s_2(a_0r_2+a_1s_2+a_2r_2s_2)$$
s'annule \`a l'ordre $1$ sur $\mathrm{F}=\{s_2=0\}$ d'o\`u $m_1=1.$ On remarque que
$$(r_3,s_3)\to(r_3s_3^2,s_3)_{(y,z)}\to(r_3s_3^2:s_3^2:r_3s_3^3)\to s_3^2(a_0r_3+a_1+a_2r_3)$$
s'annule \`a l'ordre $2$ sur $\mathrm{G}=\{s_3=0\},$ par suite $m_2=2.$ On en d\'eduit que la matrice caract\'eristique de $\rho$ dans la base $\{\mathrm{H},\,\mathrm{E},\,\mathrm{F},\,\mathrm{G}\}$ $$M_\rho=\left[\begin{array}{cccc}1 & 1 & 0 & 1\\ -1 & 0 & 0 & -1\\ -1 & -1 & 1 & -1 \\ -2 & -2 & 0 & -1\end{array}\right].$$

\subsection{\'Etude de $\tau$}\label{tau}

\noindent Rappelons que $\tau$ ne contracte qu'une droite $\Delta=\{x=0\}$ et n'\'eclate qu'un point $P=(0:0:1).$

\begin{figure}[H]
\begin{center}
\input{tau4.pstex_t}
\end{center}
\end{figure}

\noindent \'Eclatons le point $P;$ posons $x=r_1s_1$ et $y=s_1.$

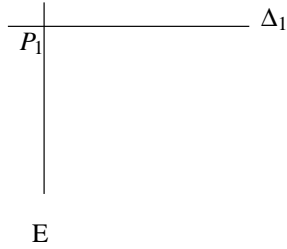
\begin{figure}[H]
\begin{center}
\input{tau3.pstex_t}
\end{center}
\end{figure}

\noindent On constate que $(r_1,s_1)$ sont des coordonn\'ees au voisinage de $P_1=(0,0)_{(r_1,s_1)}$ dans lesquelles le diviseur exceptionnel $\mathrm{E}$ est donn\'e par $\{s_1=0\}$ et $\Delta_1$ par $\{r_1=0\}.$ On a 
\begin{align*}
&(r_1,s_1)\to(r_1s_1,r_1)_{(x,y)}\to(r_1^2s_1:r_1s_1:s_1-r_1)\\
&\hspace{1cm}= \left(\frac{r_1^2s_1}{s_1-r_1},\frac{r_1s_1}{s_1-r_1}\right)_{(x,y)}\to \left(r_1,\frac{r_1 s_1}{s_1-r_1}\right)_{(r_1,s_1)}.
\end{align*}

\noindent On en d\'eduit que $P_1=(0,0)_{(r_1,s_1)}$ est ind\'etermin\'e, $\Delta_1$ est contract\'e sur $P_1$ et $\mathrm{E}$ est fix\'e.

\noindent \'Eclatons le point $P_1.$ Posons $r_1=r_2s_2$ et $s_1=s_2.$

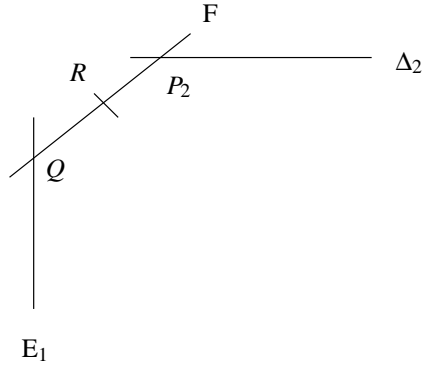
\begin{figure}[H]
\begin{center}
\input{tau2.pstex_t}
\end{center}
\end{figure}

\noindent On remarque que $(r_2,s_2)$ sont des coordonn\'ees au voisinage de $(0,0)_{(r_2,s_2)},$ coordonn\'ees dans lesquelles le diviseur exceptionnel $\mathrm{F}$ a pour \'equation $\{s_2=0\}$ et $\Delta_2=\{r_2=0\}.$ On a 
\begin{align*}
&(r_2,s_2) \to(r_2s_2,s_2)_{(r_1,s_1)}\to(r_2^2s_2^2:r_2s_2:1-r_2)\to\left( \frac{r_2^2s_2^2}{1-r_2},\frac{r_2s_2}{1-r_2}\right)_{(x,y)}\\
&\hspace{1cm}\to\left(r_2s_2,\frac{r_2s_2}{1-r_2}\right)_{(r_1,s_1)}\to\left(1-r_2, \frac{r_2s_2}{1-r_2}\right)_{(r_2,s_2)}.
\end{align*}

\noindent On constate que $P_2=(1,0)_{(r_2,s_2)}$ est d'ind\'etermination; par ailleurs $\mathrm{F}$ est fix\'e et $\Delta_2$ contract\'e sur $P_2.$

\noindent \'Eclatons le point $P_2.$ Posons $r_2=u_3+1$ et $s_2=u_3v_3$

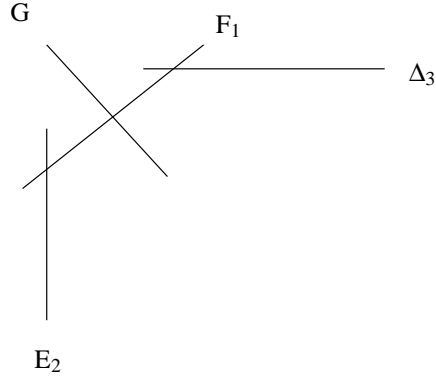
\begin{figure}[H]
\begin{center}
\input{tau1.pstex_t}
\end{center}
\end{figure}

\noindent Dans ce cas on note $\mathrm{G}=\{u_3 =0\}$ le diviseur exceptionnel. Dans les coordonn\'ees $(u_3,v_3)$ on~a 
\begin{align*}
&(u_3,v_3)\to(u_3+1,u_3v_3)_{(r_2,s_2)}\to((u_3+1)^2u_3v_3^2:v_3(u_3+1):-1)\\
&\hspace{1cm}=\left( -(u_3+1)^2u_3v_3^2,-v_3(u_3+1)\right)_{(x,y)} \to\left((u_3+1)u_3v_3,-v_3(u_3+1)\right)_{(r_1,s_1)}\\
&\hspace{1cm}\to\left(-u_3,-v_3(u_3+1)\right)_{(r_2,s_2)}\to\left(-u_3-1,v_3\right)_{(u_3,v_3)};
\end{align*}

\noindent on constate que $\mathrm{G}\to\{u_3=-1\}.$ Par ailleurs il n'y a plus de point d'ind\'etermination.

\noindent Ainsi $\tau$ est conjugu\'e \`a un automorphisme de $\mathbb{P}^2(\mathbb{C})$ \'eclat\'e en $P,$ $P_1$ et $P_2.$

\bigskip

\noindent On se place dans la base $\{\mathrm{H},\, \mathrm{E},\,\mathrm{F},\,\mathrm{G}\}.$ Apr\`es avoir \'eclat\'e $P,$ on constate que le diviseur exceptionnel $\mathrm{E}$ est pr\'eserv\'e et $\Delta$ est envoy\'e sur $\mathrm{E}\cap\Delta.$ L'\'eclatement de $P_1=\mathrm{E}\cap\Delta_1$ envoie $\mathrm{F}$ sur~$\mathrm{F}$ et $\Delta$ sur  $P_2=\mathrm{F}\cap\Delta_2.$ Enfin l'\'eclatement de $P_2$ envoie $\mathrm{G}$ sur $\Delta.$ Puisque $P$ est sur $\Delta$ on a~$\Delta\to\Delta + \mathrm{E}.$ Ensuite on \'eclate le point $P_1=\Delta_1\cap\mathrm{E}$ d'o\`u $$\Delta\to\Delta + \mathrm{E}\to\Delta +\mathrm{F}+\mathrm{E}+\mathrm{F}=\Delta+\mathrm{E}+2 \mathrm{F}.$$ Pour finir on \'eclate $P_2$ qui est sur $\mathrm{F}$ d'o\`u $\Delta\to\Delta + \mathrm{E}\to\Delta+\mathrm{E}+2 \mathrm{F}\to\Delta+\mathrm{E}+2 \mathrm{F}+2 \mathrm{G}.$ Il en r\'esulte que
\begin{align*}
&\tau^*\mathrm{E}=\mathrm{E},&&\tau^*\mathrm{F}=\mathrm{F}, &&
\tau^*\mathrm{G}=\mathrm{H}-\mathrm{E}-2\mathrm{F}-2\mathrm{G}.
\end{align*}

\noindent Reste \`a d\'eterminer $\tau^*\mathrm{H}.$ Puisque l'image d'une droite g\'en\'erique par $\tau$ est une conique on a~$\tau^*\mathrm{H}=2\mathrm{H}+m_0\mathrm{E}+m_1\mathrm{F}+m_2 \mathrm{G}.$ Calculons les $m_j.$ Soit $\mathrm{L}$ une droite g\'en\'erique dans $\mathbb{P}^2(\mathbb{C});$ elle a pour \'equation $\ell=0$ o\`u $\ell=a_0x+a_1y+a_2z.$ On constate que $$(r_1,s_1)\to(r_1s_1,s_1)_{(x,y)}\to(r_1^2s_1^2:r_1s_1^2:s_1^2-r_1s_1)\to s_1(a_0r_1^2s_1+a_1r_1 s_1+a_2(s_1-r_1))$$ s'annule \`a l'ordre $1$ sur $\mathrm{E}=\{s_1=0\}$ donc $m_0=1.$ De m\^eme un calcul montre que 
$$(r_2,s_2)\to(r_2s_2^2,s_2)_{(x,y)}\to(r_2^2s_2^4:r_2s_2^3:s_2^2-r_2s_2^2)\to s_2^2(a_0r_2^2s_2^2+a_1r_2s_2+a_2(1-r_2))$$

\noindent 
et $\mathrm{F}=\{s_2=0\}$ d'o\`u $m_1=2.$ Enfin on obtient que 
\begin{align*}
&(u_3,v_3)\to(u_3+1,u_3v_3)_{(r_2,s_2)}\to((u_3+1)u_3^2v_3^2,u_3v_3)_{(x,y)}\\
&\hspace{1cm}\to u_3^3(a_0(u_3+1)^2u_3 v_3^4+a_1(u_3+1)v_3^3-a_2v_3^2)
\end{align*}

\noindent 
et $\mathrm{G}= \{u_3=0\}$ ainsi $m_2=3.$ Par suite la matrice caract\'eristique de $\tau$ dans la base $\{\mathrm{H},\,\mathrm{E},\,\mathrm{F},\,\mathrm{G}\}$ $$M_\tau=\left[\begin{array}{cccc}1 & 0 & 0 & 1 \\ -1 & 1 & 0 & -1 \\ -2 & 0 & 1 & -2 \\ -3 & 0 & 0 & -2 \end{array}\right].$$

\section{Ensembles de \textsc{Fatou}}

\subsection{D\'efinitions et propri\'et\'es}

\noindent Soit $f$ un automorphisme sur une vari\'et\'e $\mathcal{M}$ complexe, compacte. Rappelons que l'{\it ensemble de \textsc{Fatou}}\label{ind50} $\mathcal{F}(f)$ de $f$ est l'ensemble des points qui poss\`edent un voisinage $\mathcal{V}$ tel que~$\{f^n_{\vert\mathcal{V}},\, n\geq~0\}$ soit une famille normale. Consid\'erons l'ensemble $$\mathcal{G}=\mathcal{G}(\mathcal{U})=\{g\colon \mathcal{U}\to\overline{\mathcal{U}}\,\vert\, g=\lim_{n_j\to +\infty} f^{n_j}\}.$$ On dit que $\mathcal{U}$ est un {\it domaine de rotation}\label{ind51} si $\mathcal{G}$ est un sous-groupe de $\mathrm{Aut}(\mathcal{U})$ autrement dit si tout \'el\'ement de~$\mathcal{G}$ d\'efinit un automorphisme de $\mathcal{U}.$ Une autre d\'efinition possible est la suivante: si $\mathcal{U}$ est une composante de~$\mathcal{F}(f)$ invariante par $f,$ on dit que $\mathcal{U}$ est un domaine de rotation si $f_{\vert\mathcal{U}}$ est conjugu\'e \`a une rotation lin\'eaire; en dimension $1$ cela correspond \`a poss\'eder un disque de \textsc{Siegel}. On a la liste de propri\'et\'es suivantes (\cite{BK4}).
\begin{itemize}
\item Si $f$ pr\'eserve une forme volume lisse, alors toute composante de \textsc{Fatou} est un domaine de rotation.

\item Si $\mathcal{U}$ est un domaine de rotation, $\mathcal{G}$ est un sous-groupe de $\mathrm{Aut}(\mathcal{M}).$

\item Une composante de \textsc{Fatou} $\mathcal{U}$ est un domaine de rotation si et seulement s'il existe une sous-suite $(n_j)$ tendant vers $+\infty$ telle que $(f^{n_j})$ converge uniform\'ement vers l'identit\'e sur des sous-ensembles compacts de $\mathcal{U}.$

\item Si $\mathcal{U}$ est un domaine de rotation, $\mathcal{G}$ est un groupe de \textsc{Lie} ab\'elien, compact et l'action de~$\mathcal{G}$ sur $\mathcal{U}$ est analytique r\'eelle.
\end{itemize}

\noindent Soit $\mathcal{G}_0$ une composante connexe de l'identit\'e de $\mathcal{G}.$ Puisque $\mathcal{G}$ est un groupe de \textsc{Lie} ab\'elien, infini et compact,~$\mathcal{G}_0$ est un tore de dimension $d\geq 0;$ notons que $d\leq \dim_\mathbb{C}\mathcal{M}.$ On dit que $d$ est le {\it rang du domaine de rotation}\label{ind53}. Le rang est \'egal \`a la dimension de l'adh\'erence de l'orbite g\'en\'erique d'un point de $\mathcal{U}.$

\noindent On a des renseignements de nature g\'eom\'etrique sur les domaines de rotation: si $\mathcal{U}$ est un domaine de rotation, il est pseudo-convexe (\cite{BK4}). 

\noindent Donnons quelques pr\'ecisions lorsque $\mathcal{M}$ est une surface k\"{a}hl\'erienne poss\'edant un automorphisme d'entropie positive.

\begin{thm}[\cite{BK4}]
{\sl Soient $\mathcal{Z}$ une surface k\"{a}hl\'erienne compacte et $f$ un automorphisme sur $\mathcal{Z}$ d'entropie positive. Soit $\mathcal{U}$ un domaine de rotation de rang $d.$ Alors $d\leq 2.$

\noindent Si $d=2$ la $\mathcal{G}_0$-orbite d'un point g\'en\'erique de $\mathcal{U}$ est un $2$-tore r\'eel.

\noindent Si $d$ vaut $1,$ il existe un champ de vecteurs holomorphes induisant sur $\mathcal{Z}$ un feuilletage en surfaces de Riemann dont chaque feuille est invariante par $\mathcal{G}_0.$}
\end{thm}

\noindent On peut utiliser un argument de lin\'earisation locale pour montrer que certains points fixes appartiennent \`a l'ensemble de \textsc{Fatou}. R\'eciproquement lin\'eariser un point fixe de l'ensemble de \textsc{Fatou} est toujours possible.

\subsection{Ensemble de \textsc{Fatou} des automorphismes de \textsc{H\'enon}}

\noindent Soit $f$ un automorphisme de \textsc{H\'enon}. On d\'esigne par $\mathcal{K}^\pm$ le sous-ensemble de $\mathbb{C}^2$ dont l'orbite positive/n\'egative est born\'ee $$\mathcal{K}^\pm=\{(x,y)\in\mathbb{C}^2\,\vert \,\{f^{\pm n}(x,y)\,\vert\, n\geq 0\} \text{ est born\'e}\}.$$ Posons 
\begin{align*}
& \mathcal{K}=\mathcal{K}^+\cap \mathcal{K}^-, && \mathcal{J}^\pm=\partial \mathcal{K}^\pm, && \mathcal{J}=\mathcal{J}^+\cap \mathcal{J}^-, && \mathcal{U}^+=\mathbb{C}^2\setminus \mathcal{K}^+.
\end{align*}

\noindent \'Enon\c{c}ons quelques propri\'et\'es.
\begin{itemize}
\item Les it\'er\'es $f^n,$ $n\geq 0,$ forment une famille normale dans l'int\'erieur de $\mathcal{K}^+.$

\item Si $(x,y)$ appartient \`a $\mathcal{J}^+$ il n'existe pas de voisinage $U$ de $(x,y)$ sur lequel la famille $\{f_{\vert U}^n\,\vert\, n\geq 0\}$ soit normale.
\end{itemize}

\noindent On a l'\'enonc\'e suivant.

\begin{pro}
{\sl L'ensemble de \textsc{Fatou} d'une application de \textsc{H\'enon} est $\mathbb{C}^2\setminus \mathcal{J}^+.$}
\end{pro}

\begin{defis}
Soit $\Omega$ une composante de \textsc{Fatou}; $\Omega$ est dite {\it r\'ecurrente}\label{ind54} s'il existe un compact $C$ de $\Omega$ et un point $m$ de $C$ tels que $f^{n_j}(m)$ appartienne \`a $C$ pour une infinit\'e de $n_j\to +\infty.$ Une composante r\'ecurrente de \textsc{Fatou} est n\'ecessairement p\'eriodique.

\noindent Un point fixe $m$ de $f$ est un {\it foyer}\label{ind55} si $m$ appartient \`a l'int\'erieur de la vari\'et\'e stable $$\mathcal{W}^s(m)=\{p\,\vert\,\lim_{n\to +\infty} \text{dist}(f^n(m),f^n(p))=0 \}.$$ On dit que $\mathcal{W}^s(m)$ est le {\it bassin}\label{ind56} de $m.$ Si $m$ est un foyer, les valeurs propres de $Df_{(m)}$ sont de module inf\'erieur \`a $1.$ 

\noindent Un {\it disque de \textsc{Siegel}}\label{ind57} (resp. {\it anneau de \textsc{Herman}}\label{ind58}) est l'image d'un disque (resp. d'un anneau)~$\Delta$ par une application holomorphe injective $\varphi$ avec la propri\'et\'e suivante: pour tout $z$ dans $\Delta$ on~a 
\begin{align*}
& f\varphi(z)=\varphi(\alpha z), && \alpha=\mathrm{e}^{2\mathrm{i}\pi\theta},\,\theta\in\mathbb{R} \setminus\mathbb{Q}.
\end{align*} 
\end{defis}

\noindent On peut d\'ecrire les composantes de \textsc{Fatou} r\'ecurrentes d'un automorphisme de \textsc{H\'enon}.

\begin{thm}[\cite{BS}]
{\sl Soient $f$ un automorphisme de \textsc{H\'enon} de d\'eterminant jacobien strictement inf\'erieur \`a $1$ et $\Omega$ une composante de \textsc{Fatou} r\'ecurrente. Alors $\Omega$ est 
\begin{itemize}
\item ou bien le bassin d'un foyer;

\item ou bien un disque de \textsc{Siegel};

\item ou bien un anneau de \textsc{Herman}.
\end{itemize}}
\end{thm}

\noindent Sous certaines conditions les composantes de \textsc{Fatou} d'un automorphisme de \textsc{H\'enon} sont r\'ecurrentes.

\begin{pro}
{\sl Les composantes de \textsc{Fatou} d'une application de \textsc{H\'enon} pr\'eservant le volume sont p\'eriodiques et r\'ecurrentes.}
\end{pro}

\subsection{Ensemble de \textsc{Fatou} des automorphismes d'entropie positive sur les tores, (quotients de) K$3,$ surfaces rationnelles}

\noindent Si $\mathcal{Z}$ est un tore complexe, un automorphisme d'entropie positive est essentiellement un \'el\'ement de $\mathrm{GL}_2(\mathbb{Z});$ le fait que l'entropie soit positive implique que les valeurs propres satisfont: $\vert\lambda_1\vert<~1<~\vert\lambda_2\vert$ et l'ensemble de \textsc{Fatou} est vide. 

\noindent Supposons que $\mathcal{Z}$ soit une surface K$3$ ou un quotient d'une surface K$3.$ Puisqu'il existe une forme volume invariante les seules composantes de \textsc{Fatou} possibles sont les domaines de rotation. \textsc{McMullen} a montr\'e l'existence de surfaces K$3$ non alg\'ebriques avec des domaines de rotation de rang $2$ (\emph{voir} \cite{Mc2}); une autre r\'ef\'erence est la suivante \cite{Og}.  

\noindent Les autres surfaces compactes poss\'edant des automorphismes d'entropie positive sont des surfaces rationnelles; dans ce cas il peut y avoir des domaines de rotation de rang $1$ et $2$ (\emph{voir} \cite{BK2, Mc}). D'autres ph\'enom\`enes peuvent arriver tels des bassins attractifs et/ou r\'epulsifs (\cite{BK2, Mc}).

\chapter{Automorphismes d'entropie positive et groupes de \textsc{Weyl}}\label{mcmullen}

\noindent Le lien entre les groupes de \textsc{Weyl} et la g\'eom\'etrie birationnelle sur le plan projectif complexe a \'et\'e abord\'e d\`es $1895$ dans \cite{Ka} puis dans \cite{DV, Na, Na2, Co, Gi, Lo, Ha2, Manin, Ha, Ni, Ha3, DoOr, Hi, Zh, DZ}.

\section{Groupes de \textsc{Weyl}}\label{wweyl}

\subsection{Groupes de \textsc{Weyl}, I}

\noindent Soit $\mathcal{Z}$ une surface rationnelle. Consid\'erons la repr\'esentation 
\begin{align*}
& \mathrm{cr}\colon\mathrm{Aut}(\mathcal{Z})\to\mathrm{GL}(\mathrm{H}^2(\mathcal{Z},\mathbb{Z})),  && f\mapsto f_*;
\end{align*}

\noindent son image est un groupe d'isom\'etries pr\'eservant en particulier le produit d'intersection. Soit $\{e_0,\ldots,e_n\}$ une base de $\mathrm{H}^2(\mathcal{Z},\mathbb{Z});$ si 
\begin{align*}
&e_0\cdot e_0=1, && e_j\cdot e_j=-1,\,\,\forall\,\, 1\leq j\leq k, && e_i\cdot e_j=0,\,\,\forall\,\,0\leq i\not=j \leq n
\end{align*}  

\noindent on dit que $\{e_0,\ldots,e_n\}$ est une {\it base g\'eom\'etrique}\label{ind40}.

\begin{eg}
Les bases $\{\mathrm{H},\,\mathrm{E},\,\mathrm{F},\,\mathrm{G}\}$ introduites au \S\ref{sigma}, \S\ref{rho} et \S\ref{tau} sont des bases g\'eom\'etriques.

\noindent Plus g\'en\'eralement supposons que $\pi\colon\mathcal{Z}\to\mathbb{P}^2(\mathbb{C})$ soit obtenue en \'eclatant $n$ points distincts de~$\mathbb{P}^2(\mathbb{C}).$ Si~$\mathrm{E}_j=\pi^{-1}(p_j)$ d\'esigne la fibre exceptionnelle obtenue en \'eclatant $p_j$ et $\pi^*\mathrm{H}$ la classe d'une droite, alors $\{\pi^*\mathrm{H},\,\mathrm{E}_1,\,\ldots,\,\mathrm{E}_n\}$ est une base g\'eom\'etrique.
\end{eg}

\noindent Repla\c{c}ons-nous dans les conditions du Th\'eor\`eme \ref{nagata}. Soient $\mathcal{Z}$ une surface rationnelle et $f$ un automorphisme sur $\mathcal{Z}$ tel que $f_*$ soit d'ordre infini. Il existe une suite d'applications holomorphes $\pi_j\colon\mathcal{Z}_{j+1}\to \mathcal{Z}_j$ telles que $\mathcal{Z}_1=\mathbb{P}^2(\mathbb{C})$ et $\mathcal{Z}_{n+1}= \mathcal{Z}.$ Notons $\mathrm{E}_j=\pi_{j+1}^{-1}(p_j)\subset\mathcal{Z}_{j+1}$ le diviseur exceptionnel obtenu en \'eclatant $p_j$ et $e_j=\pi_{n+1}^*\ldots\pi_{j+1}^*(\mathrm{E}_j).$ Une {\it configuration exceptionnelle}\label{ind41} $\mathcal{E}$ est la donn\'ee des $e_j.$ Soit $\alpha$ un \'el\'ement de $\mathrm{H}^2(\mathcal{Z},\mathbb{Z})$ tel que $\alpha\cdot\alpha=-2,$ alors~$R_\alpha(x)=x+(x\cdot\alpha)\alpha$ envoie $\alpha$ sur~$-\alpha$ et $R_\alpha$ fixe tout \'el\'ement de $\alpha^\perp;$ 	autrement dit $R_\alpha$ est une r\'eflexion dans la direction $\alpha.$ 

\noindent Consid\'erons les vecteurs d\'efinis par 
\begin{align*}
&\alpha_0=e_0-e_1-e_2-e_3, && \alpha_j=e_{j+1}-e_j,\, 1\leq j\leq n-1.
\end{align*}

\noindent Pour tout $j$ dans $\{0,\ldots,n-1\}$ on a l'\'egalit\'e $\alpha_j\cdot\alpha_j=- 2.$ Pour $j$ non nul la r\'eflexion~$R_{\alpha_j}$ induit une permutation sur $\{e_j,\, e_{j+1}\}.$ Le sous-groupe engendr\'e par les $R_{\alpha_j},$ avec $1\leq j\leq~n-~1,$ co\"incide avec l'ensemble des permutations de $\{e_1,\,\ldots,\,e_n\}.$ Le groupe $\langle R_{\alpha_j}\,\vert\, 0\leq j\leq n-1\rangle$ est appel\'e {\it groupe de \textsc{Weyl}}\label{ind42} et not\'e $\mathrm{W}_n\subset\mathrm{O}(\mathbb{Z}^{1,n}).$

\noindent Les groupes de \textsc{Weyl} sont, pour $3\leq n\leq 8,$ isomorphes au groupes de \textsc{Coxeter} finis 
\begin{align*}
&A_1\times A_2, && A_4, && D_5, && E_6, && E_7, &&E_8
\end{align*}

\noindent et sont associ\'es aux surfaces de Del Pezzo\footnote{Une surface de Del Pezzo est isomorphe \`a $\mathbb{P}^2(\mathbb{C})$ ou \`a $\mathbb{P}^1(\mathbb{C})\times\mathbb{P}^1(\mathbb{C})$ ou encore \`a $\mathbb{P}^2(\mathbb{C})$ \'eclat\'e en $1\leq r\leq 8$ points en \og position g\'en\'erale\fg.}. Pour $k\geq 9$ les groupes de \textsc{Weyl} sont infinis et pour $k\geq 10$ ils contiennent des \'el\'ements dont le rayon spectral est strictement sup\'erieur \`a~$1.$

\noindent Donnons quelques propri\'et\'es de ce groupe. Si $f$ est un automorphisme de $\mathcal{Z},$ un r\'esultat de \textsc{Nagata} assure l'existence d'un unique \'el\'ement~$w$ de $\mathrm{W}_n$ tel que le diagramme $$\xymatrix{\mathbb{Z}^{1,n}\ar[d]_{\varphi}\ar[r]^{w} &\mathbb{Z}^{1,n}\ar[d]^{\varphi} \\
\mathrm{H}^2(\mathcal{Z},\mathbb{Z})\ar[r]^{f_*} &\mathrm{H}^2(\mathcal{Z},
\mathbb{Z})}$$ commute; on dit que $w$ est {\it r\'ealis\'e par} l'automorphisme $f.$ On a aussi l'\'enonc\'e suivant.

\begin{thm}[\cite{Dol}]
{\sl Soit $\mathcal{Z}$ une surface rationnelle qui domine $\mathbb{P}^2(\mathbb{C}).$ 
\begin{itemize}
\item Le groupe de \textsc{Weyl} $\mathrm{W}_k\subset\mathrm{GL}(\mathrm{Pic}(\mathcal{Z}))$ ne d\'epend pas de la configuration exceptionnelle choisie.

\item Si $\mathcal{E}$ et $\mathcal{E}'$ sont deux configurations exceptionnelles distinctes, il existe $w$ dans $\mathrm{W}_k$ tel que $w(\mathcal{E})=\mathcal{E}'.$

\item Si $\mathcal{Z}$ est obtenue en \'eclatant $k$ points g\'en\'eriques et si $\mathcal{E}$ est une configuration exceptionnelle, alors, pour tout~$w$ dans le groupe de \textsc{Weyl}, $w(\mathcal{E})$ est une configuration exceptionnelle.
\end{itemize}
}
\end{thm}

\noindent Un produit de g\'en\'erateurs $R_{\alpha_j}$ est un {\it \'el\'ement de \textsc{Coxeter}}\label{ind43} $w$ dans $\mathrm{W}_n.$ Notons que tous les \'el\'ements de \textsc{Coxeter} sont conjugu\'es donc le rayon spectral de $w$ est bien d\'efini.

\noindent L'involution de \textsc{Cremona} est repr\'esent\'ee par la r\'eflexion $\kappa_{ijk}=R_{\alpha_{ijk}}$ o\`u $\alpha_{ijk}=e_0-e_i-e_j-e_k$ pour $i,\,j,\,k\geq 1$ distincts; elle agit comme suit
\begin{align*}
& e_0\to 2 e_0-e_i-e_j-e_k, && e_i\to e_0-e_j-e_k, && e_j\to e_0-e_i-e_k
\end{align*}
\begin{align*}
& e_k\to e_0-e_i-e_j, && e_\ell\to e_\ell \text{ si }\ell\not\in\{0,\,i,\,j,\,k\} 
\end{align*}

\noindent Lorsque $n=3,$ on dira que $\kappa_{123}$ est l'{\it\'el\'ement standard}\label{ind44} de $\mathrm{W}_3.$ 

\noindent Soit $\pi_n=\kappa_{123}R_{\alpha_1}\ldots R_{\alpha_{n-1}}\in\Sigma_n\subset\mathrm{W}_n$ la permutation cyclique $(123\ldots n).$ Pour $n\geq 4,$ on appelle \'el\'ement standard du groupe de \textsc{Weyl} l'\'el\'ement $w$ de $\mathrm{W}_n$ d\'efini par $w=\pi_n\kappa_{123}.$ Il v\'erifie
\begin{align*}
& w(e_0)=2e_0-e_2-e_3-e_4, && w(e_1)=e_0-e_3-e_4, &&w(e_2)=e_0-e_2-e_4,
\end{align*}
\begin{align*}
& w(e_3)=e_0-e_2-e_3,&&w(e_j)=e_{j+1}, \,\, 4\leq j \leq n-2, && w(e_{n-1})=e_1.
\end{align*}

\noindent Comme nous le verrons au \S\ref{resmc} il existe une condition suffisante permettant de r\'ealiser des \'el\'ements du groupe de \textsc{Weyl} par des automorphismes des surfaces rationnelles obtenues en \'eclatant $k$ points distincts de $\mathbb{P}^2(\mathbb{C})$ le long d'une cubique. 

\subsection{Groupes de \textsc{Weyl}, II} Une autre fa\c{c}on de voir le groupe de \textsc{Weyl} est la suivante. Si $n\geq 3$ on consid\`ere le graphe $\Gamma_n$ suivant 

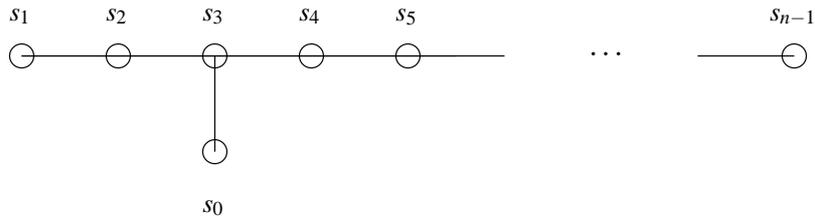
\begin{figure}[H]
\begin{center}
\input{coxeter.pstex_t}
\end{center}
\end{figure}

\noindent dont l'ensemble des sommets est $S_n=\{s_0,\ldots,s_{n-1}\}.$ Lorsque $n=3,$ $\Gamma_3$ poss\`ede un seul c\^ot\'e, joignant $s_1$ \`a $s_2.$ Soit $M=(m_{ij})$ la matrice de $\mathcal{M}_n(\mathbb{C})$ d\'efinie par 
$$m_{ij}=\left\{\begin{array}{lll} 1\text{ si }i=j\\ 3 \text{ si $s_i$ et $s_j$ sont \og voisins\fg\, dans $\Gamma_n$}\\ 2 \text{ sinon}\end{array}\right.$$ Le groupe de \textsc{Weyl} est donn\'e par $$\mathrm{W}_n=\langle s_0,\,\ldots,\,s_{n-1}\,\vert\, (s_is_j)^{m_{ij}}=1\rangle.$$ Consid\'erons $V_n=\mathbb{R}^{S_n}$ muni de la base $\{\alpha_0,\,\ldots,\,\alpha_{n-1}\},$ base duale de $\{s_0,\,\ldots,\,s_{n-1}\},$ et du produit int\'erieur $$B_n(\alpha_j,\alpha_j)=-2\cos(\pi/ m_{ij}).$$ Tout \'el\'ement $\alpha$ de $V_n$ tel que $B(\alpha,\alpha)=\pm 2$ d\'etermine une r\'eflexion 
\begin{equation}\label{refl}
R_\alpha(x)=x-\frac{2B(x,\alpha)}{B(\alpha,\alpha)}\alpha
\end{equation}

\noindent dans le groupe orthogonal $(V_n,B_n).$ L'homomorphisme 
\begin{align*}
& \mathrm{W}_n\to(V_n,B_n), &&s_i\mapsto R_{\alpha_i}
\end{align*}

\noindent d\'efinit l'action g\'eom\'etrique de $\mathrm{W}_n$ sur $V_n.$ 

\noindent Les $\alpha_i$ sont les {\it racines simples}\label{ind700} de $\mathrm{W}_n.$ Les  {\it racines}\label{ind7000} de $\mathrm{W}_n$ sont les orbites des racines simples $$\Theta_n=\bigcup\mathrm{W}_n\cdot\alpha_i.$$ Notons $L_n$ le r\'eseau $\oplus\mathbb{Z}\alpha_i.$ L'\'egalit\'e (\ref{refl}) assure que $L_n$ est invariant sous l'action de $\mathrm{W}_n.$  Un vecteur $v=\sum c_i\alpha_i$ de $V_n$ est {\it positif}\label{ind701} si $c_i\geq 0$ pour tout~$i.$ On d\'esigne par $\Theta_n^+$ l'ensemble des racines positives. Toute racine de $\mathrm{W}_n$ est positive ou n\'egative, {\it i.e.} $\Theta_n=\Theta_n^+\cup(-\Theta_n^+)$ (\emph{voir} \cite{Bo, Hum}). Un produit du type $$s_{\varsigma(0)} s_{\varsigma(1)}\ldots s_{\varsigma(n-1)},$$ o\`u $\varsigma$ d\'esigne une permutation de $\{0,\,\ldots,\, n-1\},$ est un \'el\'ement de \textsc{Coxeter}. Puisque tous les \'el\'ements de $\mathrm{W}_n$ sont conjugu\'es, ils ont tous le m\^eme ordre $h_n.$ On a $h_n=6,\,5,\,8,\,12,\,18,\,30$ pour $n=3,\,4,\,5,\,6,\,7,\,8$ et $h_n=\infty$ pour $n\geq 9.$

\medskip

\noindent Consid\'erons $A(\Gamma_n)=2\mathrm{Id}-B_n$ la matrice adjointe de $\Gamma_n$ qui peut \^etre vu comme un op\'erateur sur $V_n.$ On a $$A(\Gamma_n)_{_{ij}}=\left\{\begin{array}{ll} 1\text{ si $s_i$ et $s_j$ sont \og voisins\fg\, dans $\Gamma_n$}\\ 0\text{ sinon}\end{array}\right.$$

\begin{pro}[\cite{Mc}]\label{rayonspectral}
{\sl Le rayon spectral $\lambda_n=\lambda(A(\Gamma_n))$ croit strictement avec $n$ et~$\lambda_9=2.$}
\end{pro}

\noindent Le graphe du groupe de \textsc{Weyl} satisfait la propri\'et\'e suivante: toute ar\^ete joint un sommet impair \`a un sommet pair d'o\`u la d\'ecomposition 
\begin{align*}
& V_n=V_n^0\oplus V_n^1, && V_n^0=\langle\alpha_i\,\vert\, i\text{ pair}\rangle, && V_n^1=\langle\alpha_i\,\vert\, i\text{ impair}\rangle
\end{align*}

\noindent et $A(\Gamma_n)=\left[\begin{array}{cc} 0 & \transp C_n\\ C_n & 0\end{array}\right].$ Consid\'erons l'\'el\'ement de \textsc{Coxeter} $w_n$ d\'efini par $$w_n=(s_0s_2s_4\ldots )\cdot(s_1s_3s_5 \ldots)=w_n^0\cdot w_n^1.$$ Notons que, puisque $s_is_{i+2}=s_{i+2}s_i,$ l'ordre des \'el\'ements apparaissant dans l'\'ecriture de~$w_n^0$ (resp. $w_n^1$) n'a pas d'importance. \`A partir de $B_n= 2\mathrm{Id}-A(\Gamma_n)$ on obtient 
\begin{align*}
&w_n^0=\left[\begin{array}{cc} -\mathrm{Id}&\transp C_n\\0&\mathrm{Id}\end{array}\right], && w_n^1=\left[\begin{array}{cc}\mathrm{Id}&0\\C_n& -\mathrm{Id}\end{array}\right]
\end{align*}

\noindent puis $$w_n=\left[\begin{array}{cc}\transp C_n C_n -\mathrm{Id} & -\transp C_n\\ C_n& -\mathrm{Id}\end{array}\right].$$ Le th\'eor\`eme de Perron-Frobenius assure l'existence d'un vecteur positif $v_n$ dans $V_n,$ unique modulo multiplication par un scalaire, tel que $A(\Gamma_n) \cdot v_n=\lambda_nv_n.$ Consid\'erons $G_n\subset V_n$ le sous-espace vectoriel de dimension $2$ engendr\'e par les parties paire et impaire de $v_n=v_n^0+~v_n^1.$ L'\'egalit\'e $A(\Gamma_n) \cdot v_n=\lambda_n v_n$ implique que $(C_n\cdot v_n^0,\transp C_n\cdot v_n^1)=\lambda_n(v_n^1,v_n^0).$ Par suite $G_n$ est invariant par $w_n;$ plus pr\'ecis\'ement $$w_{n_{\vert G_n}}=\left[\begin{array}{cc} \lambda_n^2-1 & -\lambda_n\\ \lambda_n & -1\end{array}\right].$$ Cette \'egalit\'e associ\'ee \`a la Proposition \ref{rayonspectral} permet d'\'enoncer le r\'esultat suivant.

\begin{thm}[\cite{Mc}]
{\sl L'application lin\'eaire $w_{n_{\vert G_n}}$ est 
\begin{itemize}
\item elliptique, d'ordre $h_n,$ pour $n\leq 8,$

\item parabolique, d'ordre infini, pour $n=9,$

\item hyperbolique, d'ordre infini, pour $n\geq 10.$
\end{itemize}}
\end{thm}

\begin{rem}
Dans le cas elliptique $w_{n_{\vert G_n}}$ est une rotation d'angle $2\pi/h_n$ (\emph{voir} \cite{Hum}, \S 3.7).
\end{rem}

\begin{thm}[\cite{Mc}]\label{proj}
{\sl Pour $n\not=9,$ toute racine $\alpha$ de $\Theta_n$ a une projection orthogonale non triviale sur $G_n\subset V_n.$}
\end{thm}

\begin{proof}[{\sl D\'emonstration}]
Supposons que $\alpha$ appartienne \`a $\Theta_n^+.$ Soit $\beta$ une projection de $\alpha$ sur $G_n.$ Puisque~$\lambda_n\not=~2$ pour $n\not=9$ et puisque $\alpha=\sum x_i\alpha_i,$ $v_n=\sum y_i\alpha_i$ sont des vecteurs positifs on~a: $$B(v_n,\beta)=B(v_n,\alpha)=(2-\lambda_n)\sum x_iy_i \not=0;$$ par suite $\beta\not=0.$
\end{proof}

\begin{cor}[\cite{Mc}]
{\sl Soit $w$ un \'el\'ement de \textsc{Coxeter} de $\mathrm{W}_n.$

\noindent Si $n<9,$ toute orbite de $w_{\vert\Theta_n}$ est form\'ee de $h_n$ \'el\'ements.

\noindent Si $n>9,$ toute orbite de $w_{\vert\Theta_n}$ est infinie, {\it i.e.} $w$ n'a pas de racine p\'eriodique.}
\end{cor}

\subsection{\'El\'ements de \textsc{Coxeter} et nombres de \textsc{Salem}}

\noindent Il existe un lien entre les nombres de \textsc{Salem} et les \'el\'ements de \textsc{Coxeter}: les premiers apparaissent comme valeurs propres des seconds. Le polyn\^ome caract\'eristique d'un \'el\'ement de \textsc{Coxeter} $w$ de $\mathrm{W}_n$ est donn\'e par $$P_n(t)=\det(t\mathrm{Id}-w) =\frac{t^{n-2}(t^3-t-1)+(t^3+t^2-1)}{t-1}.$$

\noindent Pour $n$ distinct de $9,$ le polyn\^ome $P_n$ a des racines simples et pour $n\geq 10$ le polyn\^ome $P_n$ s'\'ecrit comme le produit de deux polyn\^omes $Q_n$ et $R_n$ o\`u $Q_n$ d\'esigne un polyn\^ome de \textsc{Salem} et $R_n$ un produit de polyn\^omes cyclotomiques. Les racines $t=\lambda_n^{\pm 1}$ de $Q_n$ sont des valeurs propres de $w_{n\vert G_n}.$

\begin{rem}
Le polyn\^ome $Q_{10}$ co\"incide avec le polyn\^ome de \textsc{Lehmer} d'o\`u $\lambda_{10}= \lambda_{\text{Lehmer}}.$
\end{rem}

\noindent \'Etendons la factorisation de $P_n$ de la fa\c{c}on suivante:
\begin{itemize}
\item on d\'efinit $Q_8$ comme \'etant le polyn\^ome cyclotomique des racines $h_8$i\`eme de l'unit\'e;

\item on pose $Q_9(t)=t-1.$
\end{itemize}

\noindent On dit que $\lambda$ est une {\it valeur propre dominante}\label{ind702} de $w$ si $Q_n(\lambda)=0;$ autrement dit les valeurs propres dominantes sont les valeurs propres de $w_{n\vert G_n}$ et leur conjugu\'es de \textsc{Galois}. Les vecteurs propres associ\'es \`a ces valeurs propres sont les {\it vecteurs propres dominants}\label{ind703}; ils v\'erifient la propri\'et\'e suivante.

\begin{thm}[\cite{Mc}]\label{racine}
{\sl Soient $w$ un \'el\'ement de \textsc{Coxeter} de $\mathrm{W}_n$ et $v\in L_n\otimes\mathbb{C}$ un vecteur propre dominant. Si $n$ est distinct de $9,$ on a $v\cdot\alpha\not=0$ pour toute racine $\alpha$ de~$\Theta_n.$}
\end{thm}

\begin{proof}[{\sl D\'emonstration}]
Il suffit de prouver l'\'enonc\'e lorsque $w(v)=\lambda_n^{\pm 1}v$ auquel cas $v$ appartient \`a $G_n\otimes\mathbb{C}.$ Puisque $n$ est distinct de $9,$ le sous-espace vectoriel $G_n$ est engendr\'e par $v$ et ses conjugu\'es de \textsc{Galois}. Si~$v\cdot\alpha=0,$ alors, pour tout conjugu\'e de \textsc{Galois} $v'$ de $v,$ on a l'\'egali\-t\'e~$v'\cdot\alpha=0$ et la projection de $\alpha$ sur $G_n$ est nulle: contradiction avec le Th\'eor\`eme \ref{proj}.
\end{proof}

\section{Cubiques marqu\'ees, \'eclatements marqu\'es, paires marqu\'ees}

\subsection{Cubiques marqu\'ees}

\noindent Une {\it cubique}\label{ind800} $\mathcal{C}\subset\mathbb{P}^2(\mathbb{C})$ est une courbe r\'eduite de degr\'e $3.$ Elle peut \^etre singuli\`ere ou r\'eductible; on d\'esigne par $\mathcal{C}^*$ la partie lisse. Commen\c{c}ons par quelques rappels sur le groupe de \textsc{Picard} d'une telle courbe; pour plus de d\'etails on renvoie \`a \cite{HM}. On a la suite exacte suivante $$0\longrightarrow\mathrm{Pic}_0(\mathcal{C})\longrightarrow \mathrm{Pic}(\mathcal{C})\longrightarrow\mathrm{H}^2(\mathcal{C},\mathbb{Z}) \longrightarrow 0$$ o\`u $\mathrm{Pic}_0(\mathcal{C})$ est isomorphe
\begin{itemize}
\item soit \`a un tore $\mathbb{C}/\Lambda$ (quand $\mathcal{C}$ est lisse);

\item soit au groupe multiplicatif $\mathbb{C}^*$ (ceci correspond au cas o\`u $\mathcal{C}$ est une cubique nodale ou l'union d'une cubique et d'une droite transverse \`a celle-ci ou l'union de trois droites en position g\'en\'erale); 

\item soit au groupe additif $\mathbb{C}$ (lorsque $\mathcal{C}$ est une cubique cuspidale ou l'union d'une conique et d'une droite tangente \`a celle-ci ou l'union de trois droites concourantes). 
\end{itemize}

\noindent Une {\it cubique marqu\'ee}\label{ind801} est un couple $(\mathcal{C},\eta)$ form\'e d'une
courbe abstraite $\mathcal{C}$ et d'un homomorphisme $\eta\colon\mathbb{Z}^{1,n}\to\mathrm{Pic} (\mathcal{C})$ tel que 
\begin{itemize}
\item les sections du fibr\'e en droites $\eta(e_0)$ proviennent du plongement de $\mathcal{C}$ dans $\mathbb{P}^2(\mathbb{C});$ 

\item il existe des points bases distincts $p_i$ sur $\mathcal{C}^*$ pour lesquels $\eta(e_i)=[p_i]$ pour $i=2,$ $\ldots,$ $n.$
\end{itemize}

\noindent Les points bases $p_i$ sont uniquement d\'etermin\'es par $\eta$ puisque $\mathbb{C}^*$ se plonge dans $\mathrm{Pic}(\mathcal{C}).$ R\'eciproquement une cubique $\mathcal{C}$ qui se plonge dans $\mathbb{P}^2(\mathbb{C})$ et une collection de points distincts sur $\mathcal{C}^*$ d\'eterminent un marquage de $\mathcal{C}.$

\begin{rem}
Diff\'erents marquages de $\mathcal{C}$ peuvent conduire \`a diff\'erents plongements de $\mathcal{C}$ dans $\mathbb{P}^2(\mathbb{C})$ mais tous ces plongements sont \'equivalents via l'action de $\mathrm{Aut}(\mathcal{C}).$ 
\end{rem}

\noindent Soient $(\mathcal{C},\eta)$ et $(\mathcal{C}',\eta')$ deux cubiques marqu\'ees; un {\it isomorphisme}\label{ind802} entre $(\mathcal{C},\eta)$ et $(\mathcal{C}',\eta')$ est une application biholomorphe $f\colon\mathcal{C}\to\mathcal{C}'$ telle que $\eta'=f_*\circ\eta.$

\noindent \'Etant donn\'ee une cubique marqu\'ee $(\mathcal{C},\eta)$ on pose $$W(\mathcal{C},\eta)=\{w\in\mathrm{W}_n\,\vert\,(\mathcal{C},\eta w)\text{ est une cubique marqu\'ee}\},$$ $$\mathrm{Aut}(\mathcal{C},\eta)=\{w\in W(\mathcal{C},\eta)\,\vert \,(\mathcal{C},\eta) \text{ et } (\mathcal{C}',\eta') \text{ sont isomorphes }\}.$$ 

\noindent On peut d\'ecomposer le marquage $\eta$ de $\mathcal{C}$ en deux parties 
\begin{align*}
& \eta_0\colon\ker(\deg\circ\eta)\to\mathrm{Pic}_0(\mathcal{C}), && \deg\circ\eta\colon \mathbb{Z}^{1,n}\to\mathrm{H}^2(\mathcal{C},\mathbb{Z}).
\end{align*}

\noindent On a la propri\'et\'e suivante.

\begin{thm}[\cite{Mc}]
{\sl Soit $(\mathcal{C},\eta)$ une cubique marqu\'ee. Les applications $\eta_0$ et $\deg\circ\eta$ d\'eterminent $(\mathcal{C},\eta)$ \`a isomorphisme pr\`es.}
\end{thm}

\noindent Une cons\'equence de cet \'enonc\'e est la suivante.

\begin{cor}[\cite{Mc}]\label{cubcub}
{\sl Une cubique irr\'eductible marqu\'ee $(\mathcal{C},\eta)$ est d\'etermin\'ee, \`a isomorphisme pr\`es, par $\eta_0\colon L_n\to\mathrm{Pic}_0(\mathcal{C}).$}
\end{cor}

\subsection{\'Eclatements marqu\'es}

\noindent Un {\it \'eclatement marqu\'e}\label{ind666} $(\mathcal{Z},\Phi)$ est la donn\'ee d'une surface projective lisse $\mathcal{Z}$ et d'un isomorphisme $\Phi\colon\mathbb{Z}^{1,n}\to \mathrm{H}^2(\mathcal{Z},\mathbb{Z})$ tels que 
\begin{itemize}
\item $\Phi$ envoie le produit de \textsc{Minkowski} $(x\cdot x)=x^2=x_0^2-x_1^2-\ldots-x_n^2$ sur le produit d'intersection sur $\mathrm{H}^2(\mathcal{Z},\mathbb{Z});$

\item il existe un morphisme birationnel $\pi\colon\mathcal{Z}\to\mathbb{P}^2(\mathbb{C})$ pr\'esentant $\mathcal{Z}$ comme l'\'eclat\'e de~$\mathbb{P}^2(\mathbb{C})$ en $n$ points distincts $p_1,$ $\ldots,$ $p_n;$

\item $\Phi(e_0)=[\mathrm{H}]$ et $\Phi(e_i)=[\mathrm{E}_i]$ pour $i=1,\ldots,n$ o\`u $\mathrm{H}$ est la pr\'eimage d'une droite g\'en\'erique de $\mathbb{P}^2(\mathbb{C})$ et $\mathrm{E}_i$ le diviseur obtenu en \'eclatant $p_i.$
\end{itemize}

\noindent Le marquage d\'etermine le morphisme $\pi\colon\mathcal{Z}\to\mathbb{P}^2(\mathbb{C})$ modulo l'action d'un automorphisme de $\mathbb{P}^2(\mathbb{C}).$

\noindent Soient $(\mathcal{Z},\Phi)$ et $(\mathcal{Z}',\Phi)$ deux \'eclatements marqu\'es; un {\it isomorphisme}\label{ind666a} entre $(\mathcal{Z},\Phi)$ et $(\mathcal{Z}',\Phi')$ est la donn\'ee d'une application biholomorphe $F\colon\mathcal{Z}\to\mathcal{Z}'$ tel que le diagramme suivant $$\xymatrix{& \mathbb{Z}^{1,n}\ar[dl]_{\Phi}\ar[dr]^{\Phi'} &\\
\mathrm{H}^2(\mathcal{Z},\mathbb{Z})\ar[rr]_{F_*} & &\mathrm{H}^2(\mathcal{Z}',\mathbb{Z})}$$ commute. Si $(\mathcal{Z},\Phi)$ et $(\mathcal{Z}',\Phi')$ sont isomorphes, il existe un automorphisme $\varphi$ de $\mathbb{P}^2(\mathbb{C})$ tel que $p'_i=~\varphi(p_i).$


\noindent Supposons qu'il existe deux morphismes birationnels $\pi,$ $\pi'\colon\mathcal{Z}\to \mathbb{P}^2(\mathbb{C})$ tels que $\mathcal{Z}$ soit la surface obtenue en \'eclatant $\mathbb{P}^2(\mathbb{C})$ en $p_1,$ $\ldots,$ $p_n$ (resp. $p'_1,$ $\ldots,$ $p'_n$) via $\pi$ (resp. $\pi'$). Il existe une transformation birationnelle $f\colon\mathbb{P}^2(\mathbb{C}) \dashrightarrow\mathbb{P}^2(\mathbb{C})$ telle que le diagramme $$\xymatrix{& \mathcal{Z}\ar[dl]_{\pi}\ar[dr]^{\pi'} &\\
\mathbb{P}^2(\mathbb{C})\ar@{-->}[rr]_f & &\mathbb{P}^2(\mathbb{C})}$$ commute; de plus il existe un unique \'el\'ement $w$ dans $\mathbb{Z}^{1,n}$ tel que $\Phi'=\Phi w.$

\noindent Le groupe de \textsc{Weyl} satisfait la propri\'et\'e suivante d\^ue \`a \textsc{Nagata}: soient $(\mathcal{Z},\Phi)$ un \'eclatement marqu\'e et $w$ un \'el\'ement de $\mathbb{Z}^{1,n}.$ Si $(\mathcal{Z},\Phi w)$ est encore un \'eclatement marqu\'e, alors~$w$ appartient au groupe de \textsc{Weyl} $\mathrm{W}_n.$ \'Etant donn\'e un \'eclatement marqu\'e $(\mathcal{Z},\Phi)$ on note~$W(\mathcal{Z},\Phi)$ l'ensemble des \'el\'ements $w$ de $\mathrm{W}_n$ tels que $(\mathcal{Z},\Phi w)$ soit un \'eclatement marqu\'e: $$W(\mathcal{Z},\Phi)=\{w\in\mathrm{W}_n\,\vert\,(\mathcal{Z}, \Phi w)\text{ est un \'eclatement marqu\'e}\}.$$ L'action \`a droite du groupe sym\'etrique consiste \`a r\'eordonner les points base de la suite d'\'eclatements donc le groupe des permutations est contenu dans $W(\mathcal{Z},\Phi).$ L'\'enonc\'e qui suit donne d'autres exemples d'\'el\'ements de $W(\mathcal{Z},\Phi).$ 

\begin{thm}[\cite{Mc}]\label{eclmar}
{\sl Soient $(\mathcal{Z},\Phi)$ un \'eclatement marqu\'e et $\sigma$ l'involution de \textsc{Cremona}. D\'esignons par $p_1,$ $\ldots,$ $p_n$ les points base de $(\mathcal{Z},\Phi).$ Si pour $4\leq k\leq n,$ le point~$p_k$ n'appartient pas \`a la droite passant par $p_i$ et $p_j,$ o\`u $1\leq i,j\leq 3,$ $i\not=j,$ alors $(\mathcal{Z},\Phi\kappa_{123})$ est un \'eclatement marqu\'e.}
\end{thm}

\begin{proof}[{\sl D\'emonstration}]
Soient $\pi\colon\mathcal{Z}\to\mathbb{P}^2(\mathbb{C})$ le morphisme birationnel associ\'e \`a l'\'eclatement marqu\'e $(\mathcal{Z},\Phi).$ D\'esignons par $q_1,$ $q_2$ et $q_3$ les points d'ind\'etermination de $\sigma.$ Choisissons des coordonn\'ees dans lesquelles $p_i=q_i$ pour $i=1,$ $2,$ $3;$ alors $\pi'=\sigma\pi\colon \mathcal{Z}\to\mathbb{P}^2(\mathbb{C})$ est un morphisme birationnel permettant de voir $(\mathcal{Z},\Phi\kappa_{123})$ comme un \'eclatement marqu\'e avec points base $p_1,$ $p_2,$ $p_3$ et $\sigma(p_i)$ pour $i\geq 4.$ Ces points sont distincts puisque, par hypoth\`e\-se,~$p_4,$ $\ldots,$ $p_n$ n'appartiennent pas aux droites contract\'ees par $\sigma.$ 
\end{proof}

\noindent Une racine $\alpha$ de $\Theta_n$ est une {\it racine nodale}\label{ind805} pour $(\mathcal{Z},\Phi)$ si $\Phi(\alpha)$ est repr\'esent\'e par un diviseur effectif $D.$ Dans ce cas $D$ se projette sur une courbe de degr\'e $d>0$ sur $\mathbb{P}^2(\mathbb{C});$ par sui\-te~$\alpha=de_0-\sum_{i\geq 1}m_ie_i$ est une racine positive. Une racine nodale est dite {\it g\'eom\'etrique}\label{ind806} si on peut \'ecrire $D$ comme une somme de courbes rationnelles lisses.

\begin{thm}[\cite{Mc}]\label{racnod}
{\sl Soit $(\mathcal{Z},\Phi)$ un \'eclatement marqu\'e. Si trois des points bases sont colin\'eaires, $(\mathcal{Z},\Phi)$ poss\`ede une racine nodale g\'eom\'etrique.}
\end{thm}

\begin{proof}[{\sl D\'emonstration}]
Quitte \`a r\'eordonner les points bases $p_1,$ $\ldots,$ $p_n,$ on peut supposer que $p_1,$ $p_2$ et $p_3$ sont colin\'eaires; notons $L$ la droite passant par ces trois points. \`A r\'eindexation pr\`es on se ram\`ene \`a: les points bases appartenant \`a $L$ sont $p_1,$ $\ldots,$ $p_k.$ La transform\'ee stricte~$\widetilde{L}$ de~$L$ induit une courbe rationnelle lisse sur $\mathcal{Z}$ avec $[\widetilde{L}]=[\mathrm{H}-\sum_{i=1}^k\mathrm{E}_i]$ d'o\`u $$\Phi(\alpha_{123})=[ \widetilde{L}+\sum_{i=1}^k\mathrm{E}_i].$$
\end{proof}

\begin{thm}[\cite{Mc}]\label{egal}
{\sl Soit $(\mathcal{Z},\Phi)$ un \'eclatement marqu\'e. Si $(\mathcal{Z},\Phi)$ n'a pas de racine g\'eom\'etrique nodale, alors $$W(\mathcal{Z}, \Phi)= \mathrm{W}_n.$$}
\end{thm}

\begin{proof}[{\sl D\'emonstration}]
Commen\c{c}ons par remarquer que si $(\mathcal{Z},\Phi)$ n'a pas de racine nodale g\'eom\'etrique et si $w$ appartient \`a $W(\mathcal{Z},\Phi),$ alors $(\mathcal{Z},\Phi w)$ ne poss\`ede pas de racine nodale g\'eom\'etrique. Il suffit donc de d\'emontrer que les g\'en\'erateurs de $\mathrm{W}_n$ appartiennent \`a $W(\mathcal{Z},\Phi).$ Pour les transpositions cela d\'ecoule de l'inclusion du groupe des permutations dans $W(\mathcal{Z},\Phi);$ pour~$\kappa_{123}$ c'est une cons\'equence des Th\'eor\`emes \ref{eclmar} et \ref{racnod}.
\end{proof}

\begin{cor}[\cite{Mc}]
{\sl Une surface marqu\'ee poss\`ede une racine nodale si et seulement si elle poss\`ede une racine nodale g\'eom\'etrique.}
\end{cor}

\subsection{Paires marqu\'ees}

\subsubsection{Premi\`eres d\'efinitions}

\noindent Soit $(\mathcal{Z},\Phi)$ un \'eclatement marqu\'e. Rappelons qu'une {\it courbe anticanonique} est une courbe r\'eduite $Y\subset\mathcal{Z}$ telle que la classe de $Y$ dans $\mathrm{H}^2(\mathcal{Z},\mathbb{Z})$ v\'erifie
\begin{equation}\label{antican}
[Y]=[3\mathrm{H}-\sum_i\mathrm{E}_i]=-\mathrm{K}_\mathcal{Z}.
\end{equation}

\noindent Une {\it paire marqu\'ee}\label{ind820} $(\mathcal{Z},\Phi,Y)$ est la donn\'ee d'un \'eclatement marqu\'e $(\mathcal{Z},\Phi)$ et d'une courbe anticanonique $Y.$ Un {\it isomorphisme}\label{ind821} entre les paires marqu\'ees $(\mathcal{Z},\Phi,Y)$ et $(\mathcal{Z}',\Phi',Y')$ est un biholomorphisme $f$ de $\mathcal{Z}$ dans $\mathcal{Z}',$ compatible avec les marquages et qui envoie $Y$ sur $Y'.$ Si~$n\geq 10,$ alors $\mathcal{Z}$ contient au plus une courbe anticanonique irr\'eductible; en effet, si une telle courbe $Y$ existe, alors $Y^2=9-n<0.$

\subsubsection{Des surfaces aux cubiques}\label{surfacescubiques}

\noindent Consid\'erons une paire marqu\'ee $(\mathcal{Z},\Phi,Y).$ Soit $\pi$ une projection de $\mathcal{Z}$ sur $\mathbb{P}^2(\mathbb{C})$ compatible avec $\Phi.$ L'\'egalit\'e (\ref{antican}) implique que $\mathcal{C}=\pi(Y)$ est une cubique passant par chaque point base $p_i$ avec multiplicit\'e $1.$ De plus, le fait que~$\mathrm{E}_i\cdot Y=1$ assure que $\pi\colon Y\to\mathcal{C}$ est un isomorphisme. L'identification de $\mathrm{H}^2(\mathcal{Z}, \mathbb{Z})$ et $\mathrm{Pic}(\mathcal{Z})$ permet d'obtenir le marquage naturel $$\eta\colon \mathbb{Z}^{1,n}\stackrel{\Phi}{\longrightarrow}\mathrm{H}^2(\mathcal{Z}, \mathbb{Z})=\mathrm{Pic} (\mathcal{Z})\stackrel{r}{\longrightarrow}\mathrm{Pic}(Y)\stackrel{\pi_*}{\longrightarrow}\mathrm{Pic}(\mathcal{C})$$ o\`u $r$ est l'application de restriction $\mathrm{Pic}(\mathcal{Z}) \stackrel{r}{\longrightarrow}\mathrm{Pic}(Y).$ Ainsi une paire marqu\'ee $(\mathcal{Z},Y,\Phi)$ d\'etermine canoniquement une cubique marqu\'ee $(\mathcal{C},\eta).$

\subsubsection{Des cubiques aux surfaces}\label{cubiquessurfaces}

\noindent R\'eciproquement consid\'erons une cubique mar\-qu\'ee~$(\mathcal{C},\eta).$ On dispose de points bases $p_i\in\mathcal{C}$ d\'etermin\'es par $(\eta(e_i))_{1\leq i\leq n}$ et un plongement de $\mathcal{C}$ dans~$\mathbb{P}^2(\mathbb{C})$ d\'etermin\'e par $\eta(e_0).$ Soient $(\mathcal{Z},\Phi)$ l'\'eclatement marqu\'e avec points bases~$p_i$ et $Y\subset\mathcal{Z}$ la transform\'ee stricte de $\mathcal{C}.$ On obtient ainsi une paire marqu\'ee $(\mathcal{Z},\Phi,Y)$ qu'on appelle \'eclatement de $(\mathcal{C},\eta)$ et que l'on note $\mathrm{Bl}(\mathcal{C},\eta).$ 

\smallskip

\noindent Les constructions pr\'esent\'ees au \S\ref{surfacescubiques} et \S\ref{cubiquessurfaces} sont inverses l'une de l'autre. En d'autres termes on a l'\'enonc\'e suivant.

\begin{thm}[\cite{Mc}]\label{foncteur}
{\sl Une paire marqu\'ee d\'etermine canoniquement une cubique marqu\'ee et r\'eciproquement.}
\end{thm}

\section{R\'esultats et id\'ee de d\'emonstration}\label{resmc}

\subsubsection{\'Enonc\'es}

\noindent Dans \cite{Mc} \textsc{McMullen} produit des exemples d'automorphismes d'entropie positive \`a l'aide d'\'el\'ements du groupe de \textsc{Weyl}. 

\begin{thm}[\cite{Mc}]
{\sl Pour $n\geq 10,$ tout \'el\'ement de \textsc{Coxeter} de $\mathrm{W}_n$ peut \^etre r\'ealis\'e par un automorphisme $f_n$ sur une surface rationnelle d'entropie positive $\log(\lambda_n).$}
\end{thm}

\noindent Plus pr\'ecis\'ement l'automorphisme $f_n\colon\mathcal{Z}_n\to\mathcal{Z}_n$ peut \^etre choisi de sorte 
\begin{itemize}
\item que $\mathcal{Z}_n$ soit le plan projectif complexe \'eclat\'e en $n$ points distincts $p_1,$ $\ldots,$ $p_n$ situ\'es sur une cubique cuspidale $\mathcal{C},$ 

\item qu'il existe une $2$-forme m\'eromorphe $\eta$ sur $\mathcal{Z}_n$ \`a p\^oles simples le long de la transform\'ee propre de $\mathcal{C}$ qui ne s'annule pas,

\item que $f_n^*(\eta)=\lambda_n\cdot\eta,$

\item que $(\langle f_n\rangle,\mathcal{Z}_n)$ soit minimal au sens de \textsc{Manin}\footnote{Soient $\mathcal{Z}$ une surface et $G$ un sous-groupe de $\mathrm{Aut}(\mathcal{Z}).$ Une application birationnelle $f\colon\mathcal{Z}\dashrightarrow\widetilde{\mathcal{Z}}$ est dite $G$-\'equivariante si $\widetilde{G}=f Gf^{-1}\subset \mathrm{Aut}(\widetilde{\mathcal{Z}}).$ Le couple $(G,\mathcal{Z})$ est minimal si tout morphisme birationnel $G$-\'equivariant est un isomorphisme.}.
\end{itemize}

\noindent Les trois premi\`eres propri\'et\'es d\'eterminent $f_n$ de mani\`ere unique. Les points $p_i$ admettent une description simple en termes de vecteurs propres ce qui permet de donner des formules pour~$f_n.$

\noindent Notons que le plus petit nombre de \textsc{Salem} connu est la racine $\lambda_{\text{Lehmer}}\sim 1.17628081$ du polyn\^ome de \textsc{Lehmer} $$L(t)=t^{10}+t^9-t^7-t^6-t^5-t^4-t^3+t+1.$$ 

\noindent \textsc{McMullen} montre que ce nombre est une borne inf\'erieure pour l'entropie.

\begin{thm}[\cite{Mc}]
{\sl Si $f$ est un automorphisme d'une surface complexe compacte $\mathcal{Z}$ d'entropie positive, l'entropie de $f$ est sup\'erieure ou \'egale \`a $\log\lambda_{\text{Lehmer}}.$}
\end{thm}

\noindent Constatant que $\lambda_{10}=\lambda_{\text{Lehmer}},$ il obtient que cette borne inf\'erieure est r\'ealis\'ee.

\begin{cor}[\cite{Mc}]
{\sl La transformation $f_{10}\colon\mathcal{Z}_{10}\to\mathcal{Z}_{10}$ est un automorphisme d'une surface rationnelle d'entropie minimum.}
\end{cor}

\noindent On peut voir qu'un automorphisme sur un tore complexe ne peut pas avoir de disque de \textsc{Siegel}; ce n'est pas le cas sur une surface K$3$ du moins si celle-ci est non projective (\cite{Mc2}). Dans le contexte des surfaces rationnelles \textsc{McMullen} \'etablit l'\'enonc\'e suivant.

\begin{thm}[\cite{Mc}]
{\sl Il y a une infinit\'e de $n$ pour lesquels l'\'el\'ement standard de~$\mathrm{W}_n$ peut \^etre r\'ealis\'e comme un automorphisme de $\mathbb{P}^2(\mathbb{C})$ \'eclat\'e en un nombre fini de points poss\'edant un disque de \textsc{Siegel}.}
\end{thm}

\bigskip 

\subsubsection{Id\'ee de d\'emonstration}

\noindent Les automorphismes construits pour \'etablir les r\'esultats pr\'ec\'edents sont obtenus en \'eclatant des points bases situ\'es sur une cubique $\mathcal{C},$ les cubiques jouant un r\^ole particulier du fait que leurs transform\'ees propres $Y$ sont des courbes anticanoniques. 

\noindent Supposons que l'\'el\'ement $w$ de $\mathrm{W}_n$ soit r\'ealis\'e par un automorphisme $F$ de $\mathcal{Z}$ pr\'eservant $Y.$ Comme on l'a vu au Th\'eor\`eme \ref{foncteur}, \`a la paire marqu\'ee $(\mathcal{Z},\Phi,Y)$ est canoniquement associ\'ee une cubique marqu\'ee $(\mathcal{C},\eta).$ Alors il existe une transformation birationnelle $f\colon\mathbb{P}^2(\mathbb{C}) \dashrightarrow~\mathbb{P}^2(\mathbb{C})$ telle que: le relev\'e de $f$ \`a $\mathcal{Z}$ co\"incide avec $F,$ $f$ pr\'eserve $\mathcal{C}$ et $f$ induit un automorphisme $f_*$ de $\mathrm{Pic}_0(\mathcal{C})$ satisfaisant $\eta_0 w=f_*\eta_0.$ En d'autres termes $[\eta_0]$ est un point fixe pour l'action naturelle de $w$ sur l'espace des modules des marquages. 

\smallskip

\noindent R\'eciproquement pour r\'ealiser un \'element donn\'e $w$ du groupe $\mathrm{W}_n$ on commence par chercher un point fixe $\eta_0$ dans l'espace des modules des marquages. Le Corollaire \ref{cubcub} associe \`a~$\eta_0$ une cubique marqu\'ee $(\mathcal{C},\eta)$ \`a isomorphisme pr\`es. Notons $(\mathcal{Z},\Phi,Y)$ la paire marqu\'ee canoniquement d\'etermin\'ee par $(\mathcal{C},\eta).$ Supposons que pour tout $\alpha$ dans $\Theta_n,$ $\eta_0(\alpha)$ soit non nul (ce qui est une condition g\'en\'erique); les points bases $p_i$ ne satisfont pas de relation nodale (ils sont tous distincts, il n'y en a pas $3$ align\'es, pas $6$ sur une conique etc). Un th\'eor\`eme de \textsc{Nagata} assure alors l'existen\-ce d'une seconde projection $\pi'\colon\mathcal{Z}\to\mathbb{P}^2(\mathbb{C})$ correspondant au marquage $\Phi w.$ D\'esignons par~$\mathcal{C}'$ la cubique $\pi'(Y).$ Puisque $[\eta_0]$ est un point fixe de $w,$ les cubiques marqu\'ees~$(\mathcal{C}', \eta w)$ et~$(\mathcal{C},\eta)$ sont isomorphes. Mais un tel isomorphisme est un automorphisme $F$ de $\mathcal{Z}$ satisfai\-sant~$F_*\Phi=\Phi w.$

\smallskip

\noindent Notons que dans \cite{Hi, Ha2, PS, Di2} il y a aussi des constructions faisant intervenir \`a la fois les automorphismes de surfaces et les cubiques.

\section{Exemples}

\noindent Consid\'erons la famille de transformations birationnelles $f\colon\mathbb{P}^2(\mathbb{C}) \dashrightarrow\mathbb{P}^2(\mathbb{C})$ donn\'ee, dans la carte affine $z=1,$ par 
\begin{align*}
&f(x,y)=\left(a+y,b+\frac{y}{x}\right),&& a,\,b\in\mathbb{C}. 
\end{align*}

\noindent Notons que le cas $b=-a$ a \'et\'e \'etudi\'e dans \cite{PS} et \cite{BR}.

\noindent Les points d'ind\'etermination de $f$ sont $p_1=(0:0:1),$ $p_2=(0:1:0)$ et $p_3=(1:0:0).$ Posons $p_4=(a:b:1)$ et d\'esignons par $\Delta$ (resp. $\Delta'$) le triangle de sommets $p_1,$ $p_2,$ $p_3$ (resp. $p_2,$ $p_3,$ $p_4$). La transformation $f$ envoie $\Delta$ sur $\Delta':$ le point $p_1$ (resp. $p_2,$ resp. $p_3$) est \'eclat\'e sur la droite $(p_1p_4)$ (resp. $(p_2p_3),$ resp. $(p_3p_4)$) et les droites $(p_1p_2)$ (resp. $(p_1p_3),$ resp. $(p_2p_3)$) sont contract\'ees sur $p_2$ (resp. $p_4,$ resp. $p_3$).

\noindent Si $a$ et $b$ sont choisis de sorte que $p_1=p_4,$ alors $\Delta$ est invariant par $f$ et quitte \`a \'eclater~$\mathbb{P}^2(\mathbb{C})$ en $p_1,$ $p_2,$ $p_3$ on obtient une r\'ealisation de l'\'el\'ement standard de \textsc{Coxeter} dans $\mathrm{W}_3.$ En effet~$f$ envoie une droite g\'en\'erique sur une conique passant par les $p_i;$ ainsi $w(e_0)=2e_0-e_1-e_2-e_3.$ Le point $p_1$ (resp. $p_2,$ resp. $p_3$) est \'eclat\'e sur la droite passant par $p_2$ et $p_3$ (resp. $p_1$ et $p_3,$ resp. $p_1$ et $p_2$). Il s'en suit que 
\begin{align*}
& w(e_1)=e_0-e_2-e_3, && w(e_2)=e_0-e_1-e_3, && w(e_3)=e_0-e_1-e_2.
\end{align*}

\noindent Plus g\'en\'eralement on a l'\'enonc\'e suivant.

\begin{thm}[\cite{Mc}]\label{refMc}
{\sl D\'esignons par $p_{i+4}$ l'it\'er\'e $i$\`eme $f^i(p_4)$ de $p_4.$ 

\noindent La r\'ealisation de l'\'el\'ement standard de \textsc{Coxeter} de $\mathrm{W}_n$ correspond aux couples $(a,b)$ de~$\mathbb{C}^2$ tels que 
\begin{align*}
& p_i\not\in (p_1p_2)\cup(p_2p_3)\cup(p_3p_1), && p_{n+1}=p_1.
\end{align*}}
\end{thm}

\begin{proof}[{\sl D\'emonstration}]
Supposons qu'il existe un entier $i$ tel que $f^i(p_4)=p_{i+4}.$ Soit $(\mathcal{Z},\pi)$ l'\'eclatement marqu\'e de points bases $p_i.$ La transformation $f$ se rel\`eve en un morphisme $F_0$ de~$\mathcal{Z}$ dans~$\mathbb{P}^2(\mathbb{C}).$ Comme tout $p_i$ est d\'esormais l'image $F_0(\ell_i)$ d'une droite de $\mathcal{Z},$ le morphisme $F_0$ se rel\`eve en un automorphisme $F$ de $\mathcal{Z},$ $F$ satisfaisant la propri\'et\'e suivante: $f$ se rel\`eve en $F.$ D\'eterminons l'\'el\'ement $w$ r\'ealis\'e par $F.$ Remarquons que $f$ envoie une droite g\'en\'erique sur une conique passant par $p_2,$ $p_3$ et $p_4$ d'o\`u~$w(e_0)=2e_0-e_2-e_3-e_4.$ Le point~$p_1$ est \'eclat\'e sur la droite passant par $p_3$ et $p_4$ donc~$w(e_1)=e_0-e_3-e_4;$ de m\^eme on obtient 
\begin{small}
\begin{align*}
&w(e_2)=e_0-e_2-e_4,&& w(e_3)=e_0-e_2-e_3, && w(e_i)=e_{i+1}\text{ pour $4\leq i<n$ } && \text{ et } &&w(e_n)=e_1.
\end{align*}
\end{small}

\noindent R\'eciproquement si un automorphisme $F\colon\mathcal{Z}\to\mathcal{Z}$ r\'ealise la transformation standard de \textsc{Coxeter} $w=\pi_n\kappa_{123},$ on peut normaliser les points bases de sorte que $$\{p_1,\,p_2,\,p_3\}=\{(0:0:1),\, (0:1:0),\, (1:0:0)\};$$ la transformation birationnelle $f\colon\mathbb{P}^2(\mathbb{C})\dashrightarrow\mathbb{P}^2(\mathbb{C})$ dont le relev\'e co\"incide avec~$F$ est la compos\'ee de l'involution de \textsc{Cremona} et d'un automorphisme envoyant $(p_1,\, p_2)$ sur~$(p_2,\,p_3).$ Un tel $f$ est, dans la carte affine $z=1,$ de la forme $f(x,y)=(a',b')+(Ay,By/x)$ soit, \`a conjugaison pr\`es par $(Bx,By/A),$ du type $f(x,y)=(a,b)+(y,y/x).$  
\end{proof}

\chapter{Exemples d'automorphismes d'entropie positive}\label{chapbedkim1}

\noindent Dans ce paragraphe nous pr\'esentons un exemple d\^u \`a \textsc{Bedford} et \textsc{Kim} (\cite{BK1,BK2}).

\noindent Une possibilit\'e pour produire un automorphisme $f$ sur une surface rationnelle $\mathcal{Z}$ est la suivante: on part d'une transformation birationnelle $f$ de $\mathbb{P}^2( \mathbb{C}),$ on trouve une suite d'\'eclatements $\pi\colon\mathcal{Z}\to\mathbb{P}^2(\mathbb{C})$ telle que l'application induite $f_\mathcal{Z}=\pi f\pi^{-1}$ soit un automorphisme sur~$\mathcal{Z}.$ La difficult\'e est de trouver une telle suite $\pi$... Si $f$ n'est pas un automorphisme du plan projectif complexe, $f$ contracte une courbe $\mathcal{C}_1$ sur un point $p_1;$ la premi\`ere chose \`a faire pour construire un automorphisme \`a partir de $f$ est d'\'eclater le point $p_1$ via $\pi_1\colon\mathcal{Z}_1\to\mathbb{P}^2(\mathbb{C}).$ Dans le meilleur des cas $f_{\mathcal{Z}_1}=\pi_1f\pi_1^{-1}$ envoie la transform\'ee de $\mathcal{C}_1$ sur le diviseur exceptionnel $\mathrm{E}_1.$ Mais si~$p_1$ n'est pas un point d'ind\'etermination $f_{\mathcal{Z}_1}$ contracte $\mathrm{E}_1$ sur $p_2=f(p_1).$ Autrement dit ce proc\'ed\'e n'aboutit que si $f$ n'est pas alg\'ebriquement stable ce qui va \^etre le cas dans l'exemple qui suit.

\bigskip

\noindent Soit $f_{a,b}$ la transformation birationnelle de $\mathbb{P}^2(\mathbb{C})$ d\'efinie par $$f_{a,b}(x,y,z)=(x(bx+y):z(bx+y):x(ax+z)),$$ soit en carte affine $$f_{a,b}(y,z)=\left(z,\frac{a+z}{b+y} \right).$$ On constate que $\mathrm{Ind}\,f_{a,b}=\{p_1,\, p_2,\,p_*\}$ et que $\mathrm{Exc}\,f_{a,b}=\Sigma_0\cup \Sigma_\beta\cup\Sigma_\gamma$ avec 
\begin{align*}
& p_1=(0:1:0),&& p_2=(0:0:1), &&p_*=(1:-b:-a),\\
&\Sigma_0=\{x=0\}, && \Sigma_\beta=\{bx+y=0\}, &&\Sigma_\gamma=\{ax+z=0\}. 
\end{align*}

\noindent Plus pr\'ecis\'ement $\Sigma_0$ (resp. $\Sigma_\beta,$ resp. $\Sigma_\gamma$) est contract\'e sur $p_1$ (resp. $p_2,$ resp. $q=(1:-a:0)$) et~$p_2$ (resp. $p_1,$ resp. $q=(1:-a:0)$) est \'eclat\'e sur $\Sigma_0$ (resp. sur la droite $\Sigma_B=\{z=0\}$ passant par les points $p_1$ et $q,$ resp. sur la droite $\Sigma_C$ passant par les points $p_2$ et $q$).

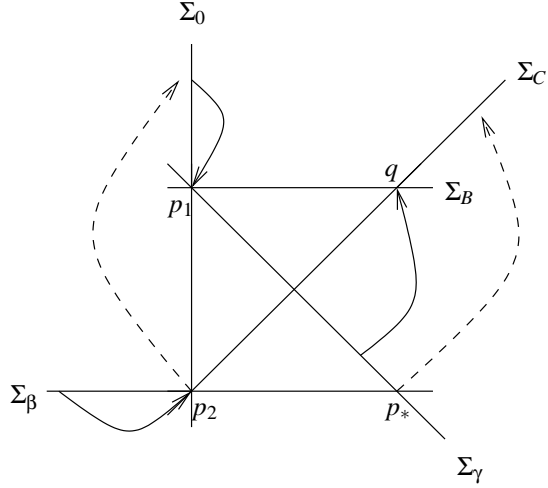
\begin{figure}[H]
\begin{center}
\input{bedkim.pstex_t}
\end{center}
\end{figure}

\noindent Soient $\mathcal{Y}=\mathrm{Bl}_{p_1,p_2}\mathbb{P}^2,$ $\pi\colon\mathcal{Y}\to\mathbb{P}^2(\mathbb{C})$ et $f_{a,b,\,\mathcal{Y}}=\pi^{-1} f_{a,b}\pi.$ Montrons qu'apr\`es cet \'eclatement~$\Sigma_0$ n'appartient pas \`a $\mathrm{Exc} (f_{a,b,\,\mathcal{Y}}).$  

\noindent Posons $x=r_2$ et $y=r_2s_2;$ alors $(r_2,s_2)$ est un syst\`eme de coordonn\'ees locales dans lequel~$\Sigma_\beta=\{s_2+b=0\}$ et $\mathrm{E}_2=\{r_2=0\}.$ On constate que 
\begin{align*}
& (r_2,s_2)\to(r_2,r_2s_2)_{(x,y)}\to(r_2(b+s_2):b+s_2:ar_2+1)= \left(\frac{r_2(b+s_2)}{ar_2+1},\frac{b+s_2}{ar_2+1}\right)_{(x,y)}\\
&\hspace{1cm}\to \left(\frac{r_2(b+s_2)}{ar_2+1},\frac{1}{r_2}\right)_{(r_2,s_2)}.
\end{align*}

\noindent Par suite $\Sigma_\beta$ est envoy\'e sur $\mathrm{E}_2$ et $\mathrm{E}_2$ sur $\Sigma_0.$ 

\noindent Posons $x=u_2v_2$ et $y=v_2;$ le diviseur exceptionnel $\mathrm{E}_2$ est d\'ecrit par $v_2=0$ et $\Sigma_0$ par $u_2=0.$ On a 
\begin{align*}
&(u_2,v_2)\to(u_2v_2,v_2)_{(x,y)}\to(u_2v_2(bu_2+1):bu_2+1:u_2(au_2v_2+1))\\
&\hspace{1cm}=\left(\frac{v_2(bu_2+1)}{au_2v_2+1},\frac{bu_2+1}{u_2(au_2v_2+1)}\right)_{(x,y)}
\to\left(u_2v_2,\frac{bu_2+1}{u_2(au_2v_2+1)}\right)_{(u_2,v_2)};
\end{align*}

\noindent on retrouve que $\mathrm{E}_2$ est envoy\'e sur $\Sigma_0.$ 

\noindent Posons $x=r_1,$ $z=r_1s_1;$ dans ces coordonn\'ees $(r_1,s_1)$ on a $\mathrm{E}_1=\{r_1=0\}.$ De plus 
\begin{align*}
& (r_1,s_1)\to(r_1,r_1s_1)_{(x,z)}\to(br_1+1:b+s_1(br_1+1):r_1(a+s_1)).
\end{align*}

\noindent Il s'en suit que $\mathrm{E}_1$ est envoy\'e sur $\Sigma_B.$ 

\noindent Posons $x=u_1v_1$ et $z=v_1;$ dans ces coordonn\'ees $\Sigma_0=\{u_1=0\},$ $\mathrm{E}_1=\{v_1=0\}$ et 
\begin{align*}
&(u_1,v_1)\to(u_1v_1,v_1)_{(x,z)}\to(u_1(bu_1v_1+1):bu_1v_1+1:u_1v_1(au_1+1))\\
&\hspace{1cm}=\left(u_1, \frac{u_1v_1(au_1+1)}{bu_1v_1+1}\right)_{(x,z)}\to\left(u_1,\frac{v_1(au_1+1)}{bu_1v_1+1}\right)_{(r_1,s_1)}.
\end{align*} 

\noindent Il en r\'esulte que $\Sigma_0\to\mathrm{E}_1$ et que $\Sigma_\beta\to\mathrm{E}_2\to\Sigma_0\to\mathrm{E}_1\to \Sigma_B.$ En particulier $\mathrm{Ind}\,f_{a,b,\,\mathcal{Y}}=\{p_*\}$ et $\mathrm{Exc}\,f_{a,b,\,\mathcal{Y}}=\{\Sigma_\gamma\}.$

\noindent On remarque que $\{\mathrm{H},\,\mathrm{E}_1,\,\mathrm{E}_2\}$ est une base de $\mathrm{Pic}(\mathcal{Y}).$ Le diviseur exceptionnel $\mathrm{E}_1$ est envoy\'e sur $\Sigma_B;$ puisque~$p_1$ appartient \`a $\Sigma_B$ on a $\mathrm{E}_1\to\Sigma_B\to\Sigma_B+\mathrm{E}_1.$ Par ailleurs $\mathrm{E}_2$ est envoy\'e sur~$\Sigma_0;$ comme~$p_1$ et $p_2$ appartiennent \`a $\Sigma_0$ on a $\mathrm{E}_2\to\Sigma_0\to\Sigma_0+\mathrm{E}_1+\mathrm{E}_2.$ Soit $\mathrm{H}$ une droite g\'en\'erique de~$\mathbb{P}^2(\mathbb{C})$ d'\'equation $\ell=0$ avec $\ell=a_0x+a_1y+a_2z.$ Son image par $f_{a,b,\,\mathcal{Y}}$ est une conique d'o\`u $$f_{a,b,\,\mathcal{Y}}^*\mathrm{H}=2\mathrm{H}+\sum_{i=1}^2m_i \mathrm{E}_i.$$ D\'eterminons les $m_i.$ \'Etant donn\'e que 
\begin{align*}
&(r_2,s_2)\to(r_2,r_2s_2)_{(x,y)}\to 
(r_2(b+s_2):b+s_2:ar_2+1)\\
&\hspace{1cm}\to r_2\Big(a_0r_2(b+s_2) +a_1(b+s_2)+a_2(ar_2+1)\Big)
\end{align*}

\noindent 
et que $\mathrm{E}_2=\{r_2=0\}$ l'entier $m_2$ vaut $1.$ D'autre part 
\begin{align*}
&(r_1,s_1)\to(r_1,r_1s_1)_{(x,z)}\to(br_1+1:b+s_1(br_1+1):r_1(a+s_1))\\
&\hspace{1cm}\to s_1r_1\Big(a_0(bs_1r_1+1)+a_1s_1(bs_1r_1+1)+s_1r_1(a+s_1)\Big)
\end{align*}

\noindent 
et $\mathrm{E}_1=\{s_1=0\}$ d'o\`u $m_1=1.$ Il s'en suit que $$M_{f_{a,b,\,\mathcal{Y}}}=\left[\begin{array}{ccc} 2 & 1 & 1 \\ -1 & -1 & -1 \\ -1 & 0 & -1\end{array}\right].$$ Le polyn\^ome caract\'eristique de $M_{f_{a,b,\,\mathcal{Y}}}$ est $1+t-t^3.$ Traduisons les renseignements que contient $M_{f_{a,b,\,\mathcal{Y}}}.$ Soient $\mathrm{L}$ une droite et $\{\mathrm{L}\}$ sa classe dans $\mathrm{Pic}(\mathcal{Y}).$ Si $\mathrm{L}$ n'intersecte ni $\mathrm{E}_1,$ ni~$\mathrm{E}_2,$ alors $\{\mathrm{L}\}=\mathrm{H}.$ Puisque $f_{a,b,\,\mathcal{Y}}^*\mathrm{H}=2\mathrm{H}-\mathrm{E}_1-\mathrm{E}_2$ l'image de~$\mathrm{L}$ par $f_{a,b,\,\mathcal{Y}}$ est une conique intersectant $\mathrm{E}_1$ et $\mathrm{E}_2$ avec multiplicit\'e $1.$ Si $\mathrm{L}$ contient $p_*,$ alors $f_{a,b,\,\mathcal{Y}}(\mathrm{L})$ est l'union de~$\Sigma_C$ et d'une seconde droite. Supposons que $p_*$ n'appartienne pas \`a $\mathrm{L}\cup f_{a,b,\,\mathcal{Y}}(\mathrm{L}),$ alors $$\{f^2_{a,b,\,\mathcal{Y}}(\mathrm{L})\}=M_{f_{a,b}}^2\left[\begin{array}{ccc} 1 \\ 0 \\ 0 \end{array}\right]=2\mathrm{H}-\mathrm{E}_2;$$ autrement dit $f_{a,b,\,\mathcal{Y}}^2(\mathrm{L})$ est une conique intersectant $\mathrm{E}_2$ mais pas $\mathrm{E}_1.$ De m\^eme si $p_*$ n'appartient pas \`a $\mathrm{L}\cup f_{a,b,\,\mathcal{Y}}(\mathrm{L})\cup f_{a,b,\,\mathcal{Y}}^2(\mathrm{L}),$ alors $$\{f^3_{a,b,\,\mathcal{Y}}(\mathrm{L})\}=M_{f_{a,b}}^3\left[\begin{array}{ccc} 1 \\ 0 \\ 0 \end{array}\right]=3 \mathrm{H}-\mathrm{E}_1-\mathrm{E}_2,$$ {\it i.e.} $f_{a,b,\,\mathcal{Y}}^3(\mathrm{L})$ est une cubique intersectant $\mathrm{E}_1$ et $\mathrm{E}_2$ avec multiplicit\'e $1.$ Si $p_*$ n'appartient pas \`a $\mathrm{L}\cup f_{a,b,\,\mathcal{Y}}(\mathrm{L})\cup\ldots\cup f_{a,b,\,\mathcal{Y}}^{n-1}(\mathrm{L}),$ les it\'er\'es de $f_{a,b,\,\mathcal{Y}}$ sont holomorphes au voisinage de $\mathrm{L}$ et~$(f_{a,b,\,\mathcal{Y}}^*)^n (\mathrm{H})=\{f_{a,b,\,\mathcal{Y}}^n\mathrm{L}\}.$ Les param\`etres $a$ et~$b$ sont dits {\it g\'en\'eriques} si $p_*$ n'appartient pas \`a~$\bigcup_{j=0}^\infty f_{a,b,\,\mathcal{Y}}^j(\mathrm{L}).$

\begin{thm}[\cite{BK2}]\label{abgen}
{\sl Supposons que $a$ et $b$ soient g\'en\'eriques; alors $f_{a,b,\,\mathcal{Y}}$ est alg\'ebriquement stable et $\lambda(f_{a,b})\sim 1.324$ est la plus grande valeur propre du polyn\^ome caract\'eristique $t^3-t-1.$}
\end{thm}

\section{Construction de surfaces et d'automorphismes}

\noindent Consid\'erons le sous-espace $\mathcal{V}_n$ de $\mathbb{C}^2$ d\'efini par $$\mathcal{V}_n=\{(a,b)\in \mathbb{C}^2\,\vert\, f_{a,b,\,\mathcal{Y}}^j(q)\not=p_*\,\,\,\forall\,0\leq j\leq n-1,\,f_{a,b,\,\mathcal{Y}}^n(q)=p_*\}.$$ L'\'enonc\'e suivant donne une condition n\'ecessaire et suffisante pour que $f_{a,b,\,\mathcal{Y}}$ soit un automorphisme sur une certaine surface rationnelle.

\begin{thm}[\cite{BK2}]
{\sl La transformation $f_{a,b,\,\mathcal{Y}}$ est conjugu\'ee \`a un automorphisme d'une certaine surface rationnelle si et seulement si $(a,b)$ appartient \`a $\mathcal{V}_n$ pour un certain~$n.$}
\end{thm}

\begin{proof}[{\sl D\'emonstration}]
Si $(a,b)$ n'appartient pas \`a $\mathcal{V}_n,$ le Th\'eor\`eme \ref{abgen} assure que $\lambda( f_{a,b})$ est la plus grande racine de $t^3-t-1;$ on constate que $\lambda(f_{a,b})$ n'est pas un nombre de \textsc{Salem} donc~$f_{a,b}$ n'est pas conjugu\'e \`a un automorphisme (Th\'eor\`eme \ref{lambda}).

\noindent R\'eciproquement supposons qu'il existe un entier $n$ tel que $(a,b)$ appartienne \`a $\mathcal{V}_n.$ Soit $\mathcal{Z}$ la surface obtenue \`a partir de $\mathcal{Y}$ en \'eclatant les points $q,$ $f_{a,b,\,\mathcal{Y}}(q),$ $\ldots,$ $f_{a,b,\,\mathcal{Y}}^n(q)=p_*$ de l'orbite de $q.$ On peut v\'erifier que la transformation induite $f_{a,b,\,\mathcal{Z}}$ est un automorphisme.
\end{proof} 

\noindent On s'int\'eresse maintenant \`a $f_{a,b,\,\mathcal{Z}}^*$ que nous noterons $f_{a,b}^*.$

\begin{thm}[\cite{BK2}]
{\sl Supposons que $(a,b)$ appartienne \`a $\mathcal{V}_n$ pour un certain entier~$n.$ Si $n\leq 5,$ l'application $f_{a,b}$ est p\'eriodique de p\'eriode inf\'erieure ou \'egale \`a $30.$ Si $n$ vaut~$6,$ la croissance des degr\'es de $f_{a,b}$ est quadratique. Enfin si $n\geq 7,$ alors $(\deg f_{a,b}^k)_k$ a une croissance exponentielle et $\lambda(f_{a,b})$ est la plus grande valeur propre du polyn\^ome caract\'eristique $$\chi_n(t)=t^{n+1}(t^3-t-1)+t^3+t^2-1.$$ De plus, lorsque $n$ tend vers l'infini $\lambda(f_{a,b})$ tend vers la plus grande valeur propre de $t^3-t-1.$}
\end{thm}

\noindent L'action de $f_{a,b,\,\mathcal{Z}*}$ sur la cohomologie est donn\'ee par $$\mathrm{E}_2\to\Sigma_0=\mathrm{H}-\mathrm{E}_1-\mathrm{E}_2\to\mathrm{E}_1\to\Sigma_B=\mathrm{H}- \mathrm{E}_1-\mathrm{Q}$$ o\`u $\mathrm{Q}$ d\'esigne le diviseur obtenu en \'eclatant le point $q$ qui est sur $\Sigma_B.$ Par ailleurs $p_*$ \'etant \'eclat\'e par $f_{a,b}$ sur~$\Sigma_C,$ on a $$\mathrm{Q}\to f_{a,b}(\mathrm{Q})\to\ldots\to f_{a,b}^n(\mathrm{Q})\to\Sigma_C=\mathrm{H}-\mathrm{E}_2- \mathrm{Q}.$$ Finalement une droite g\'en\'erique $\mathrm{L}$ intersecte $\Sigma_0,$ $\Sigma_\beta$ et $\Sigma_\gamma$ avec multiplicit\'e $1;$ l'image de $\mathrm{L}$ est donc une conique passant par $q,$ $p_1$ et $p_2$ d'o\`u $\mathrm{H}\to 2\mathrm{H}-\mathrm{E}_1-\mathrm{E}_2 -\mathrm{Q}.$ Ainsi dans la ba\-se~$\{\mathrm{H},\,\mathrm{E}_1,\,\mathrm{E}_2,\,\mathrm{Q},\, f_{a,b}(\mathrm{Q}),\,\ldots,\,f_{a,b}^n(\mathrm{Q})\}$ on a $$M_{f_{a,b}}=\left[\begin{array}{ccccccccc} 
2& 1 &1 &0 &0 &\ldots &\ldots & 0 &1\\
-1 & -1 & -1 & 0 &0 &\ldots& \ldots& 0& 0\\
-1 & 0 & -1 & 0 & 0 &\ldots&\ldots& 0& -1\\
-1 & -1 & 0 & 0 & 0 &\ldots&\ldots& 0& -1\\
0 & 0 & 0 & 1 & 0 &\ldots & \ldots & 0 & 0\\
0 & 0 & 0 & 0 & 1 &0 & \ldots & 0 & \vdots\\
\vdots & \vdots & \vdots & \vdots & 0 &\ddots & \ddots & \vdots & \vdots\\
\vdots & \vdots & \vdots & \vdots & \vdots &\ddots & \ddots & 0 & 0\\
0 & 0 & 0 & 0 & 0 & \ldots &0 & 1 & 0\\
\end{array}\right].$$

\bigskip

\section{Courbes invariantes}

\noindent Dans l'esprit de \cite{DJS} \textsc{Bedford} et \textsc{Kim} \'etudient les courbes invariantes 
par~$f_{a,b}.$ Il existe des transformations rationnelles $\varphi_j\colon\mathbb{C}\to\mathbb{C}^2$ telles que si $(a,b)=~\varphi_j(t)$ pour un certain complexe $t,$ alors $f_{a,b}$ a une courbe invariante $\mathcal{C}$ ayant $j$ composantes irr\'eductibles. Posons
\begin{small} 
\begin{align*}
&\varphi_1(t)=\left(\frac{t-t^3-t^4}{1+2t+t^2},\frac{1-t^5}{t^2+t^3}\right), &&\varphi_2(t)=\left(\frac{t+t^2+t^3}{1+2t+t^2},\frac{t^3-1}{t+t^2}\right),&&\varphi_3(t)=\left(1+t,t-\frac{1}{t}\right).
\end{align*}
\end{small}

\begin{thm}[\cite{BK2}]
{\sl Soit $t$ dans $\mathbb{C}\setminus\{-1,\,1,\,0,\,\mathbf{j},\,\mathbf{j}^2\}.$ Il existe une cubique $\mathcal{C}$ invariante par $f_{a,b}$ si set seulement si $(a,b)=\varphi_j(t)$ pour un certain $1\leq j\leq 3$ auquel cas $\mathcal{C}$ est d\'ecrite par un polyn\^ome homog\`ene $P_{t,a,b}$ de degr\'e $3.$

\noindent De plus, si $P_{t,a,b}$ existe, il est donn\'e, \`a multiplication pr\`es par une constante, par
\begin{small}
\begin{align*}
&P_{t,a,b}(x,y,z)=ax^3(t-1)t^4+yz(t-1)t(z+ty)+x\Big(2byzt^3+y^2(t-1)t^3+z^2(t-1)(1+bt)\Big)\\ &\hspace{2cm}+x^2(t-1)t^3\Big(a(y+tz)+t(y+(t-2b)z)\Big).
\end{align*}
\end{small}}
\end{thm}

\noindent Plus pr\'ecis\'ement on a la description suivante.

\begin{itemize}
\item Si $(a,b)=\varphi_1(t),$ alors $\Gamma_1=(P_{t,a,b}=0)$ est une cubique irr\'eductible cuspidale. La transformation $f_{a,b}$ a deux points fixes dont l'un est le point singulier de $\mathcal{C}.$ 

\item Si $(a,b)=\varphi_2(t),$ alors $\Gamma_2=(P_{t,a,b}=0)$ est l'union d'une conique et d'une droite tangente \`a celle-ci. La transformation $f_{a,b}$ a, l\`a aussi, deux points fixes.

\item Si $(a,b)=\varphi_3(t),$ alors $\Gamma_3=(P_{t,a,b}=0)$ est l'union de trois droites concourantes; $f_{a,b}$ a encore deux points fixes dont l'un est le point d'intersection des trois composantes de~$\mathcal{C}.$ 
\end{itemize}

\noindent Il existe un lien \'etroit entre les param\`etres $(a,b)$ pour lesquels il existe un complexe $t$ tel que~$\varphi_j(t)=(a,b)$ et les racines du polyn\^ome caract\'eristique $\chi_n.$

\begin{thm}[\cite{BK2}]
{\sl Soient $n$ un entier, $j$ un entier compris entre $1$ et $3$ et $t$ un complexe. Supposons que $(a,b):=\varphi_j(t)$ n'appartienne \`a aucun des $\mathcal{V}_k$ pour $k<n.$ Alors~$(a,b)$ appartient \`a $\mathcal{V}_n$ si et seulement si $j$ divise $n$ et $t$ est une racine de $\chi_n.$}
\end{thm}

\noindent On peut \'ecrire $\chi_n$ sous la forme $C_n\psi_n$ o\`u $C_n$ est le produit des facteurs cyclotomiques et $\psi_n$ le polyn\^ome minimal de $\lambda(f_{a,b}).$

\begin{thm}[\cite{BK2}]
{\sl Supposons que $n$ soit sup\'erieur ou \'egal \`a $7.$ Soit $t$ une racine de~$\chi_n$ distincte de $1.$ Alors ou bien $t$ est une racine de $\psi_n,$ ou bien $t$ est une racine de $\chi_j$ pour un certain $0\leq j\leq 5.$ }
\end{thm}

\noindent En d\'ecrivant de mani\`ere pr\'ecise ce qu'il se passe lorsque $t$ n'est pas une racine de $\psi_n,$ \textsc{Bedford} et \textsc{Kim} montrent que le nombre d'\'el\'ements de $\Gamma_j\cap\mathcal{V}_n$ est, pour $n\geq 7,$ d\'etermin\'e par le nombre de conjugu\'es de \textsc{Galois} de l'unique racine de $\psi_n$ strictement sup\'erieure \`a $1:$ si~$n\geq 7$ et $1\leq j\leq 3$ divise $n,$ alors $$\Gamma_j\cap\mathcal{V}_n=\{\varphi_j(t)\,\vert\, t\text{ racine de }\psi_n\};$$ en particulier $\Gamma_j\cap\mathcal{V}_n$ n'est pas vide.

\noindent Soient $\mathcal{X}$ une surface rationnelle et $g$ un automorphisme sur $\mathcal{X}.$ Le couple $(\mathcal{X},g)$ est dit {\it minimal}\label{ind59} si tout morphisme birationnel $\pi\colon \mathcal{X}\to\mathcal{X}'$ qui envoie $(\mathcal{X},g)$ sur $(\mathcal{X}',g'),$ o\`u $g'$ est un automorphisme sur $\mathcal{X}',$ est un isomorphisme. Rappelons une question pos\'ee dans \cite{Mc}. Soient $\mathcal{X}$ une surface rationnelle et $g$ un automorphisme sur~$\mathcal{X}.$ Supposons que $(\mathcal{X},g)$ soit minimal. Existe-t-il une puissance n\'egative de la classe du diviseur canonique $\mathrm{K}_\mathcal{X}$ admettant une section holomorphe ? On sait depuis \cite{Ha} que la r\'eponse est n\'egative si on enl\`eve l'hypoth\`ese \og $(\mathcal{X},g)$ minimal\fg. 

\begin{thm}[\cite{BK2}]\label{bedfordkim}
{\sl Il existe une surface $\mathcal{Z}$ et un automorphisme $f_{a,b}$ sur $\mathcal{Z}$ d'entropie positive tels que $(\mathcal{Z},f_{a,b})$ soit minimal et tels que $f_{a,b}$ ne poss\`ede pas de courbe invariante.}
\end{thm}

\noindent Si $g$ est un automorphisme sur une suface rationnelle $\mathcal{X}$ telle qu'une puissance n\'egative de~$\mathrm{K}_\mathcal{X}$ admette une section holomorphe, $g$ pr\'eserve une courbe; par cons\'equent le Th\'eor\`eme~\ref{bedfordkim} r\'epond \`a la question de \textsc{McMullen}.

\section{Domaines de rotation}

\noindent Supposons que $n$ soit sup\'erieur ou \'egal \`a $7$ auquel cas $f_{a,b}$ n'est pas p\'eriodique; s'il y a un domaine de rotation, son rang vaut $1$ ou $2.$ Nous allons voir que les deux arrivent; commen\c{c}ons par ceux de rang $1.$

\begin{thm}[\cite{BK2}]
{\sl Supposons que $n$ soit sup\'erieur ou \'egal \`a $7,$ que $j$ divise $n$ et que $(a,b)$ appartienne \`a~$\Gamma_j\cap\mathcal{V}_n.$ Il existe un complexe $t$ tel que $(a,b)= \varphi_j(t).$ Si $t$ est un conjugu\'e de \textsc{Galois} de $\lambda(f_{a,b})$ distinct de~$\lambda( f_{a,b})$ et $\lambda( f_{a,b})^{-1},$ alors $f_{a,b}$ a un domaine de rotation de rang $1$ centr\'e en 
\begin{align*}
& \left(\frac{t^3}{1+t},\frac{t^3}{1+t}\right)\text{ si } j=1,&& \left(-\frac{t^2}{1+t},-\frac{t^2}{1+t}\right)\text{ si } j=2,&& (-t,-t)\text{ si } j=3.
\end{align*}
}
\end{thm}

\noindent Passons \`a ceux de rang $2.$

\begin{thm}[\cite{BK2}]
{\sl Consid\'erons un entier $n$ sup\'erieur ou \'egal \`a $8,$ un entier $j$ \`a valeurs dans $\{2,\,3\}$ qui divise $n.$ Supposons que $(a,b)=\varphi_j(t)$ et que $\vert t\vert=1;$ supposons de plus que $t$ soit une racine de $\psi_n.$ Notons $\eta_1,$ $\eta_2$ les valeurs propres de~$Df_{a,b}$ au point
\begin{align*}
&m=\left(\frac{1+t+t^2}{t+t^2},\frac{1+t+t^2}{t+t^2}\right) \text{ si }j=2, && m=\left(1+\frac{1}{t},1+\frac{1}{t}\right) \text{ si }j=3.
\end{align*}

\noindent Si $\eta_1$ et $\eta_2$ sont de module $1,$ alors $f_{a,b}$ a un domaine de rotation de rang $2$ centr\'e en $m.$}
\end{thm}

\noindent Il peut y avoir coexistence de domaines de rotation de rang $1$ et $2.$

\begin{thm}
{\sl Supposons que $n$ soit sup\'erieur ou \'egal \`a $8,$ que $j$ soit \'egal \`a $2$ ou $3$ et que $j$ divise $n.$ Il existe $(a,b)$ dans $\Gamma_j\cap\mathcal{V}_n$ tel que $f_{a,b}$ poss\`ede un domaine de rotation de rang $2$ centr\'e en 
\begin{align*}
&\left(\frac{1+t+t^2}{t+t^2},\frac{1+t+t^2}{t+t^2}\right) \text{ si }j=2, && \left(1+\frac{1}{t},1+\frac{1}{t}\right) \text{ si }j=3
\end{align*}

\noindent ainsi qu'un domaine de rotation de rang $1$ centr\'e en 
\begin{align*}
& \left(-\frac{t^2}{1+t},-\frac{t^2}{1+t}\right)\text{ si } j=2,&& (-t,-t)\text{ si } j=3.
\end{align*}}
\end{thm}

\section{Groupe de \textsc{Weyl}}

\noindent Rappelons que $\mathrm{E}_1$ et $\mathrm{E}_2$ sont les diviseurs obtenus en \'eclatant $p_1$ et~$p_2.$ Pour simplifier introduisons les notations suivantes: $\mathrm{E}_0=\mathrm{H}$ et $\mathrm{E}_3= \mathrm{Q},$ $\mathrm{E}_4=f(\mathrm{Q}),$ $\ldots,$ $\mathrm{E}_n=f^{n-3}(\mathrm{Q}),$ $\pi_i$ \'eclatement associ\'e au diviseur $\mathrm{E}_i.$ Posons 
\begin{align*}
& e_0=\mathrm{E}_0, && e_i=(\pi_{i+1}\ldots\pi_n)^*\mathrm{E}_i,&& 1\leq i\leq n;
\end{align*}
\noindent la base $\{e_0,\ldots,e_n\}$ de $\mathrm{Pic}(\mathcal{Z})$ est g\'eom\'etrique.

\noindent \textsc{Bedford} et \textsc{Kim} montrent qu'ils peuvent appliquer le Th\'eor\`eme \ref{refMc} et en d\'eduisent l'\'enonc\'e suivant.

\begin{thm}[\cite{BK2}]
{\sl Soient $\mathcal{X}$ une surface rationnelle obtenue en \'eclatant $\mathbb{P}^2(\mathbb{C})$ en un nombre fini de points $\pi\colon\mathcal{X}\to\mathbb{P}^2(\mathbb{C})$ et $F$ un automorphisme sur~$\mathcal{X}$ repr\'esentant l'\'el\'ement standard du groupe de \textsc{Weyl}~$\mathrm{W}_n$ avec $n\geq 5.$ Il existe un automorphisme $A$ de $\mathbb{P}^2(\mathbb{C})$ et des complexes $a$ et $b$ tels que $$f_{a,b} A\pi= A\pi F.$$}
\end{thm}

\noindent De plus ils obtiennent qu'une repr\'esentation de l'\'el\'ement standard du groupe de \textsc{Weyl} peut \^etre obtenue \`a partir de $f_{a,b,\,\mathcal{Y}}.$

\begin{thm}[\cite{BK2}]
{\sl Soient $\mathcal{X}$ une surface rationnelle et $F$ un automorphisme sur~$\mathcal{X}$ repr\'esentant l'\'el\'ement standard du groupe de \textsc{Weyl} $\mathrm{W}_n.$ Il existe
\begin{itemize}
\item une surface $\widetilde{\mathcal{Y}}$ obtenue en \'eclatant $\mathcal{Y}$ en un nombre fini de points distincts $\pi\colon\widetilde{\mathcal{Y}}\to\mathcal{Y},$

\item un automorphisme $g$ sur $\widetilde{\mathcal{Y}},$

\item $(a,b)$ dans $\mathcal{V}_{n-3}$
\end{itemize}

\noindent tels que $(F,\mathcal{X})$ soit conjugu\'e \`a $(g,\widetilde{\mathcal{Y}})$ et $\pi g=f_{a,b,\,\mathcal{Y}}\pi.$}
\end{thm}

\section{Description des conditions pour obtenir un automorphisme}

\noindent Dans \cite{BK1} est \'etudi\'ee la famille $(F_{A,B})$ d\'efinie dans la carte $x=1$ par 
\begin{align*}
&F_{A,B}=\left(z,\frac{a_0+a_1y+a_2z}{b_0+b_1y+b_2z}\right), &&  A=(a_0,a_1,a_2),\,B=(b_0,b_1,b_2);
\end{align*}
\noindent on remarque que $(f_{a,b})$ en est une sous-famille. On peut dans ce cadre d\'efinir des espaces de param\`etres $\mathcal{V}_n$ qui correspondent aux param\`etres pour lesquels l'application n'est pas alg\'ebriquement stable \og au $n$-i\`eme it\'er\'e\fg. Il y a une action de groupes naturels sur l'espace des param\`etres $\mathcal{V}_n;$ en effet si $(\eta,c,\mu)$ appartient \`a $\mathbb{C}^*\times\mathbb{C}^*\times\mathbb{C}$ on a les actions suivantes
\begin{align*}
& \mathfrak{(a1)} && (u,v)\mapsto(\eta u,\eta v)\\
& \mathfrak{(a2)} && (u,v)\mapsto(u_0,cu_1,cu_2,cv_0,c^2v_1,c^2v_2)\\
& \mathfrak{(a3)} && (u,v)\mapsto(u_0+\mu(u_1+u_2)-\mu(v_0+\mu(v_1+v_2)),u_1-\mu v_1,u_2-\mu v_2,v_0+\mu(v_1+v_2),v_1,v_2).
\end{align*}

\noindent La premi\`ere action correspond \`a l'homog\'en\'eit\'e de $F_{A,B},$ les deux derni\`eres \`a des actions par conjugaison lin\'eaire sur $F_{A,B}$
$$\left(\frac{1}{c}y,\frac{1}{c}z\right)F_{A,B}(cy,cz)=\left(z,\frac{a_0+a_1cy+a_2cz}{b_0c+b_1c^2y+b_2c^2z}\right)$$ et \begin{align*}
&(y-\mu,z-\mu)F_{A,B}(y+\mu,z+\mu)\\
&\hspace{1cm}=\left(z,\frac{a_0+\mu(a_1+a_2)-\mu(b_0+\mu(b_1+b_2))+(a_1-\mu b_1)y+(a_2-\mu b_2)z}{b_0+\mu(b_1+b_2)+b_1y+b_2z}\right).
\end{align*}

\noindent L'ensemble $\mathcal{V}_n$ est d\'efini par la condition $f_{a,b}^n(q)=p.$ Les coefficients des \'equations d\'efinissant $\mathcal{V}_n$ sont des entiers positifs; en particulier $\mathcal{V}_n$ est invariant par conjugaison complexe. En exploitant ces \'equations \textsc{Bedford} et \textsc{Kim} obtiennent la description suivante (\cite{BK1})
\begin{itemize}
\item $\mathcal{V}_0:$ orbite de $(a,b)=(0,0)$ sous l'action de $\mathfrak{a1}-\mathfrak{a3}.$

\item $\mathcal{V}_1:$ orbite de $(a,b)=(1,0)$ sous l'action de $\mathfrak{a1}-\mathfrak{a3}.$

\item $\mathcal{V}_2:$ orbite de $(a,b)=\left(\frac{1+\mathrm{i}}{2},\mathrm{i}\right).$

\item $\mathcal{V}_3:$ orbites de $\left\{\left(\frac{2-\sqrt{3}+\mathrm{i}}{2},\mathrm{i}\right),\,\left(\frac{2+\sqrt{3}+\mathrm{i}}{2},\mathrm{i}\right)\right\}$ et de leurs conjugu\'es.

\item $\mathcal{V}_4:$ orbites de $\{(a_0,b_0),\,(a_1,b_1),\,(a_2,b_2)\}$ et de leurs conjugu\'es o\`u $a_i$ (resp. $b_i$) est une racine de $1-3x+9x^2-24x^3+36x^4-27x^5+9x^6$ (resp. $1+6x^2+9x^4+3x^6$).

\item $\mathcal{V}_5:$ orbites de $\{(a_0,b_0),\,(a_1,b_1),\,(a_2,b_2),\,(a_3,b_3)\}$ et de leurs conjugu\'es o\`u $a_i$ (resp. $b_i$) est une racine de $1+3x^2-20x^3+49x^4-60x^5+37x^6-10x^7+x^8$ (resp. $1+7x^2+14x^4+8x^6+x^8$).

\item $\mathcal{V}_6:$ les \'equations de $\mathcal{V}_6$ sont divisibles par $b^2$ donc tous les couples de param\`etres $(a,0)$ avec $a\not\in~\{0,\,1\}$ appartiennent \`a $\mathcal{V}_6;$ de plus $\mathcal{V}_6$ contient les orbites de $$\left\{\left(\frac{3\pm\sqrt{5}+2\mathrm{i}\sqrt{\frac{5\pm\sqrt{5}}{2}}}{4},\mathrm{i}\sqrt{\frac{5\pm\sqrt{5}}{2}}\right)\right\}$$ et de leurs conjugu\'es. 
\end{itemize}

\chapter{D'autres exemples}\label{chapbedkim2}

\section{Familles continues d'automorphismes d'entropie positive}

\noindent Dans \cite{BK3} \textsc{Bedford} et \textsc{Kim} \'etudient la famille de transformations birationnelles d\'efinie dans la carte $x=1$ par \begin{align*}
&f(y,z)=\Big(z,-y+cz+\sum_{\stackrel{j=1}{j \text{ pair}}}^{k-2}\frac{a_j} {y^j}+\frac{1}{y^k}\Big), &&a_j\in\mathbb{C},\,\,\, c\in\mathbb{R},\,\,\,k\geq 2. 
\end{align*}

\subsection{Construction de surfaces rationnelles et d'automorphismes}

\noindent Les auteurs construi\-sent des surfaces en \'eclatant $\mathbb{P}^2(\mathbb{C})$ en un nombre fini de points sur lesquelles $f$ est conjugu\'e \`a un automorphisme d'entropie positive.

\begin{thm}\label{bk3}
{\sl Soient $j,$ $n$ deux entiers premiers entre eux tels que $1\leq j\leq n.$ Il existe un sous-ensemble non vide $C_n$ de $\mathbb{R}$ tel que, pour tout $k\geq 2$ pair et pour tout $(c,a_j)$ dans $C_n\times\mathbb{C},$ la transformation $f$ soit conjugu\'ee \`a un automorphisme sur une certaine surface rationnelle d'entropie $\log\lambda_{n,k}$ o\`u $\log\lambda_{n,k}$ est la plus grande racine du polyn\^ome $$\chi_{n,k}=1-k\displaystyle\sum_{j=1}^{n-1}x^j+x^n.$$}
\end{thm}

\noindent Disons un mot sur la construction de $C_n.$ La droite $\Delta=\{x=0\}$ est invariante par $f.$ \'Ecrivons les points de $\Delta\setminus\{(0:0:1)\}$ sous la forme $(0:1:w);$ alors $f(0:1:w) =\left(0:1:c-\frac{1}{w}\right).$ La restriction de $f$ \`a $\Delta$ co\"incide avec $g(w)=c-\frac{1}{w}.$ L'ensemble des valeurs de $c$ pour lesquelles~$g$ est p\'eriodique de p\'eriode $n$ est $$\{2\cos(j\pi/n)\,\vert\,0<j<n,\,(j,n)=1\}.$$ Posons $w_s=g^{s-1} (c)$ pour $1\leq s\leq n-1,$ autrement dit les $w_i$ codent l'orbite de $(0:1:0)$ sous l'action de $f.$ Les $w_j$ satisfont les propri\'et\'es suivantes
\begin{itemize}
\item $w_jw_{n-1-j}=1;$

\item si $n$ est pair $w_1\ldots w_{n-2}=1;$

\item si $n$ est impair, posons $w_*(c)=w_{(n-1)/2}$ auquel cas $w_1\ldots w_{n-2}=w_*.$
\end{itemize}


\bigskip

\noindent Nous allons d\'etailler l'\'eventualit\'e $n=3,$ $k=2,$ auquel cas $C_3=\{-1,\,1\}.$ Supposons que~$c=~1;$ autrement dit $$f=(xz^2:z^3:x^3+z^3-yz^2).$$ La transformation $f$ contracte une seule droite $\Delta''=\{z=0\}$ sur le point $R=(0:0:1)$ et \'eclate un seul point, $Q=(0:1:0).$

\noindent Dans un premier temps nous allons \'eclater $Q$ \`a la source et $R$ au but. Posons 
\begin{align*}
&\left\{\begin{array}{ll} x=u_1\\ z=u_1v_1\end{array}\right. && \begin{array}{ll} \mathrm{E}=\{u_1 =0\}\\ \Delta''_1=\{v_1=0\}\end{array} && \hspace{2cm} &&\left\{\begin{array}{ll} x=a_1\\ z=a_1b_1\end{array}\right. && \begin{array}{ll} \mathrm{F}=\{a_1 =0\}\\ \end{array}\\
&\left\{\begin{array}{ll} x=r_1s_1\\ z=s_1\end{array}\right. && \begin{array}{ll} \mathrm{E}=\{s_1 =0\}\\ \end{array} && \hspace{2cm} &&\left\{\begin{array}{ll} x=c_1d_1\\ z=d_1\end{array}\right. && \begin{array}{ll} \mathrm{F}=\{d_1 =0\}\\ \end{array}
\end{align*}

\noindent On a 
\begin{align*}
&(u_1,v_1)\to(u_1,u_1v_1)_{(x,z)}\to(u_1v_1^2:u_1v_1^3:u_1+u_1v_1^3-v_1^2)\\
&\hspace{1cm}=\left(\frac{u_1v_1^2}{u_1+u_1v_1^3-v_1^2},\frac{u_1v_1^3}{u_1+u_1v_1^3-v_1^2}\right)_{(x,y)}\to\left(\frac{u_1v_1^2}{u_1+u_1v_1^3-v_1^2},v_1\right)_{(a_1,b_1)}
\end{align*} 

\noindent et $$(r_1,s_1)\to(r_1s_1,s_1)_{(x,z)}\to(r_1s_1: s_1^2:r_1^3s_1+s_1-1).$$ Par suite $\mathrm{E}$ est envoy\'e sur $\mathrm{F},$ $\Delta''_1$ est contract\'e sur $S=(0,0)_{(a_1,b_1)}$ et $Q_1=(0,0)_{(u_1,v_1)}$ est un point d'ind\'etermination. 

\noindent Poursuivons en \'eclatant $Q_1$ \`a la source et $S$ au but
\begin{align*}
&\left\{\begin{array}{ll} u_1=u_2\\ v_1=u_2v_2\end{array}\right. && \begin{array}{ll} \mathrm{G}=\{u_2 =0\}\\ \end{array} && \hspace{2cm} &&\left\{\begin{array}{ll} a_1=a_2\\ b_1=a_2b_2\end{array}\right. && \begin{array}{ll} \mathrm{H}=\{a_2 =0\}\\ \end{array}\\
&\left\{\begin{array}{ll}u_1=r_2s_2\\ v_1=s_2\end{array}\right. && \begin{array}{ll} \mathrm{G}=\{s_2 =0\}\\  \end{array} && \hspace{2cm} &&\left\{\begin{array}{ll} a_1=c_2d_2\\ b_1=d_2\end{array}\right. && \begin{array}{ll} \mathrm{H}=\{d_2 =0\}\\ \end{array}
\end{align*}

\noindent On a dans les diff\'erents syst\`emes de coordonn\'ees
\begin{align*}
&(u_2,v_2)\to(u_2,u_2v_2)_{(u_1,v_1)}\to(u_2^2v_2^2:u_2^3v_2^3:1+u_2^3v_2^3-v_2^2)=\left(\frac{u_2^2v_2^2}{1+u_2^3v_2^3-v_2^2},\frac{u_2^3v_2^3}{1+u_2^3v_2^3-v_2^2}\right)_{(x,y)}\\
&\hspace{1cm}\to\left(\frac{u_2^2v_2^2}{1+u_2^3v_2^3-v_2^2},u_2v_2\right)_{(a_1,b_1)}\to\left(\frac{u_2v_2}{1+u_2^3v_2^3-v_2^2},u_2v_2\right)_{(c_2,d_2)};
\end{align*} ainsi $\mathrm{G}$ est contract\'e sur $T=(0,0)_{(c_2,d_2)}$ et $U=(0,0)_{(r_2,s_2)}$ est ind\'etermin\'e. De plus on a $$(u_1,v_1)\to \left(\frac{u_1v_1^2}{u_1+u_1v_1^3-v_1^2},v_1\right)\to\left(\frac{u_1v_1}{u_1+u_1v_1^3-v_1^2},v_1\right)_{(c_2,d_2)}$$ donc $\Delta''_2$ est contract\'e sur $T.$

\noindent Continuons en \'eclatant $U$ \`a la source et $T$ au but
\begin{align*}
&\left\{\begin{array}{ll} r_2=u_3\\ s_2=u_3v_3\end{array}\right. && \begin{array}{ll} \mathrm{K}=\{u_3 =0\}\\ \mathrm{G}_1=\{v_3=0\}\end{array} && \hspace{2cm} &&\left\{\begin{array}{ll} c_2=a_3\\ d_2=a_3b_3\end{array}\right. && \begin{array}{ll} \mathrm{L}=\{a_3 =0\}\\ \end{array}\\
&\left\{\begin{array}{ll}r_2=r_3s_3\\ s_2=s_3\end{array}\right. && \begin{array}{ll} \mathrm{K}=\{s_3 =0\}\\  \end{array} && \hspace{2cm} &&\left\{\begin{array}{ll} c_2=c_3d_3\\ d_2=d_3\end{array}\right. && \begin{array}{ll} \mathrm{L}=\{d_3 =0\}\\ \end{array}
\end{align*}

\noindent On constate que 
\begin{align*}
&(u_3,v_3)\to(u_3,u_3v_3)_{(r_2,s_2)}\to(u_3^2v_3^2:u_3^3v_3^3:1+u_3^3v_3^3-v_3)\\
&\hspace{1cm}=\left(\frac{u_3^2v_3^2}{1+u_3^3v_3^3-v_3},\frac{u_3^3v_3^3}{1+u_3^3v_3^3-v_3}\right)_{(x,y)}\to\left(\frac{u_3^2v_3^2}{1+u_3^3v_3^3-v_3},u_3v_3\right)_{(a_1,b_1)}\\
&\hspace{1cm}\to\left(\frac{u_3v_3}{1+u_3^3v_3^3-v_3},u_3v_3\right)_{(c_2,d_2)}\to\left(\frac{1}{1+u_3^3v_3^3-v_3},u_3v_3\right)_{(c_3,d_3)}
\end{align*}

\noindent et $$(r_3,s_3)\to(r_3s_3,s_3)_{(r_2,s_2)}\to(r_3s_3^2:r_3s_3^3:r_3+r_3s_3^3-1);$$ par cons\'equent $W=(1,0)_{(r_3,s_3)}$ est ind\'etermin\'e, $\mathrm{K}$ est envoy\'e sur $\mathrm{L}$ et $\mathrm{G}_1$ est contract\'e sur $V=(1,0)_{(c_3,d_3)}.$ Par ailleurs on note que $$(u_1,v_1) \to\left(\frac{u_1v_1}{u_1+u_1v_1^3-v_1^2},v_1\right)_{(c_2,d_2)}\to\left(\frac{u_1}{u_1+u_1v_1^3-v_1^2},v_1\right)_{(c_3,d_3)}$$ ainsi $\Delta''_3$ est contract\'e sur $V.$

\noindent \'Eclatons $W$ \`a la source et $V$ au but
\begin{align*}
&\left\{\begin{array}{ll} r_3=u_4+1\\ s_3=u_4v_4\end{array}\right. && \begin{array}{ll} \mathrm{M}=\{u_4 =0\}\\ \end{array} && \hspace{2cm} &&\left\{\begin{array}{ll} c_3=a_4+1\\ d_3=a_4b_4\end{array}\right. && \begin{array}{ll} \mathrm{N}=\{a_4 =0\}\\ \end{array}\\
&\left\{\begin{array}{ll}r_3=r_4s_4+1\\ s_3=s_4\end{array}\right. && \begin{array}{ll} \mathrm{M}=\{s_4 =0\}\\  \end{array} && \hspace{2cm} &&\left\{\begin{array}{ll} c_3=c_4d_4+1\\ d_3=d_4\end{array}\right. && \begin{array}{ll} \mathrm{N}=\{d_4 =0\}\\ \end{array}
\end{align*}

\noindent D'une part $$(u_1,v_1)\to\left(\frac{u_1}{u_1+u_1v_1^3-v_1^2},v_1\right)_{(c_3,d_3)}\to\left(\frac{u_1^2v_1 -v_1}{u_1+u_1v_1^3-v_1^2},v_1\right)_{(c_4,d_4)}$$ et $$(u_3,v_3)\to\left(\frac{1}{1+u_3^3v_3^3-v_3},u_3v_3\right)_{(c_3,d_3)}\to\left(\frac{1-u_3^3v_3}{u_3(1+u_3^3v_3^3-v_3)},u_3v_3\right)_{(c_4,d_4)}$$ dont on d\'eduit que $\Delta''_4$ est contract\'e sur $X=(0,0)_{(c_4,d_4)}$ et que $\mathrm{G}_1$ est envoy\'e sur $\mathrm{N}.$

\noindent D'autre part on a
\begin{align*}
&(u_4,v_4)\to(u_4+1,u_4v_4)_{(r_3,s_3)}\to(u_4v_4^2(u_4+1):u_4^3v_4^3(u_4+1):1+u_4^2v_4^3(u_4+1))\\
&\hspace{1cm}=\left(\frac{u_4v_4^2(u_4+1)}{1+u_4^2v_4^3(u_4+1)},\frac{u_4^3v_4^3(u_4+1)}{1+u_4^2v_4^3(u_4+1)}\right)_{(x,y)}\\ &\hspace{1cm}\to\left(\frac{u_4v_4^2(u_4+1)}{1+u_4^2v_4^3(u_4+1)},u_4v_4 \right)_{(a_1,b_1)}\to\left(\frac{v_4(u_4+1)}{1+u_4^2v_4^3(u_4+1)},u_4v_4\right)_{(c_2,d_2)}
\end{align*}

\noindent et $$(r_4,s_4)\to(r_4s_4+1,s_4)_{(r_3,s_3)}\to(s_4(r_4s_4+1):s_4^2(r_4s_4+1): r_4+(r_4s_4 +1)s_4^2);$$ il en r\'esulte que $Y=(0,0)_{(r_4, s_4)}$ est ind\'etermin\'e et $\mathrm{M}$ est envoy\'e sur $\mathrm{H}.$

\noindent Finalement \'eclatons $Y$ \`a la source et $X$ au but
\begin{align*}
&\left\{\begin{array}{ll} r_4=u_5\\ s_4=u_5v_5\end{array}\right. && \begin{array}{ll} \Lambda=\{u_5 =0\}\\ \end{array} && \hspace{2cm} &&\left\{\begin{array}{ll} c_4=a_5\\ d_4=a_5b_5\end{array}\right. && \begin{array}{ll} \Omega=\{a_5 =0\}\\ \end{array}\\
&\left\{\begin{array}{ll}r_4=r_5s_5\\ s_4=s_5\end{array}\right. && \begin{array}{ll} \Lambda=\{s_5 =0\}\\  \end{array} && \hspace{2cm} &&\left\{\begin{array}{ll} c_4=c_5d_5\\ d_4=d_5\end{array}\right. && \begin{array}{ll} \Omega=\{d_5 =0\}\\ \end{array}
\end{align*}

\noindent Tout d'abord remarquons que $\Delta''_5$ est envoy\'e sur $\Omega$ $$(u_1,v_1)\to \left(\frac{u_1v_1^2-v_1}{u_1+u_1v_1^3-v_1^2},v_1\right)_{(c_4,d_4)}\to\left(\frac{u_1v_1-1} {u_1+u_1v_1^3-v_1^2},v_1\right)_{(c_5,d_5)}.$$

\noindent Ensuite on a dans chacun des syst\`emes de coordonn\'ees $$(u_5,v_5)\to(u_5,u_5v_5)_{( r_4,s_4)}\to((u_5^2v_5+1)v_5:(u_5^2v_5+1)u_5v_5^2:1+(u_5^2v_5+1)u_5v_5^2)$$ et $$(r_5,s_5)\to(r_5s_5,s_5)_{(r_4,s_4)}\to(r_5s_5^2+1:s_5(r_5s_5^2+1):r_5+s_5(r_5s_5^2+1)).$$ En particulier on note que $\Lambda$ est envoy\'e sur $\Delta''_5.$ 

\noindent Notons $\widehat{P}_1$ (resp. $\widehat{P}_2$) le point infiniment proche obtenu en \'eclatant $Q,$ $Q_1,$ $U,$ $W$ et $Y$ (resp. $R,$ $S,$ $T,$ $V$ et $X$). Les calculs qui pr\'ec\`edent permettent d'affirmer que $f$ induit un isomorphisme entre $\mathrm{Bl}_{\widehat{P}_1}\mathbb{P}^2$ et $\mathrm{Bl}_{\widehat{P}_2}\mathbb{P}^2$ les composantes \'etant \'echang\'ees comme suit 
\begin{align*}
& \mathrm{E}\to\mathrm{F}, &&\Delta''\to\Omega, && \mathrm{K}\to\mathrm{L}, &&\mathrm{M}\to\mathrm{H}, && \Lambda\to\Delta'', &&\mathrm{G}\to\mathrm{N}.
\end{align*}

\noindent Pour qu'un conjugu\'e de $f$ soit un automorphisme d'entropie positive sur $\mathbb{P}^2(\mathbb{C})$ \'eclat\'e en $\ell$ points il faut que $\ell\geq 10;$ on cherche donc un automorphisme $A$ de $\mathbb{P}^2(\mathbb{C})$ tel que $(Af)^2A$ envoie $\widehat{P}_2$ sur $\widehat{P}_1.$ On remarque que $f(R)=(0:1:1)$ et $f^2(R)=Q$ puis que $f^2(\widehat{P}_2)=\widehat{P}_1$ donc~$A=\mathrm{id}$ convient. Les composantes sont \'echang\'ees de la fa\c{c}on suivante
\begin{align*}
& \Delta''\to f\Omega, &&\mathrm{E}\to f\mathrm{F}, &&\mathrm{G}\to f\mathrm{N}, &&\mathrm{K}\to f\mathrm{L}, && \mathrm{M}\to f\mathrm{H},\\
& \Lambda\to f\Delta'', &&f\mathrm{F}\to f^2\mathrm{F},&&f\mathrm{N}\to f^2\mathrm{N}, &&f\mathrm{L}\to f^2\mathrm{L},&&f\mathrm{H}\to f^2\mathrm{H},\\
&f\Omega\to f^2\Omega, &&f^2\mathrm{F}\to \mathrm{E},&&f^2\mathrm{N}\to \mathrm{G},&&f^2\mathrm{L}\to \mathrm{K},&&f^2\mathrm{H}\to \mathrm{M},\\
&f^2\Omega\to \Lambda.
\end{align*}

\noindent Il s'en suit que la matrice caract\'eristique de $f$ est donn\'ee dans la base $$ \{\Delta'',\,\mathrm{E},\,\mathrm{G},\,\mathrm{K},\,\mathrm{M},\,\Lambda,\,f\mathrm{F},\,f\mathrm{N},\, f\mathrm{L},\,f\mathrm{H},\,f\Omega,\,f^2\mathrm{F},\,f^2\mathrm{N},\,f^2\mathrm{L},\,f^2\mathrm{H},\,f^2\Omega\}$$ par 
$$\left[\begin{array}{cccccccccccccccc}
0 & 0 & 0 & 0 & 0 & 1  & 0 & 0 & 0 & 0 & 0 & 0 & 0 & 0 & 0 & 0\\  
0 & 0 & 0 & 0 & 0 & 1  & 0 & 0 & 0 & 0 & 0 & 1 & 0 & 0 & 0 & 0\\ 
0 & 0 & 0 & 0 & 0 & 2  & 0 & 0 & 0 & 0 & 0 & 0 & 1 & 0 & 0 & 0\\ 
0 & 0 & 0 & 0 & 0 & 3  & 0 & 0 & 0 & 0 & 0 & 0 & 0 & 1 & 0 & 0\\ 
0 & 0 & 0 & 0 & 0 & 3  & 0 & 0 & 0 & 0 & 0 & 0 & 0 & 0 & 1 & 0\\ 
0 & 0 & 0 & 0 & 0 & 3  & 0 & 0 & 0 & 0 & 0 & 0 & 0 & 0 & 0 & 1\\ 
0 & 1 & 0 & 0 & 0 & -1 & 0 & 0 & 0 & 0 & 0 & 0 & 0 & 0 & 0 & 0\\ 
0 & 0 & 0 & 0 & 1 & -3 & 0 & 0 & 0 & 0 & 0 & 0 & 0 & 0 & 0 & 0\\ 
0 & 0 & 0 & 1 & 0 & -3 & 0 & 0 & 0 & 0 & 0 & 0 & 0 & 0 & 0 & 0\\ 
0 & 0 & 1 & 0 & 0 & -2 & 0 & 0 & 0 & 0 & 0 & 0 & 0 & 0 & 0 & 0\\ 
1 & 0 & 0 & 0 & 0 & -3 & 0 & 0 & 0 & 0 & 0 & 0 & 0 & 0 & 0 & 0\\ 
0 & 0 & 0 & 0 & 0 & 0  & 1 & 0 & 0 & 0 & 0 & 0 & 0 & 0 & 0 & 0\\ 
0 & 0 & 0 & 0 & 0 & 0  & 0 & 1 & 0 & 0 & 0 & 0 & 0 & 0 & 0 & 0\\ 
0 & 0 & 0 & 0 & 0 & 0  & 0 & 0 & 1 & 0 & 0 & 0 & 0 & 0 & 0 & 0\\ 
0 & 0 & 0 & 0 & 0 & 0  & 0 & 0 & 0 & 1 & 0 & 0 & 0 & 0 & 0 & 0\\ 
0 & 0 & 0 & 0 & 0 & 0  & 0 & 0 & 0 & 0 & 1 & 0 & 0 & 0 & 0 & 0\\ 
\end{array}\right];$$ la plus grande racine du polyn\^ome caract\'eristique  $$(X^2-3X+1)(X^2-X+1)(X+1)^2(X^2+X+1)^3(X-1)^4$$ de celle-ci est $\frac{3+\sqrt{5}}{2},$ {\it i.e.} le premier degr\'e dynamique de $f$ est $\frac{3+\sqrt{5}}{2}.$ Remarquons que le polyn\^ome $\chi_{3,2}$ introduit dans le Th\'eor\`eme \ref{bk3} est $1-2X-2X^2+X^3$ dont la plus grande racine est $\frac{3+\sqrt{5}}{2}.$ 

\subsection{Une famille de syst\`emes dynamiques}

\noindent Dans cette section nous allons voir un r\'esultat permettant d'\'etablir que la famille de transformations birationnelles consid\'er\'ee n'est pas triviale.

\begin{thm}[\cite{BK3}]
{\sl \`A $c$ fix\'e, la famille de transformations d\'efinies par
\begin{align*}
&f(y,z)=\Big(z,-y+cz+\sum_{\stackrel{j=1}{j \text{ pair}}}^{k-2}\frac{a_j} {y^j}+\frac{1}{y^k}\Big), && a_j\in\mathbb{C},\, c\in\mathbb{R},k\geq 2. 
\end{align*} induit une famille de syst\`emes dynamiques de dimension $k/2-1.$}
\end{thm}

\noindent La strat\'egie est la suivante. Un tel $f$ poss\`ede $k+1$ points fixes que nous noterons $p_1,\,\ldots,\,p_{k+1}.$ Posons $a=(a_1,\ldots,a_{k-2}).$ \textsc{Bedford} et \textsc{Kim} montrent que les valeurs propres de $Df_a$ au point~$p_j(a)$ d\'ependent de $a;$ il s'en suit que la famille varie avec $a$ de mani\`ere non triviale. Plus pr\'ecis\'ement ils montrent que la trace de $Df_{a}$ varie de mani\`ere non triviale. D\'esignons par~$\tau_j(a)$ la trace de la diff\'erentielle $Df_{a}$ au point $p_j(a)$ et consid\'erons l'application $T$ d\'efinie par 
\begin{align*}
& a\to T(a)=(\tau_1(a),\ldots,\tau_{k+1}(a)).
\end{align*}

\subsubsection{Premi\`ere \'etape.} Le rang de l'application $T$ est $\frac{k}{2}-1$ au point $a=0.$ En effet les points fixes de $f$ sont de la forme $(\xi_s,\xi_s)$ o\`u $\xi_s$ est une racine de 
\begin{equation}\label{ptsfixes}
\xi=(c-1)\xi+\sum_{\stackrel{j=1}{j \text{ pair}}}^{k-2}\frac{a_j}{\xi^j}+\frac{1}{\xi^k}.
\end{equation}
\noindent Lorsque $a$ est nul, on a pour tout point fixe $\xi^{k+1}=\frac{1}{2-c}.$ En diff\'erenciant (\ref{ptsfixes}) par rapport aux~$a_\ell$ on obtient pour $a=0$ l'\'egalit\'e $$\left(2-c+\frac{k}{\xi^{k+1}}\right)\frac{\partial\xi}{\partial a_\ell}=\frac{1}{\xi^\ell}$$ dont on d\'eduit $$\frac{\partial\xi}{\partial a_\ell}\Big\vert_{a=0}=\frac{1}{(2-c)(k+1)\xi^\ell}.$$ La trace de $Df_{(y,z)}$ est donn\'ee par $$\tau=c-\sum_{\stackrel{j=1}{j \text{ pair}}}^{k-2}\frac{ja_j}{y^{j+1}}-\frac{k}{y^{k+1}}.$$ Pour $y=\xi_a$ on a 
$$\frac{\partial\tau(\xi_a)}{\partial a_\ell}\Big\vert_{a=0}=-\frac{\ell}{y^{\ell+1}}+\frac{k(k+1)}{y^{k+2}}\frac{\partial\xi_a}{\partial a_\ell}=-\frac{\ell}{y^{\ell+1}}+\frac{k}{2-c}\frac{1}{\xi^{k+1}\xi^{\ell+1}}=-\frac{\ell}{y^{\ell+1}}+\frac{k}{y\xi^\ell}=\frac{k-\ell}{\xi^{\ell+1}}.$$

\noindent Si $\xi_j$ parcourt les $\frac{k}{2}-1$ racines distinctes de $\frac{1}{(2-c)^{k+1}},$ la matrice est essentiellement une matrice de Vandermonde de taille $(\frac{k}{2}-1)\times(\frac{k}{2}-1)$ donc de rang $\frac{k}{2}-1.$

\subsubsection{Deuxi\`eme \'etape.} D\'esignons par $f_a$ une transformation du type
\begin{align*}
&\Big(z,-y+cz+\sum_{\stackrel{j=1}{j \text{ pair}}}^{k-2}\frac{a_j} {y^j}+\frac{1}{y^k}\Big), && a_j\in\mathbb{C},\, c\in\mathbb{R},k\geq 2. 
\end{align*} 

\noindent Il existe un voisinage $\mathcal{U}$ de $0$ dans $\mathbb{C}^{\frac{k}{2}-1}$ tel que pour tous $a',$ $a''$ dans $\mathcal{U}$ avec $a'\not=a''$ l'application $f_{a'}$ n'est pas diff\'eomorphe \`a $f_{a''}.$ En effet, d'apr\`es ce qui pr\'ec\`ede l'application $\mathbb{C}^{\frac{k}{2}-1}\to\mathbb{C}^{k+1},$ $a\mapsto T(a)$ est localement injective au voisinage de $0.$ De plus pour $a$ nul les points fixes $p_1,$ $\ldots,$ $p_{k+1},$ et donc les valeurs $\tau_1(0),$ $\ldots,$ $\tau_{k+1}(0)$ sont distinctes. Par suite $\mathbb{C}^{\frac{k}{2}-1}\ni a \mapsto\{ \tau_1(a),\ldots,\tau_{k+1}(a)\}$ est localement injective en $0.$ Ainsi si $\mathcal{U}$ est un voisinage suffisamment petit de $0$ et si $a'$ et $a''$ sont deux \'el\'ements distincts de $\mathcal{U}$ les ensembles de multiplicateurs aux points fixes ne sont pas les m\^emes; il s'en suit que $f_{a'}$ et~$f_{a''}$ ne sont pas diff\'eomorphes.

\subsection{D'autres r\'esultats}

\noindent Soient $f$ une transformation satisfaisant le Th\'eor\`eme \ref{bk3}; notons~$\mathcal{Z}_f$ la surface sur laquelle $f$ est conjugu\'ee \`a un automorphisme que l'on notera encore~$f$ pour simplifier. \`A l'aide de ce m\^eme th\'eor\`eme \textsc{Bedford} et \textsc{Kim} d\'eterminent les courbes de~$\mathcal{Z}_f$ invariantes par $f$ (\emph{voir} \cite{BK3}, Th\'eor\`eme 3.5). Ils en d\'eduisent que si $n>2,$ alors $(f,\mathcal{Z}_f)$ est minimal; de plus si~$n=2$ il est minimal apr\`es avoir contract\'e $\Delta$ (\emph{voir} \cite{BK3}, Th\'eor\`eme 3.6). 

\noindent \`A chaque $f$ est donc associ\'ee une surface rationnelle $\mathcal{Z}=\mathcal{Z}_f$ sur laquelle $f$ est conjugu\'ee \`a un automorphisme. Une question naturelle est la suivante: que peut-on dire de $\mathrm{Aut}(\mathcal{Z})$ ?  Rappelons que $\mathrm{cr}$ est la repr\'esentation de $\mathrm{Aut}(\mathcal{Z})$ dans $\mathrm{GL}(\mathrm{Pic}(\mathcal{Z}))$ d\'efinie par (cf \S\ref{wweyl}) 
\begin{align*}
&\mathrm{cr}\colon\mathrm{Aut}(\mathcal{Z})\to\mathrm{GL}(\mathrm{Pic}(\mathcal{Z})), && g\mapsto g_*.
\end{align*}

\noindent \textsc{Bedford} et \textsc{Kim} montrent que dans tous les cas qu'ils consid\`erent $\mathrm{cr}$ est au plus $((k^2-1):1);$ de plus si $a_{k-2}$ est non nul, $\mathrm{cr}$ est fid\`ele. Lorsque $n$ vaut $2,$ ils obtiennent le r\'esultat suivant plus pr\'ecis.

\begin{thm}[\cite{BK3}]
{\sl Supposons que $n$ soit \'egal \`a $2$ et que \begin{align*}
&f(y,z)=\Big(z,-y+cz+\sum_{\stackrel{j=1}{j \text{ pair}}}^{k-2}\frac{a_j} {y^j}+\frac{1}{y^k}\Big), && a_j\in\mathbb{C},\, c\in\mathbb{R},k\leq 2. 
\end{align*} soit conjugu\'e \`a un automorphisme sur une surface rationnelle d'entropie $\log\lambda_{n,k}$ o\`u $\log\lambda_{n,k}$ est la plus grande racine du polyn\^ome $\chi_{n,k}=1-k\displaystyle\sum_{j=1}^{n-1}x^j+x^n.$ 

\noindent Soient $\iota$ la r\'eflexion $(x,y)\mapsto (y,x)$ et  $\mathrm{cr}\colon\mathrm{Aut}(\mathcal{Z})\to\mathrm{GL}(\mathrm{Pic}(\mathcal{Z})),$ $g\mapsto g_*.$ 

\noindent L'image de $\mathrm{cr}$ est le sous-groupe de $\mathrm{GL}(\mathrm{Pic}(\mathcal{Z}))$ constitu\'e des isom\'etries qui pr\'eservent \`a la fois la classe canonique de $\mathcal{Z}$ et l'ensemble des diviseurs effectifs; ce sous-groupe est le groupe di\'edral infini de g\'en\'erateurs $f_*$ et $\iota_*.$}
\end{thm}

\noindent Ils en d\'eduisent, toujours pour $n=2,$ que g\'en\'eriquement les surfaces $\mathcal{Z}_a$ ne sont pas biholomorphiquement \'equivalentes.

\begin{thm}[\cite{BK3}]
{\sl Supposons que $n$ soit \'egal \`a $2,$ que $k$ soit pair et sup\'erieur ou \'egal \`a $4.$ Soient $a$ dans $\mathbb{C}^{k/2-1}$ et $f_a$ une transformation du type \begin{align*}
&\Big(z,-y+cz+\sum_{\stackrel{j=1}{j \text{ pair}}}^{k-2}\frac{a_j} {y^j}+\frac{1}{y^k}\Big), && a_j\in\mathbb{C},\, c\in\mathbb{R},k\geq 2. 
\end{align*} qui soit conjugu\'ee \`a un automorphisme sur une surface rationnelle $\mathcal{Z}_a.$ 

\noindent Il existe un voisinage $\mathcal{U}$ de $0$ dans $\mathbb{C}^{k/2-1}$ tel que si $a,$ $a'$ sont deux points distincts de $\mathcal{U}$ et si~$a_{k-1}$ est non nul, alors $\mathcal{Z}_a$ n'est pas biholomorphiquement \'equivalent \`a $\mathcal{Z}_{a'}.$}
\end{thm}

\section{Dynamique des automorphismes d'entropie positive: domaines de rotation}

\noindent Comme on l'a dit pr\'ec\'edemment si $\mathcal{Z}$ est une surface complexe, compacte poss\'edant un automorphisme $f$ d'entropie topologique non nulle, un th\'eor\`eme de \textsc{Cantat} assure qu'ou bien la dimension de \textsc{Kodaira} de $\mathcal{Z}$ est nulle et dans ce cas $f$ est conjugu\'e \`a un automorphisme de l'unique mod\`ele minimal de $\mathcal{Z}$ qui doit \^etre un tore, une surface K$3$ ou une surface d'\textsc{Enriques}; ou bien la surface $\mathcal{Z}$ est rationnelle non minimale et $f$ est birationnellement conjugu\'ee \`a une transformation birationnelle du plan (\cite{Can1}). On a aussi vu que si $\mathcal{Z}$ est un tore complexe, l'ensemble de \textsc{Fatou} de $f$ est vide. Si $\mathcal{Z}$ est une surface K$3$ ou un quotient d'une surface K$3,$ l'existence d'une forme volume invariante implique que les seules composantes de \textsc{Fatou} possibles sont les domaines de rotation. \textsc{McMullen} a montr\'e l'existence de surfaces K$3$ non alg\'ebriques avec des domaines de rotation de rang $2$ (\emph{voir} \cite{Mc2}). Qu'en est-il si $\mathcal{Z}$ est une surface rationnelle non minimale ? Les automorphismes d'entropie positive sur les surfaces rationnelles non minimales peuvent poss\'eder de \og grands\fg\, domaines de rotation.

\subsection{\'Enonc\'es}

\begin{thm}[\cite{BK4}]\label{rot1}
{\sl Il existe une surface rationnelle $\mathcal{Z}$ poss\'edant un automorphisme d'entropie positive $h$ et un dommaine de rotation $\mathcal{U}.$ De plus, $\mathcal{U}$ est une union de disques de \textsc{Siegel} invariants sur chacun desquels $h$ agit comme une rotation irrationnelle.}
\end{thm}

\noindent La lin\'earisation est un outil remarquable pour prouver l'existence de domaines de rotations mais c'est une technique local. Afin de comprendre la nature globale de la composante de \textsc{Fatou} $\mathcal{U},$ \textsc{Bedford} et \textsc{Kim} introduisent un mod\`ele global et obtiennent le r\'esultat suivant.

\begin{thm}[\cite{BK4}]\label{rot2}
{\sl Il existe une surface $\mathcal{L}$ obtenue en \'eclatant $\mathbb{P}^2(\mathbb{C})$ en un nombre fini de points, un automorphisme $L$ sur $\mathcal{L},$ un domaine $\Omega$ de $\mathcal{L}$ et une conjugaison biholomorphe $\Phi\colon\mathcal{U}\to\Omega$ qui envoie $(h,\mathcal{U})$ sur $(L,\mathcal{L}).$ 

\noindent En particulier, $h$ n'a pas de point p\'eriodique sur $\mathcal{U}\setminus\{z=0\}.$ }
\end{thm}

\noindent Les auteurs introduisent pour $n,$ $m\geq 1$ le polyn\^ome $$\chi_{n,m}(t)=\frac{t(t^{nm}-1)(t^n-2t^{n-1}+1)}{(t^n-1)(t-1)}+1.$$ Si $n\geq 4,$ $m\geq 1$ ou si $n=3,$ $m\geq 2$ ce polyn\^ome est un polyn\^ome de \textsc{Salem}. Soit $\delta$ une racine de $\chi_{n,m}$ qui ne soit pas une racine de l'unit\'e. Pour $1\leq j\leq n-1,$ $(j,n)=1$ po\-sons~$$c=2\sqrt{\delta} \cos(j\pi/n)$$ et consid\'erons la transformation birationnelle $f$ donn\'ee dans la carte affine $z=1$ par $$f(x,y)=\left(y,-\delta x+cy+\frac{1}{y}\right).$$ Les auteurs montrent le r\'esultat suivant.

\begin{thm}[\cite{BK4}]\label{rot3}
{\sl Il existe une surface rationnelle $\mathcal{Z}$ obtenue en \'eclatant $\mathbb{P}^2(\mathbb{C})$ en un nombre fini de points $\pi\colon\mathcal{Z}\to\mathbb{P}^2(\mathbb{C})$ telle que $\pi^{-1} f\pi$ soit un automorphisme de $\mathcal{Z}.$

\noindent De plus, l'entropie de $f$ est la plus grande racine du polyn\^ome $\chi_{n,m}.$}
\end{thm}

\noindent \textsc{Bedford} et \textsc{Kim} utilisent ensuite le couple $(f^k,\mathcal{Z})$ pour d\'emontrer les \'enonc\'es \ref{rot1} et \ref{rot2}. 

\subsection{Construction des automorphismes et des surfaces}

\noindent Nous allons d\'etailler la construction de la surface $\mathcal{Z}$ \'evoqu\'ee dans le Th\'eor\`eme \ref{rot3}.

\noindent La transformation $f$ s'\'ecrit en coordonn\'ees homog\`enes $(y^2:-\delta xy+cy^2+z^2: yz);$ elle \'eclate exactement un point $R=(1:0:0)$ et contracte exactement une droite $\Delta'= \{y=0\}$ sur le point $Q=(0:1:0).$

\noindent Commen\c{c}ons par \'eclater $R$ \`a gauche et $Q$ \`a droite 
\begin{align*}
&\left\{\begin{array}{ll} y=u_1\\ z=u_1v_1\end{array}\right. && \begin{array}{ll} \mathrm{E}=\{u_1 =0\}\\ \end{array} && \hspace{2cm} &&\left\{\begin{array}{ll} x=a_1\\ z=a_1b_1\end{array}\right. && \begin{array}{ll} \mathrm{F}=\{a_1 =0\}\\ \end{array}
\end{align*}

\begin{align*}
&\left\{\begin{array}{ll} y=r_1s_1\\ z=s_1\end{array}\right. && \begin{array}{ll} \mathrm{E}=\{s_1 =0\}\\ \Delta'_1=\{r_1=0\} \end{array} && \hspace{2cm} &&\left\{\begin{array}{ll} x=c_1d_1\\ z=d_1\end{array}\right. && \begin{array}{ll} \mathrm{F}=\{d_1 =0\}\\ \end{array}
\end{align*}

\noindent On a d'une part 
\begin{align*}
&(u_1,v_1)\to(u_1,u_1v_1)_{(y,z)}\to(u_1:-\delta+cu_1+u_1v_1^2:u_1v_1)\\ &\hspace{1cm}=\left(\frac{u_1}{-\delta+cu_1+u_1v_1^2},\frac{u_1v_1}{-\delta+cu_1+u_1v_1^2}\right)_{(x,z)}\to\left(\frac{u_1}{-\delta+cu_1+u_1v_1^2},v_1\right)_{(a_1,b_1)}
\end{align*}
\noindent et d'autre part 
\begin{align*}
&(r_1,s_1)\to (r_1s_1,s_1)_{(y,z)}\to(r_1^2s_1:-\delta r_1+cr_1^2s_1^2+s_1:r_1s_1)\\
&\hspace{1cm}=\left( \frac{r_1^2s_1}{-\delta r_1+cr_1^2s_1^2+s_1},\frac{r_1s_1}{-\delta r_1+cr_1^2s_1^2+s_1}\right)_{(x,z)}\to\left( r_1,\frac{r_1s_1}{-\delta r_1+cr_1^2s_1^2+s_1}\right)_{(c_1,d_1)}.
\end{align*}

\noindent Il s'en suit que $P=(0,0)_{( r_1,s_1)}$ est d'ind\'etermination, $\Delta'_1$ est contract\'e sur $S=(0,0)_{(c_1,d_1)}$ et~$\mathrm{E}$ est envoy\'e sur $\mathrm{F}.$

\noindent D\'esormais on \'eclate $P$ \`a gauche et $S$ \`a droite
\begin{align*}
&\left\{\begin{array}{ll} r_1=u_2\\ s_1=u_2v_2\end{array}\right. && \begin{array}{ll} \mathrm{G}=\{u_2 =0\}\\ \end{array} && \hspace{1cm} &&\left\{\begin{array}{ll} c_1=a_2\\ d_1=a_2b_2\end{array}\right. && \begin{array}{ll} \mathrm{H}=\{a_2=0\}\\ \end{array}
\end{align*}

\begin{align*}
&\left\{\begin{array}{ll} r_1=r_2s_2\\ s_1=s_2\end{array}\right. && \begin{array}{ll} \mathrm{G}=\{s_2 =0\}\\  \end{array} && \hspace{1cm} &&\left\{\begin{array}{ll} c_1=c_2d_2\\ d_1=d_2\end{array}\right. && \begin{array}{ll} \mathrm{H}=\{d_2 =0\}\\ \end{array}
\end{align*}

\noindent On constate que 
\begin{align*}
&(u_2,v_2)\to(u_2,u_2v_2)_{(r_1,s_1)}\to(u_2^2v_2:-\delta+cu_2^3v_2^2 +v_2:u_2v_2)\\
&\hspace{1cm}=\left(\frac{u_2^2v_2}{-\delta+cu_2^3v_2^2+v_2},\frac{u_2v_2}{-\delta+cu_2^3v_2^2+v_2}\right)_{(x,z)}\to\left(u_2,\frac{u_2v_2}{-\delta+cu_2^3v_2^2+v_2}\right)_{(c_1,d_1)}\\
&\hspace{1cm}\to\left( \frac{-\delta+c_u2^3v_2^2+v_2}{v_2},\frac{u_2v_2}{-\delta+c_u2^3v_2^2+v_2}\right)_{(c_2,d_2)}
\end{align*}

\noindent et 

\begin{align*}
&(r_2,s_2)\to(r_2s_2,s_2)_{(r_1,s_1)}\to(r_2^2s_2^2:-\delta r_2+cr_2^2s_2^3+1:r_2s_2)\\
&\hspace{1cm}
=\left( \frac{r_2^2s_2^2}{-\delta r_2+cr_2^2s_2^3+1},\frac{r_2s_2}{-\delta r_2+cr_2^2s_2^3+1}\right)_{(x, z)}\to\left(r_2s_2,\frac{r_2s_2}{-\delta r_2+cr_2^2s_2^3+1}\right)_{(c_1,d_1)}\\
&\hspace{1cm}\to\left(-\delta r_2+cr_2^2s_2^3+1,\frac{r_2s_2}{-\delta r_2+cr_2^2s_2^3+1}\right)_{(c_2,d_2)}.
\end{align*}

\noindent Ainsi $\mathrm{G}$ est envoy\'e sur $\mathrm{H}$ et $T=(0,\delta)_{(u_2,v_2)}$ est d'ind\'etermination.

\noindent De plus 
\begin{align*}
&(y,z)\to\left(\frac{y^2}{-\delta y+cy^2+z^2},\frac{yz}{-\delta y+cy^2+z^2} \right)_{(x,z)}\to\left(\frac{y}{z},\frac{yz}{-\delta y+cy^2+z^2}\right)_{(c_1,d_1)}\\
&\hspace{1cm}\to \left(\frac{-\delta y+cy^2+z^2}{z^2},\frac{yz}{-\delta y+cy^2+z^2}\right)_{(c_2,d_2)}
\end{align*}

\noindent ce qui montre que $\Delta'_2$ est contract\'e sur $U=(1,0)_{(c_2,d_2)}.$

\noindent Pour finir \'eclatons $T$ \`a gauche et $U$ \`a droite
\begin{align*}
&\left\{\begin{array}{ll} u_2=u_3\\ v_2=u_3v_3+\delta\end{array}\right. && \begin{array}{ll} \mathrm{K}=\{u_3=0\}\\ \end{array} && \hspace{2cm} &&\left\{\begin{array}{ll} c_2=a_3+1\\ d_2=a_3b_3\end{array}\right. && \begin{array}{ll} \mathrm{L}=\{a_3=0\}\\ \end{array}
\end{align*}

\begin{align*}
&\left\{\begin{array}{ll} u_2=r_3s_3\\ v_2=s_3+\delta\end{array}\right. && \begin{array}{ll} \mathrm{G}=\{s_3=0\}\\  \end{array} && \hspace{2cm} &&\left\{\begin{array}{ll} c_2=c_3d_3+1\\ d_2=d_3\end{array}\right. && \begin{array}{ll} \mathrm{L}=\{d_3 =0\}\\ \end{array}
\end{align*}

\noindent On a 
\begin{align*}
&(u_3,v_3)\to(u_3,u_3v_3+\delta)_{(u_2,v_2)}\to(u_3(u_3v_3+\delta):v_3+cu_3^2(u_3v_3+\delta)^2: u_3v_3+\delta)\\
&\hspace{1cm}=\left(\frac{u_3(u_3v_3+\delta)}{v_3+cu_3^2(u_3v_3+\delta)^2},\frac{u_3v_3+\delta}{v_3+cu_3^2(u_3v_3+\delta)^2}\right)_{(y,z)}\to\left(u_3,\frac{u_3v_3+\delta}{v_3+cu_3^2(u_3v_3+\delta)^2}\right)_{(c_1,d_1)}
\end{align*}

\noindent et $$(r_3,s_3)\to(r_3s_3,s_3+\delta)_{(u_2,v_2)}\to(r_3^2s_3(s_3+\delta):1+cr_3^3s_3^2( s_3+\delta)^2:r_3(s_3+\delta)).$$

\noindent Par suite $\mathrm{K}$ est envoy\'e sur $\Delta=\{x=0\}.$

\noindent Enfin
\begin{small}
$$(y,z)\to \left(\frac{-\delta y+cy^2+z^2}{z^2},\frac{yz}{-\delta y+cy^2+z^2}\right)_{(c_2,d_2)} \to\left(\frac{cy^2-\delta y}{z^2},\frac{z^3}{(cy-\delta)(-\delta y+cy^2+z^2)}\right)_{(a_3,b_3)}$$
\end{small}

\noindent en particulier $\Delta'_3$ est envoy\'e sur $\mathrm{L}.$

\noindent On peut donc \'enoncer le r\'esultat suivant.

\begin{pro}
{\sl La transformation $f$ induit un isomorphisme entre $\mathrm{Bl}_{R,P,T}\mathbb{P}^2$ et $\mathrm{Bl}_{Q,S,U}\mathbb{P}^2.$}
\end{pro}

\noindent Les diff\'erentes composantes sont \'echang\'ees comme suit
\begin{align*}
&\mathrm{E}\to\mathrm{F}, &&\mathrm{G}\to\mathrm{H}, && \mathrm{K}\to\Delta, &&\Delta'\to\mathrm{L}.
\end{align*}

\noindent Supposons que $n=4,$ que $\delta$ soit une racine de $\chi_{4,1}$ et que $c=2 \sqrt{\delta}\cos(\pi/4)=\sqrt{2}\sqrt{\delta}$ alors~$f^4(Q)=R$ donc quitte \`a \'eclater $f^i(Q),$ $f^i(S)$ et $f^i(U)$ pour $i=0,\ldots,3$ la transformation~$f$ est conjugu\'ee \`a un automorphisme sur $\mathbb{P}^2(\mathbb{C})$ \'eclat\'e en ces $12$ points. On a 
\begin{align*}
& \Delta''\to f\mathrm{L}, &&\mathrm{E}\to f\mathrm{F}, &&\mathrm{G}\to f\mathrm{H}, &&\mathrm{K}\to f\Delta, &&f\mathrm{F}\to f^2\mathrm{F},\\
&f\mathrm{H}\to f^2\mathrm{H}, &&f\mathrm{L}\to f^2\mathrm{L},&&f^2\mathrm{F}\to f^3\mathrm{F},
&&f^2\mathrm{H}\to f^3\mathrm{H}, &&f^2\mathrm{L}\to f^3 \mathrm{L},\\
&f^3\mathrm{F}\to \mathrm{E},&&f^3\mathrm{H}\to \mathrm{G},
&&f^3\mathrm{L}\to \mathrm{K}.
\end{align*}

\noindent Autrement dit la matrice de $f^*$ dans la base $$\{\Delta',\,\mathrm{E},\,\mathrm{G},\,\mathrm{K},\, f\mathrm{F},\,f\mathrm{H},\,f\mathrm{L},\,f^2\mathrm{F},\,f^2\mathrm{H},\,f^2\mathrm{L},\,f^3\mathrm{F},\,f^3\mathrm{H},\,f^3\mathrm{L}\}$$ est donn\'ee par 
$$\left[\begin{array}{ccccccccccccc}
0&0&0&1&0&0&0&0&0&0&0&0&0\\
0&0&0&1&0&0&0&0&0&0&1&0&0\\
0&0&0&2&0&0&0&0&0&0&0&1&0\\
0&0&0&2&0&0&0&0&0&0&0&0&1\\
0&1&0&-1&0&0&0&0&0&0&0&0&0\\
0&0&1&-2&0&0&0&0&0&0&0&0&0\\
1&0&0&-2&0&0&0&0&0&0&0&0&0\\
0&0&0&0&1&0&0&0&0&0&0&0&0\\
0&0&0&0&0&1&0&0&0&0&0&0&0\\
0&0&0&0&0&0&1&0&0&0&0&0&0\\
0&0&0&0&0&0&0&1&0&0&0&0&0\\
0&0&0&0&0&0&0&0&1&0&0&0&0\\
0&0&0&0&0&0&0&0&0&1&0&0&0
\end{array}\right];$$

\noindent le polyn\^ome caract\'eristique de cette matrice s'\'ecrit $(X+1)^2(X-1)^3(X^2+1)^2 \chi_{4,1}.$ La plus grande valeur propre est en module proche de $1,7.$

\chapter{Une autre fa\c{c}on de construire des automorphismes d'entropie positive}\label{chapdg}

\noindent Dans \cite{DG} les auteurs s'int\'eressent, entre autres, \`a la transformation \begin{align*}
&\Phi_n=(xz^{n-1}+y^n:yz^{n-1}:z^n), &&n\geq 3.
\end{align*}

\noindent Dans ce chapitre nous allons reprendre quelques-uns de leurs r\'esultats; pour simplifier les calculs nous supposerons que $n=3$ et on d\'esignera par $\Phi$ la transformation $\Phi_3.$ On constate que la suite $(\deg\Phi^n)_{n\in \mathbb{N}}$ est born\'ee ($\Phi^k=(x+ky^3,y)$ dans la carte affine $z=1$), donc $\Phi$ est conjugu\'e \`a un automorphisme sur une certaine surface rationnelle~$\mathcal{Z}$ et un it\'er\'e de $\Phi$ est conjugu\'e \`a un automorphisme isotope \`a l'identit\'e. La transformation $\Phi$ \'eclate un unique point $P=(1:0:0)$ et contracte une unique droite $\Delta=\{z=0\}.$

\section{Construction}

\noindent La premi\`ere \'etape est la construction de deux points $\widehat{P}_1$ et
$\widehat{P}_2$ infiniment proches de $P$ tels que~$\Phi$ induise
un isomorphisme entre $\mathbb{P}^2(\mathbb{C})$ \'eclat\'e en $\widehat{P}_1$ et $\mathbb{P}^2(\mathbb{C})$ \'eclat\'e en $\widehat{P}_2.$
La seconde \'etape consiste \`a trouver des automorphismes $\varphi$ de $\mathbb{P}^2(\mathbb{C})$
tels que $\widehat{P}_1,$ $\varphi\widehat{P}_2,$
$\varphi\Phi\varphi\widehat{P}_2$ soient \`a supports distincts et $\widehat{P}_1=(\varphi\Phi)^2\varphi\widehat{P}_2$ donc tels que $(\varphi\Phi)^2$ induise un automorphisme de $\mathbb{P}^2(\mathbb{C})$ \'eclat\'e en $\widehat{P}_1,$ $\varphi\widehat{P}_2,$ $\varphi\Phi\varphi\widehat{P}_2.$

\subsection{D\'efinitions de $\widehat{P}_1,$ $\widehat{P}_2$}\label{cons}

\noindent Commen\c{c}ons par \'eclater le point $P:$

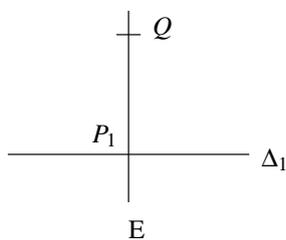
\begin{figure}[H]
\begin{center}
\input{ecl1.pstex_t}
\end{center}
\end{figure}

\noindent Posons $y=u_1$ et $z=u_1v_1;$ on remarque que $(u_1,v_1)$ sont des coordonn\'ees au voisinage de~$P_1=(0,0)_{(u_1,v_1)}$ dans lesquelles le diviseur exceptionnel est $\mathrm{E}=\{u_1=0\}$ et la transform\'ee de $\Delta$ est donn\'ee par $\Delta_1=\{v_1=0\}.$ Posons $y=r_1s_1$ et $z=s_1;$ notons que $(r_1,s_1)$ sont des coordonn\'ees au voisinage de $Q=(0,0)_{(r_1,s_1)}$ dans lesquelles $\mathrm{E}=\{s_1=0\}.$ On remarque que  
\begin{align*}
&(u_1,v_1)\to(u_1,u_1v_1)_{(y,z)}\to(v_1^2+u_1:v_1^2u_1:v_1^3u_1)\\
&\hspace{1cm}=\left(\frac{v_1^2u_1}{v_1^2+u_1},\frac{v_1^3u_1}{v_1^2+u_1}\right)_{(y,z)}
\to\left(\frac{v_1^2u_1}{v_1^2+u_1},v_1\right)_{(u_1,v_1)}
\end{align*}

\noindent et
\begin{align*} &(r_1,s_1)\to(r_1s_1,s_1)_{(y,z)}\to(1+r_1^3s_1:r_1s_1:s_1)\\
&\hspace{1cm}=\left(\frac{r_1s_1}{1+r_1^3s_1},\frac{s_1}{1+r_1^3s_1}\right)_{(y,z)}\to\left(r_1,\frac{s_1}{1+r_1^3s_1}\right)_{(r_1,s_1)};
\end{align*}

\noindent il en r\'esulte que $P_1$ est un point d'ind\'etermination, $\Delta_1$ est contract\'e sur $P_1$ et $\mathrm{E}$ est fix\'e.

\noindent \'Eclatons $P_1:$ 

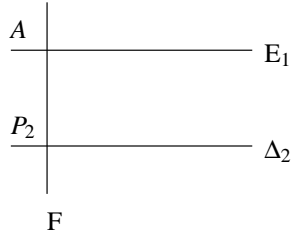
\begin{figure}[H]
\begin{center}
\input{ecl2.pstex_t}
\end{center}
\end{figure}

\noindent Posons $u_1=u_2$ et $v_1=u_2v_2.$ Remarquons que
$(u_2,v_2)$ sont des coordonn\'ees au voisinage de~$P_2=(0,0)_{(u_2,v_2)}$
dans lesquelles $\Delta_2=\{v_2=0\}$ et $\mathrm{F}=\{u_2=0\}.$ Si on pose
$u_1=r_2s_2$ et~$v_1=s_2$ alors $(r_2,s_2)$ sont des coordonn\'ees au voisinage de $A=(0,0)_{(r_2,s_2)};$ dans ces coordonn\'ees $\mathrm{F}=\{s_2=0\}.$ De plus
$$(u_2,v_2)\to(u_2,u_2v_2)_{(u_1,v_1)}\to(1+u_2v_2^2:u_2^2v_2^2:u_2^3v_2^3)$$ et
$$(r_2,s_2)\to(r_2s_2,s_2)_{(r_1,s_1)}\to(r_2+s_2:r_2s_2^2:r_2s_2^3).$$

\noindent Notons que $A$ est d'ind\'etermination. On a aussi
\begin{align*}
&(u_2,v_2)\to(u_2,u_2v_2)_{(u_1,v_1)}\to(1+u_2v_2^2:u_2^2v_2^2:u_2^3v_2^3)
\to\left(\frac{u_2^2v_2^2}{1+u_2v_2^2},\frac{u_2^3v_2^3}{1+u_2v_2^2}\right)_{(y,z)}\\
&\hspace{1cm}\to\left(\frac{u_2^2v_2^2}{1+u_2v_2^2},u_2v_2\right)_{(u_1,v_1)}
\to\left(\frac{u_2v_2}{1+u_2v_2^2},u_2v_2\right)_{(r_2,s_2)}
\end{align*}

\noindent d'o\`u $\mathrm{F}$ et $\Delta_2$ sont contract\'es sur $A.$

\noindent \'Eclatons $A:$ 

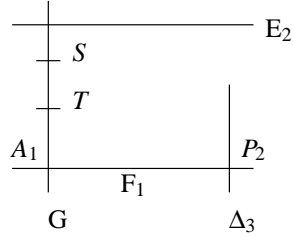
\begin{figure}[H]
\begin{center}
\input{ecl3.pstex_t}
\end{center}
\end{figure}

\noindent Posons $r_2=u_3$ et $s_2=u_3v_3;$  
$(u_3,v_3)$ sont des coordonn\'ees au voisinage de $A_1=(0,0)_{(u_3,v_3)},$ coordonn\'ees dans lesquelles
$\mathrm{F}_1=\{v_3=0\}$ et $\mathrm{G}=\{u_3=0\}.$ Si $r_2=
r_3s_3$ et $s_2=s_3,$ alors~$(r_3,s_3)$ est un syst\`eme de coordonn\'ees au voisinage 
dans lequel $\mathrm{E}_2=\{r_3=0\}$ et $\mathrm{G}=\{s_3=0\}.$ On~a
$$(u_3,v_3)\to(u_3,u_3v_3)_{(r_2,s_2)}\to(1+v_3:u_3^2v_3^2:
u_3^3v_3^3),$$ $$(r_3,s_3)\to(r_3s_3,s_3)_{(r_2,s_2)}\to(1+r_3:
r_3s_3^2:r_3s_3^3).$$

\noindent Le point $T=(-1,0)_{(r_3,s_3)}$ est d'ind\'etermination. 
De plus 
\begin{align*}
&(u_3,v_3)\to\left(\frac{u_3^2v_3^2}{1+v_3},\frac{u_3^3v_3^3}{1+v_3}\right)_{(y,z)}
\to\left(\frac{u_3^2v_3^2}{1+v_3},u_3v_3\right)_{(u_1,v_1)}\\
&\hspace{1cm}\to\left(\frac{u_3v_3}{1+v_3},u_3v_3\right)_{(r_2,s_2)}\to\left(\frac{1}{1+v_3},u_3v_3\right)_{(r_3,s_3)};
\end{align*}

\noindent d'o\`u $\mathrm{G}$ est fix\'e et $\mathrm{F}_1$ contract\'e sur $S=(1,0)_{(r_3,
s_3)}.$

\noindent \'Eclatons $T$ \`a gauche et $S$ \`a droite
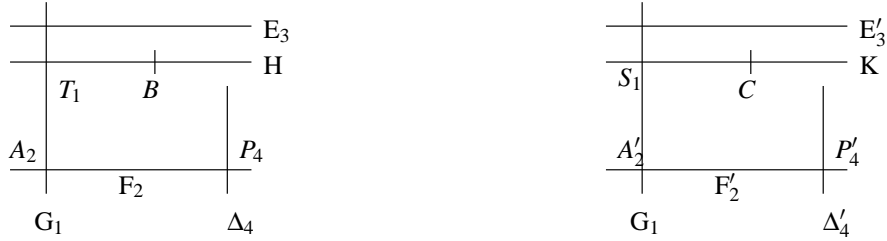
\begin{center}
\begin{tabular}{cc}
\input{ecl4.pstex_t} \hspace{4cm}&\hspace{4cm}\input{ecl6.pstex_t} 
\end{tabular}
\end{center}

\noindent Posons $r_3=u_4-1$ et $s_3=u_4v_4;$ dans le syst\`eme de coordonn\'ees $(u_4,v_4)$ on a
$\mathrm{G}_1=\{v_4=0\}$ et $\mathrm{H}=\{u_4=0\}.$ Notons aussi que $(r_4,
s_4),$ o\`u $r_3=r_4s_4-1$ et $s_3=s_4,$ est un syst\`eme de coordonn\'ees dans lequel $\mathrm{H}=\{s_4=0\}.$ D'une part
\begin{align*}
&(u_4,v_4)\to(u_4-1,u_4v_4)_{(r_3,s_3)}\to((u_4-1)u_4v_4^2,(u_4-1)u_4^2
v_4^3)_{(y,z)}\\
&\hspace{1cm}\to((u_4-1)u_4v_4^2,u_4v_4)_{(u_1,v_1)}
\to((u_4-1)v_4,u_4v_4)_{(r_2,s_2)}
\to\left((u_4-1)v_4,\frac{u_4}{u_4-1}\right)_{(u_3,v_3)}
\end{align*}
\noindent par suite $\mathrm{H}$ est envoy\'e sur $\mathrm{F}_2;$
 
\noindent d'autre part
$$(r_4,s_4)\to(r_4s_4-1,s_4)_{(r_3,s_3)}\to(r_4:(r_4s_4-1)s_4:(r_4s_4-1)s_4^2);$$

\noindent il s'en suit que $B=(0,0)_{(r_4,s_4)}$ est d'ind\'etermination.

\medskip

\noindent Posons $r_3=a_4+1,$ $s_3=a_4b_4;$ $(a_4,b_4)$
sont des coordonn\'ees dans lesquelles $\mathrm{G}_1=\{b_4=0\}$ et 
$\mathrm{K}=\{a_4=0\}.$ On peut aussi poser $r_3=c_4d_4+1$ et $s_3=d_4$ auquel cas 
$(c_4,d_4)$ sont des coordonn\'ees dans lesquelles $\mathrm{K}$ est donn\'e par $d_4=0.$

\noindent Remarquons que
$$(u_3,v_3)\to\left(\frac{1}{1+v_3},u_3v_3\right)_{(r_3,s_3)}\to\left(-\frac{v_3}{1+v_3},
-u_3(1+v_3)\right)_{(a_4,b_4)};$$

\noindent par suite $\mathrm{F}_2$ est envoy\'e sur $\mathrm{K}.$

\noindent On constate que 
\begin{align*}
&(u_1,v_1)\to(v_1^2+u_1:u_1v_1^2:u_1v_1^3)=\left(\frac{u_1v_1^2}{u_1+v_1^2},\frac{u_1v_1^3}{u_1+v_1^2}\right)_{(y,z)}\to\left(\frac{u_1v_1^2}{u_1+v_1^2},v_1\right)_{(u_1,v_1)}\\
&\hspace{1cm} \to\left(\frac{u_1v_1}{u_1+v_1^2},v_1\right)_{(r_2,s_2)}\to\left(\frac{u_1}{u_1+v_1^2},v_1\right)_{(r_3,s_3)}\to\left(-\frac{v_1}{u_1+v_1^2},v_1\right)_{(c_4,d_4)};
\end{align*}

\noindent ainsi $\Delta_4$ est contract\'e sur $C=(0,0)_{(c_4,d_4)}.$

\noindent Maintenant \'eclatons $B$ \`a gauche et $C$ \`a droite.
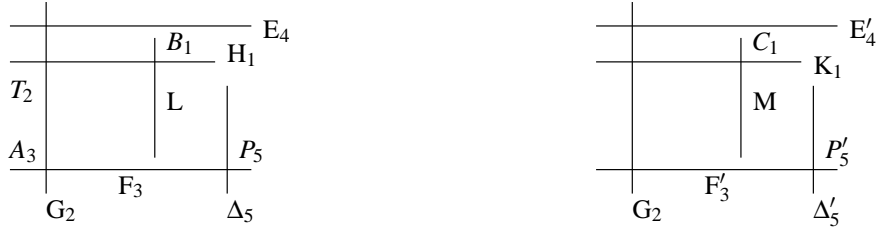
\begin{center}
\begin{tabular}{cc}
\input{ecl5.pstex_t}\hspace{4cm}&\hspace{4cm}\input{ecl7.pstex_t}
\end{tabular}
\end{center}

\noindent Posons $r_4=u_5,$ $s_4=u_5v_5$ et $r_4=r_5s_5,$ $s_4=s_5.$
Alors $(u_5,v_5)$ (resp. $(r_5,s_5)$) est un syst\`eme de coordonn\'ees dans lequel $\mathrm{L}=\{u_5=0\}$ (resp.
$\mathrm{H}_1=\{v_5=0\}$ et $\mathrm{L}=\{s_5=0\}$). Notons que
$$(u_5,v_5)\to(u_5,u_5v_5)_{r_4,s_4}\to(1:v_5(u_5^2v_5-1):u_5v_5^2(u_5^2v_5-1))$$ et 
$$r_5,s_5)\to(r_5s_5,s_5)_{r_4,s_4}\to(r_5:r_5s_5^2-1:s_5(r_5s_5^2-1)).$$

\noindent On en d\'eduit que $\mathrm{L}$ est envoy\'e sur $\Delta_5$ et qu'il n'y a pas de point d'ind\'etermination.

\noindent Posons $c_4=a_5,$ $d_4=a_5b_5$ et $c_4=c_5d_5,$ $d_4=d_5.$ Dans le premier (resp. second) syst\`eme de coordonn\'ees le diviseur exceptionnel $\mathrm{M}$ est donn\'e par $\{a_5=0\}$ (resp. $\{d_5=0\}$). Notons que 
$$(u_1,v_1)\to\left(-\frac{v_1}{u_1+v_1^2},v_1\right)_{(c_4,d_4)}\to\left(-\frac{1}{u_1+v_1^2},v_1\right)_{(c_5,d_5)};$$

\noindent en particulier $\Delta_5$ est envoy\'e sur $\mathrm{M}.$

\noindent La construction qui pr\'ec\`ede permet d'\'enoncer le r\'esultat suivant.

\begin{pro}[\cite{DG}]\label{isomorphism}
{\sl Soit $\widehat{P}_1$ (resp. $\widehat{P}_2$) le point infiniment proche de $P$ 
obtenu en \'eclatant $\mathbb{P}^2(\mathbb{C})$ at $P,$ $P_1,$ $A,$ $T$ et $U$ (resp.
$P,$ $P_1,$ $A,$ $S$ et $U'$).

\noindent La transformation $\Phi$ induit un isomorphisme entre $\mathrm{Bl}_{\widehat{P}_1}\mathbb{P}^2$ et $\mathrm{Bl}_{\widehat{P}_2}\mathbb{P}^2.$}
\end{pro}

\noindent Les diff\'erentes composantes sont \'echang\'ees comme suit
\begin{align*}
& \Delta\to \mathrm{M}, && \mathrm{E}\to \mathrm{E}, && \mathrm{F}\to \mathrm{K}, && \mathrm{G}\to \mathrm{G},
&& \mathrm{H}\to \mathrm{F}, && \mathrm{L}\to\Delta.
\end{align*}

\noindent Soit $\varphi$ un automorphisme de $\mathbb{P}^2(\mathbb{C}).$ 
Dans le paragraphe qui suit on va ajuster $\varphi$ de sorte que~$(\varphi \Phi)^2\varphi$ envoie $\widehat{P}_1$ sur $\widehat{P}_2.$

\subsection{Conditions de recollement}

\noindent Soit $g$ un germe de biholomorphisme de $(\mathbb{C}^2,0)$ donn\'e par $$g(y,z)=\Big(\sum_{0\leq i,j\leq 4} m_{i,j}y^i z^j,\sum_{0\leq i,j\leq 4} n_{i,j}y^i z^j\Big)+ o(\vert\vert(y,z)\vert\vert^4).$$ On remarque que $g(0,0)=(m_{0,0},n_{0,0})$ donc $g(P)=P$ si et seulement si $m_{0,0}=n_{0,0}=0.$ Dor\'enavant nous supposerons que $m_{0,0}=n_{0,0}=0.$ D\'esignons par $g^1$ (resp. $g^2$) la premi\`ere (resp. seconde) composante de $g.$

\noindent Posons $y=u_1$ et $z=u_1v_1.$ Consid\'erons la transformation $g_1$ donn\'ee par $$g_1(u_1,v_1)=\left(g^1(u_1,u_1v_1),\frac{g^2(u_1,u_1v_1)}{g^1(u_1,u_1v_1)}\right).$$ On remarque que $g_1(0,0)=\left(0,\frac{n_{1,0}}{m_{1,0}}\right);$ ainsi $g_1(0,0)=(0,0)$ si et seulement si $n_{1,0}=0.$ Supposons dans la suite que $n_{1,0}=0.$ 

\noindent Posons $u_1=r_2s_2$ et $v_1=s_2.$ On d\'efinit $g_2$ par $$g_2(r_2, s_2)=\left(\frac{g_1^1(r_2s_2,s_2)}{g_1^2(r_2s_2,s_2)},g_1^2(r_2s_2,s_2)\right).$$ Alors $g_2(0,0)=(0 ,0),$ {\it i.e.} $g_2(A)=A.$

\noindent Posons $r_2=r_3s_3$ et $s_2=s_3.$ Soit $g_3$ la transformation donn\'ee par $$g_3(s
r_3,s_3)=\left(\frac{g_2^1(r_3s_3,s_3)}{g_2^2(r_3s_3,s_3)},g_2^2(r_3s_3,s_3)\right).$$ On peut v\'erifier que $g_3(1,0)=\left(\frac{m_{1,0}^3}{n_{0,1}^2},0\right).$ Il s'en suit que $g_3(S)=T$ si et seulement si~$\frac{m_{1,0}^3}{n_{0,1}^2}=~-1.$ Supposons que $m_{1,0}=t^2$ et que $n_{0,1}=\mathrm{i}t^3$ auquel cas $\frac{m_{1,0}^3}{n_{0,1}^2}=~-1.$

\noindent Finalement si $r_3=r_4s_4+1$ et $s_3=s_4,$ on pose $$g_4(r_4,s_4)=\left(\frac{g_3^1(r_4s_4+1,s_4)+1}{g_3^2(r_4s_4+1,s_4)}, g_3^2(r_4s_4+1,s_4)\right).$$
On note que $g_4(0,0)=\left(\frac{\mathrm{i}(3m_{0,1}t+2\mathrm{i}n_{2,0})}{t^4}, 0\right).$
On en d\'eduit que $g_4(0,0)=(0,0)$ si et seulement si $3m_{0,1}t+2\mathrm{i}n_{2,0}=0.$

\noindent Ceci permet d'\'enoncer la proposition qui suit.

\begin{pro}[\cite{DG}]\label{gluglu}
{\sl Un germe $g$ de biholomorphisme de $(\mathbb{C}^2,0)$ donn\'e par $$g(y,z)=\Big(\sum_{0\leq i,j\leq 4} m_{i,j}y^i z^j,\sum_{0\leq i,j\leq 4} n_{i,j}y^i z^j\Big)+o(\vert\vert(y,z)\vert \vert^4)$$ envoie $\widehat{P}_2$ sur $\widehat{P}_1$ si et seulement si
\begin{align*}
&m_{0,0}=n_{0,0}=0, &&n_{1,0}=0, && m_{1,0}=t^2,\,\,n_{0,1}=\mathrm{i}t^3, &&
3m_{0,1}t+2\mathrm{i}n_{2,0}=0.
\end{align*} }
\end{pro}

\section{Une premi\`ere famille d'automorphismes}\label{15points}

\noindent La construction d\'etaill\'ee pr\'ec\'edemment permet d'exhiber des familles d'automorphismes d'entropie positive. Avant de donner celle-ci introduisons la notion suivante. Soient $U$ un ouvert de  $\mathbb{C}^d,$ $\varphi\colon U\to\mathrm{PGL}(3;\mathbb{C})$ une transformation holomorphe et $f$ une transformation de \textsc{Cremona}. La famille de transformations birationnelles $(\varphi_{\alpha_1,\,\ldots,\,\alpha_n}f)_{(\alpha_1,\,\ldots,\,\alpha_n)\in U}$ est {\it localement holomorphiquement triviale}\label{888} si pour tout $\alpha_0$ dans $U$ il existe une application holomorphe d'un voisinage $U_{\alpha_0}$ de $\alpha_0$ dans $\mathrm{PGL}(3;\mathbb{C})$ telle que 
\begin{align*}
&M_{\alpha_0}=\mathrm{Id}, &&\forall\,\alpha\in U_{\alpha_0},\,\, \varphi_{\alpha}f=M_{\alpha}
(\varphi_{\alpha_0}f)M_{\alpha}^{-1}.
\end{align*}

\begin{thm}[\cite{DG}]
{\sl  Soit $\varphi_\alpha$ l'automorphisme du plan projectif complexe d\'efini par 
\begin{align*}
&\varphi_\alpha=\left[\begin{array}{ccc}\alpha & 2(1-\alpha) & 2+\alpha-\alpha^2\\ -1 & 0 & \alpha+1\\ 1 &  -2 & 1-\alpha\end{array}\right], &&\alpha\in\mathbb{C}\setminus\{0,\,1\}.
\end{align*}

\noindent La transformation $\varphi_\alpha\Phi$ n'a pas de droite invariante et est conjugu\'ee \`a un automorphisme de~$\mathbb{P}^2(\mathbb{C})$ \'eclat\'e en $15$ points.

\noindent Le premier degr\'e dynamique de $\varphi_\alpha\Phi$ est $\frac{3+\sqrt{5}}{2}>1.$

\noindent La famille $\varphi_\alpha\Phi$ est localement holomorphiquement triviale.}
\end{thm}

\begin{proof}[{\sl D\'emonstration}]
La premi\`ere assertion est donn\'ee par la Proposition \ref{gluglu}.

\noindent Passons \`a la seconde. Les diff\'erentes composantes sont \'echang\'ees comme suit (\S\ref{cons})
\begin{small}
\begin{align*}
& \Delta\to \varphi \mathrm{M},&& \mathrm{E}\to \varphi \mathrm{E}, && \mathrm{F}\to \varphi \mathrm{K}, && \mathrm{G}\to \varphi \mathrm{G}, && \mathrm{H}\to \varphi \mathrm{F}, &&
 \mathrm{L}\to \varphi \Delta,\\
& \varphi \mathrm{E}\to\varphi\Phi \varphi \mathrm{E},&&
\varphi \mathrm{F} \to\varphi \Phi\varphi \mathrm{F}, 
&&\varphi \mathrm{G} \to\varphi \Phi\varphi \mathrm{G}, &&
\varphi \mathrm{K} \to\varphi \Phi\varphi \mathrm{K}, 
&&\varphi \mathrm{M} \to\varphi \Phi\varphi \mathrm{M}, &&
\varphi \Phi \varphi \mathrm{E}\to \mathrm{E},\\
&\varphi  \Phi\varphi \mathrm{F}\to \mathrm{F},&&
\varphi \Phi\varphi \mathrm{G}\to \mathrm{G},&&
\varphi \Phi\varphi \mathrm{K}\to \mathrm{H},&&
\varphi \Phi\varphi \mathrm{M}\to \mathrm{L},&& &&.
\end{align*}
\end{small}

\noindent Ainsi dans la base $\{\Delta,\,\mathrm{E},\,
\mathrm{F},\,\mathrm{G},\,\mathrm{H},\,\mathrm{L},\,\varphi\mathrm{E},\,\varphi\mathrm{F},
\,\varphi\mathrm{G},\,\varphi\mathrm{K},\,\varphi\mathrm{M}\,\varphi\Phi \varphi\mathrm{E},\,\varphi\Phi \varphi\mathrm{F},
\,\varphi\Phi \varphi\mathrm{G},\,\varphi\Phi \varphi\mathrm{K},\,\varphi\Phi \varphi\mathrm{M}\}$ la matrice de $(\varphi\Phi)^*$ est
\begin{small}
$$\left[\begin{array}{cccccccccccccccc}
0 & 0 & 0 & 0 & 0 & 1  & 0 & 0 & 0 & 0 & 0 & 0 & 0 & 0 & 0 & 0\\
0 & 0 & 0 & 0 & 0 & 1  & 0 & 0 & 0 & 0 & 0 & 1 & 0 & 0 & 0 & 0\\
0 & 0 & 0 & 0 & 0 & 2  & 0 & 0 & 0 & 0 & 0 & 0 & 1 & 0 & 0 & 0\\
0 & 0 & 0 & 0 & 0 & 3  & 0 & 0 & 0 & 0 & 0 & 0 & 0 & 1 & 0 & 0\\
0 & 0 & 0 & 0 & 0 & 3  & 0 & 0 & 0 & 0 & 0 & 0 & 0 & 0 & 1 & 0\\
0 & 0 & 0 & 0 & 0 & 3  & 0 & 0 & 0 & 0 & 0 & 0 & 0 & 0 & 0 & 1\\
0 & 1 & 0 & 0 & 0 & -1 & 0 & 0 & 0 & 0 & 0 & 0 & 0 & 0 & 0 & 0\\
0 & 0 & 0 & 0 & 1 & -2 & 0 & 0 & 0 & 0 & 0 & 0 & 0 & 0 & 0 & 0\\
0 & 0 & 0 & 1 & 0 & -3 & 0 & 0 & 0 & 0 & 0 & 0 & 0 & 0 & 0 & 0\\
0 & 0 & 1 & 0 & 0 & -3 & 0 & 0 & 0 & 0 & 0 & 0 & 0 & 0 & 0 & 0\\
1 & 0 & 0 & 0 & 0 & -3 & 0 & 0 & 0 & 0 & 0 & 0 & 0 & 0 & 0 & 0\\
0 & 0 & 0 & 0 & 0 & 0  & 1 & 0 & 0 & 0 & 0 & 0 & 0 & 0 & 0 & 0\\
0 & 0 & 0 & 0 & 0 & 0  & 0 & 1 & 0 & 0 & 0 & 0 & 0 & 0 & 0 & 0\\
0 & 0 & 0 & 0 & 0 & 0  & 0 & 0 & 1 & 0 & 0 & 0 & 0 & 0 & 0 & 0\\
0 & 0 & 0 & 0 & 0 & 0  & 0 & 0 & 0 & 1 & 0 & 0 & 0 & 0 & 0 & 0\\
0 & 0 & 0 & 0 & 0 & 0  & 0 & 0 & 0 & 0 & 1 & 0 & 0 & 0 & 0 & 0\\
\end{array}\right]$$
\end{small}

\noindent et son polyn\^ome caract\'eristique est donn\'e par  $$(X^2-3X+1)(X^2-X+1)(X+1)^2(X^2+X+1)^3(X-1)^4.$$ Ainsi
$$\lambda(\varphi\Phi)=\frac{3+\sqrt{5}}{2}>1.$$

\noindent Reste \`a justifier la derni\`ere assertion: soit $\alpha_0$ dans $\mathbb{C}\setminus\{0,\,1\}.$ On peut trouver localement autour de $\alpha_0$ une matrice $M_\alpha$ d\'ependant holomorphiquement de $\alpha$ et tel que pour tout $\alpha$ proche de $\alpha_0$ on ait $\varphi_\alpha\Phi=M_\alpha^{-1}\varphi_{\alpha_0}\Phi M_\alpha.$ La matrice $$M_\alpha=\left[\begin{array}{ccc}1 & 0 & \alpha_0-\alpha \\ 0 & 1 & 0 \\ 0 & 0 & 1\end{array}\right]$$ convient.
\end{proof}

\section{D'autres exemples}

\noindent Comme on l'a d\'ej\`a mentionn\'e, la construction d\'etaill\'ee pr\'ec\'edemment est dans \cite{DG} faite pour $\Phi=(xz^{n-1}+y^n:yz^{n-1}:z^n)$ pour $n\geq 3.$ On peut donc obtenir des familles d'automorphismes d'entropie positive pour $n\geq 3.$

\begin{thm}[\cite{DG}]
{\sl Supposons que $n\geq 3$ et que 
\begin{small}
\begin{align*}
&\varphi_{\alpha,\beta}=\left[\begin{array}{ccc}1 & \frac{2\beta}{\alpha} & -\frac{1+\delta_k+\delta_k^2}{\alpha} \\
0 & -1 & 0 \\
\alpha &\beta &\delta_k 
\end{array}\right]
\end{align*}
\end{small}

\noindent avec $\alpha\in\mathbb{C}^*,$ $\beta\in\mathbb{C}$ et $\delta_k=\mathrm{exp} \Big(\frac{(2k+1)\mathrm{i}\pi}{3n}\Big)-1,$ $0\leq k\leq 3n-1.$

\noindent La transformation $\varphi_{\alpha,\beta}\Phi$ est conjugu\'ee \`a un automorphisme de $\mathbb{P}^2(\mathbb{C})$ \'eclat\'e en $3(2n-1)$ points.

\noindent Le premier degr\'e dynamique de~$\varphi_{\alpha,\beta}\Phi$ est $\lambda(\varphi_{\alpha,\beta}\Phi)=\frac{n+\sqrt{n^2-4}}{2}>1.$

\noindent La famille $\varphi_{\alpha,\beta}\Phi$ est localement holomorphiquement triviale.}
\end{thm}

\begin{thm}[\cite{DG}]
{\sl Supposons que $n\geq 4$ et
\begin{align*}
&\varphi_{\alpha,\beta,\gamma,\delta}=\left[\begin{array}{ccc}
 \alpha &\beta &\frac{\beta(\gamma^2 \epsilon_k-\alpha^2)}{\delta(\alpha-\gamma)}\\
0 & \gamma & 0\\
\frac{\delta( \alpha-\gamma)}{\beta} & \delta & -\alpha 
\end{array}\right]
\end{align*}

\noindent avec $\alpha,\,\beta\in\mathbb{C},$ $\gamma,\,\delta\in\mathbb{C}^*,$ $\alpha\not=\gamma,$ $\epsilon_k=\mathrm{exp} \, \Big(\dfrac{(2k+1)\mathrm{i}\pi}{n}\Big),$ $0\leq k\leq n-1.$

\noindent La transformation $\varphi_{\alpha,\beta,\gamma,\delta}\Phi$ est conjugu\'ee \`a un automorphisme de $\mathbb{P}^2(\mathbb{C})$ \'eclat\'e en $4n-2$ points.

\noindent Le premier degr\'e dynamique de~$\varphi_{\alpha,\beta,\gamma,\delta}\Phi$ est $\lambda(\varphi_{\alpha,\beta,\gamma,\delta}\Phi)=\frac{(n-1)+\sqrt{(n-1)^2-4}}{2}>1.$

\noindent La famille $\varphi_{\alpha,\beta,\gamma,\delta}\Phi$ est localement holomorphiquement triviale.}
\end{thm}

\noindent De plus, dans \cite{DG}, les auteurs utilisent la m\^eme m\'ethode avec une transformation birationnelle qui contracte une droite et une conique et obtiennent l\`a encore une famille d'automorphismes d'entropie positive sur le plan projectif \'eclat\'e en un nombre fini de points.

\begin{thm}
{\sl Soient $f=(y^2z:x(xz+y^2):y(xz+y^2))$ et
$$\varphi_\alpha=\left[\begin{array}{ccc}\frac{2\alpha^3}{343}(37\mathrm{i}\sqrt{3}+3)& \alpha& -\frac{2\alpha^{2}}{49}(5\mathrm{i}\sqrt{3}+11)\\\frac{\alpha^2}{49}(-15+11\mathrm{i}\sqrt{3}) & 1 & -\frac{\alpha}{14}(5\mathrm{i}\sqrt{3}+11)\\ -\frac{\alpha}{7}(2\mathrm{i}\sqrt{3}+3)& 0 & 0\end{array} \right]$$

\noindent avec $\alpha \in\mathbb{C}^*.$

\noindent La transformation $\varphi_\alpha f$ est conjugu\'ee \`a un automorphisme de $\mathbb{P}^2(\mathbb{C})$  \'eclat\'e en $15$ points.

\noindent Le premier degr\'e dynamique de $\varphi_\alpha f$ est $\frac{3+\sqrt{5}}{2}>1.$

\noindent The associated family of rational surfaces parameterized by $\alpha$ is locally holomorphically trivial.}
\end{thm}

\backmatter

\chapter*{Index}

\newlength{\largeur}
\setlength{\largeur}{\textwidth} 

\begin{tabular}{p{\largeur}}

\noteBB{alg\'ebriquement stable}{ind20}{ind26r}

\noteB{anneau de \textsc{Herman}}{ind58}

\noteB{automorphisme de \textsc{H\'enon} g\'en\'eralis\'e}{ind6}

\noteB{automorphisme de type \'el\'ementaire}{ind7}

\noteB{automorphisme de type \textsc{H\'enon}}{ind8}

\noteB{automorphisme polynomial de $\mathbb{C}^2$}{ind1}

\noteB{base g\'eom\'etrique}{ind40}

\noteB{bassin}{ind56}

\noteB{composante de \textsc{Fatou} r\'ecurrente}{ind54}

\noteB{configuration exceptionnelle}{ind41}

\noteB{couple $(\mathcal{X},g)$ minimal}{ind59}

\noteB{courbes contract\'ees}{ind17}

\noteB{cubique}{ind800}

\noteB{cubique marqu\'ee}{ind801}

\noteB{degr\'e alg\'ebrique d'un automorphisme polynomial}{ind9}

\noteB{degr\'e alg\'ebrique}{ind18}

\noteB{degr\'e dynamique d'une transformation birationnelle}{ind19}

\noteB{degr\'e dynamique sur une surface complexe compacte k\"{a}hl\'erienne}{ind23}

\noteB{dimension de \textsc{Kodaira}-\textsc{Iitaka}}{ind888} 

\noteB{dimension de \textsc{Kodaira}}{ind88}

\noteB{disque de \textsc{Siegel}}{ind57} 

\noteB{diviseur}{ind28}

\noteB{diviseur exceptionnel}{ind26b}

\noteB{domaine de rotation}{ind51}

\noteB{dominante (application)}{ind26y}

\noteB{\'eclatement d'une vari\'et\'e}{ind26}

\noteB{\'eclatement marqu\'e}{ind666}

\noteB{entropie topologique}{ind33b}

\noteB{ensemble critique}{ind26x}

\noteBB{ensemble d'ind\'etermination}{ind14}{ind26w}
\end{tabular}

\begin{tabular}{p{\largeur}}
\noteBB{ensemble exceptionnel}{ind16}{ind26t}

\noteB{ensemble de \textsc{Fatou}}{ind50}

\noteB{\'el\'ement de \textsc{Coxeter}}{ind43}

\noteB{\'el\'ement standard du groupe de \textsc{Weyl}}{ind44}

\noteB{formellement lin\'earisable}{ind45}

\noteB{foyer}{ind55}

\noteB{groupe des automorphismes polynomiaux de~$\mathbb{C}^2$}{ind2}

\noteB{groupe des automorphismes affines}{ind3}

\noteB{groupe des automorphismes \'el\'ementaires}{ind4}

\noteB{groupe des automorphismes affines triangulaires}{ind5} 

\noteB{groupe de \textsc{Cremona}}{ind12}

\noteB{groupe de \textsc{Picard}}{ind30}

\noteB{groupe de \textsc{Weyl}}{ind42}

\noteB{isomorphisme entre cubiques marqu\'ees}{ind802}

\noteB{isomorphisme entre \'eclatements marqu\'es}{ind666a}

\noteB{isomorphisme entre paires marqu\'ees}{ind821} 

\noteB{localement holomorphiquement triviale}{ind888}

\noteB{lin\'eairement \'equivalents}{ind29}

\noteB{matrice caract\'eristique}{ind38}

\noteB{modification propre}{ind26z} 

\noteB{mon\^ome r\'esonant}{ind47}

\noteB{multiplicativement ind\'ependants}{ind48} 

\noteB{nombre de \textsc{Salem}}{ind25}

\noteB{paire marqu\'ee}{ind820}

\noteB{points \'eclat\'es}{ind15}

\noteB{point propre}{ind36}

\noteB{polyn\^ome de Salem}{ind39}

\noteB{racine}{ind7000}

\noteB{racine simple}{ind700}

\noteB{racine positive}{ind701}

\noteB{rang du domaine de rotation}{ind53}

\noteB{r\'esolution ordonn\'ee}{ind37}

\noteB{r\'esonance}{ind46}

\noteB{simultan\'ement diophantiens}{ind49}

\noteB{surface K$3$}{ind33}

\noteB{surface de \textsc{Coble}}{ind100}

\noteB{surface de \textsc{Enriques}}{ind34}

\noteB{surface minimale}{ind35}

\noteB{surface de \textsc{Hirzebruch}}{ind31}

\noteB{transformation \'el\'ementaire}{ind32}

\noteB{transformation rationnelle}{ind10}

\noteB{transformation birationnelle}{ind11}

\noteB{transformation de \textsc{Cremona}}{ind13}

\noteB{transform\'ee stricte}{ind27}

\noteB{transform\'ee totale}{ind27b}

\noteB{vari\'et\'e k\"{a}hl\'erienne}{ind22}

\noteB{vecteur positif}{ind701}

\noteB{valeur propre dominante}{ind702} 

\noteB{vecteur propre dominant}{ind703}
\end{tabular}

\bibliographystyle{alpha}
\bibliography{bibliodyn}

\begin{thebibliography}{BHPVdV04}

\bibitem[Bea78]{Be2}
A.~Beauville.
\newblock {\em Surfaces alg\'ebriques complexes}.
\newblock Soci\'et\'e Math\'ematique de France, Paris, 1978.
\newblock Avec une sommaire en anglais, Ast\'erisque, No. 54.

\bibitem[BHPVdV04]{BHPV}
W.~P. Barth, K.~Hulek, C.~A.~M. Peters, and A.~Van~de Ven.
\newblock {\em Compact complex surfaces}, volume~4 of {\em Ergebnisse der
  Mathematik und ihrer Grenzgebiete. 3. Folge. A Series of Modern Surveys in
  Mathematics}.
\newblock Springer-Verlag, Berlin, second edition, 2004.

\bibitem[BKa]{BK3}
E.~Bedford and K.~Kim.
\newblock Continuous families of rational surface automorphisms with positive
  entropy, {\tt arxiv:0804.2078}, 2008, to appear in {M}ath. {A}nn.

\bibitem[BKb]{BK4}
E.~Bedford and K.~Kim.
\newblock Dynamics of rational surface automorphisms: rotations domains, {\tt
  arxiv: 0907.3339}, 2009.

\bibitem[BK06]{BK1}
E.~Bedford and K.~Kim.
\newblock Periodicities in linear fractional recurrences: degree growth of
  birational surface maps.
\newblock {\em Michigan Math. J.}, 54(3):647--670, 2006.

\bibitem[BK09]{BK2}
E.~Bedford and K.~Kim.
\newblock Dynamics of rational surface automorphisms: linear fractional
  recurrences.
\newblock {\em J. Geom. Anal.}, 19(3):553--583, 2009.

\bibitem[Bou81]{Bo}
N.~Bourbaki.
\newblock {\em \'{E}l\'ements de math\'ematique}.
\newblock Masson, Paris, 1981.
\newblock Groupes et alg{\`e}bres de Lie. Chapitres 4, 5 et 6.

\bibitem[BR04]{BR}
G.~Bastien and M.~Rogalski.
\newblock Global behavior of the solutions of {L}yness' difference equation
  {$u_{n+2}u_n=u_{n+1}+a$}.
\newblock {\em J. Difference Equ. Appl.}, 10(11):977--1003, 2004.

\bibitem[BS91]{BS}
E.~Bedford and J.~Smillie.
\newblock Polynomial diffeomorphisms of {${\bf C}^2$}. {II}. {S}table manifolds
  and recurrence.
\newblock {\em J. Amer. Math. Soc.}, 4(4):657--679, 1991.

\bibitem[Can]{Can}
S.~Cantat.
\newblock Sur les groupes de transformations birationnelles des surfaces.
\newblock {\em To appear in {A}nn. of {M}ath.}

\bibitem[Can99]{Can1}
S.~Cantat.
\newblock Dynamique des automorphismes des surfaces projectives complexes.
\newblock {\em C. R. Acad. Sci. Paris S\'er. I Math.}, 328(10):901--906, 1999.

\bibitem[Can01]{Can2}
S.~Cantat.
\newblock Dynamique des automorphismes des surfaces {$K3$}.
\newblock {\em Acta Math.}, 187(1):1--57, 2001.

\bibitem[Cas01]{Ca}
G.~Castelnuovo.
\newblock Le trasformationi generatrici del gruppo cremoniano nel piano.
\newblock {\em Atti della R. Accad. delle Scienze di Torino}, 36:861--874,
  1901.

\bibitem[CD]{CeDe}
D.~Cerveau and J.~D\'eserti.
\newblock {\em Transformations birationnelles de petit degr\'e, {\tt arXiv:
  0811.2325}, 2008}.

\bibitem[Cob61]{Co}
A.~B. Coble.
\newblock {\em Algebraic geometry and theta functions}.
\newblock Revised printing. American Mathematical Society Colloquium
  Publication, vol. X. American Mathematical Society, Providence, R.I., 1961.

\bibitem[D{\'e}s08]{De}
J.~D{\'e}serti.
\newblock Exp\'eriences sur certaines transformations birationnelles
  quadratiques.
\newblock {\em Nonlinearity}, 21(6):1367--1383, 2008.

\bibitem[DF01]{DiFa}
J.~Diller and C.~Favre.
\newblock Dynamics of bimeromorphic maps of surfaces.
\newblock {\em Amer. J. Math.}, 123(6):1135--1169, 2001.

\bibitem[DG10]{DG}
J.~D\'eserti and J.~Grivaux.
\newblock Special automorphisms of rational surfaces with positive topological
  entropy, {\tt arxiv:1004.0656}.
\newblock 2010.

\bibitem[Dil]{Di2}
J.~Diller.
\newblock {\em Cremona transformations, surface automorphisms and the group
  law, {\tt arXiv: 0811.3038}, 2008}.

\bibitem[Dil96]{Di}
J.~Diller.
\newblock Dynamics of birational maps of {${\bf P}^2$}.
\newblock {\em Indiana Univ. Math. J.}, 45(3):721--772, 1996.

\bibitem[DJS07]{DJS}
J.~Diller, D.~Jackson, and A.~Sommese.
\newblock Invariant curves for birational surface maps.
\newblock {\em Trans. Amer. Math. Soc.}, 359(6):2793--2991 (electronic), 2007.

\bibitem[DO88]{DoOr}
I.~Dolgachev and D.~Ortland.
\newblock Point sets in projective spaces and theta functions.
\newblock {\em Ast\'erisque}, (165):210 pp. (1989), 1988.

\bibitem[Dol08]{Dol}
I.~V. Dolgachev.
\newblock Reflection groups in algebraic geometry.
\newblock {\em Bull. Amer. Math. Soc. (N.S.)}, 45(1):1--60 (electronic), 2008.

\bibitem[DS04]{DS}
T.-C. Dinh and N.~Sibony.
\newblock Regularization of currents and entropy.
\newblock {\em Ann. Sci. \'Ecole Norm. Sup. (4)}, 37(6):959--971, 2004.

\bibitem[DV36]{DV}
P.~Du~Val.
\newblock On the {K}antor group of a set of points in a plane.
\newblock {\em Proc. London Math. Soc.}, 42:18--51, 1936.

\bibitem[DZ01]{DZ}
I.~V. Dolgachev and D.-Q. Zhang.
\newblock Coble rational surfaces.
\newblock {\em Amer. J. Math.}, 123(1):79--114, 2001.

\bibitem[Fis76]{Fi}
G.~Fischer.
\newblock {\em Complex analytic geometry}.
\newblock Lecture Notes in Mathematics, Vol. 538. Springer-Verlag, Berlin,
  1976.

\bibitem[FM89]{FM}
S.~Friedland and J.~Milnor.
\newblock Dynamical properties of plane polynomial automorphisms.
\newblock {\em Ergodic Theory Dynam. Systems}, 9(1):67--99, 1989.

\bibitem[Giz80]{Gi}
M.~H. Gizatullin.
\newblock Rational {$G$}-surfaces.
\newblock {\em Izv. Akad. Nauk SSSR Ser. Mat.}, 44(1):110--144, 239, 1980.

\bibitem[Gro87]{Gr2}
M.~Gromov.
\newblock Entropy, homology and semialgebraic geometry.
\newblock {\em Ast\'erisque}, (145-146):5, 225--240, 1987.
\newblock S{\'e}minaire Bourbaki, Vol. 1985/86.

\bibitem[Gro03]{Gr1}
M.~Gromov.
\newblock On the entropy of holomorphic maps.
\newblock {\em Enseign. Math. (2)}, 49(3-4):217--235, 2003.

\bibitem[Har85]{Ha2}
B.~Harbourne.
\newblock Blowings-up of {${\bf P}^2$} and their blowings-down.
\newblock {\em Duke Math. J.}, 52(1):129--148, 1985.

\bibitem[Har87]{Ha}
B.~Harbourne.
\newblock Rational surfaces with infinite automorphism group and no
  antipluricanonical curve.
\newblock {\em Proc. Amer. Math. Soc.}, 99(3):409--414, 1987.

\bibitem[Har88]{Ha3}
B.~Harbourne.
\newblock Iterated blow-ups and moduli for rational surfaces.
\newblock In {\em Algebraic geometry ({S}undance, {UT}, 1986)}, volume 1311 of
  {\em Lecture Notes in Math.}, pages 101--117. Springer, Berlin, 1988.

\bibitem[Her87]{H}
M.-R. Herman.
\newblock Recent results and some open questions on {S}iegel's linearization
  theorem of germs of complex analytic diffeomorphisms of {${\bf C}^n$} near a
  fixed point.
\newblock In {\em V{III}th international congress on mathematical physics
  ({M}arseille, 1986)}, pages 138--184. World Sci. Publishing, Singapore, 1987.

\bibitem[Hir88]{Hi}
A.~Hirschowitz.
\newblock Sym\'etries des surfaces rationnelles g\'en\'eriques.
\newblock {\em Math. Ann.}, 281(2):255--261, 1988.

\bibitem[HM98]{HM}
J.~Harris and I.~Morrison.
\newblock {\em Moduli of curves}, volume 187 of {\em Graduate Texts in
  Mathematics}.
\newblock Springer-Verlag, New York, 1998.

\bibitem[Hum90]{Hum}
J.~E. Humphreys.
\newblock {\em Reflection groups and {C}oxeter groups}, volume~29 of {\em
  Cambridge Studies in Advanced Mathematics}.
\newblock Cambridge University Press, Cambridge, 1990.

\bibitem[HV00a]{HV2}
J.~Hietarinta and C.~Viallet.
\newblock Discrete {P}ainlev\'e {I} and singularity confinement in projective
  space.
\newblock {\em Chaos Solitons Fractals}, 11(1-3):29--32, 2000.
\newblock Integrability and chaos in discrete systems (Brussels, 1997).

\bibitem[HV00b]{HV}
J.~Hietarinta and C.~Viallet.
\newblock Singularity confinement and degree growth.
\newblock In {\em S{IDE} {III}---symmetries and integrability of difference
  equations ({S}abaudia, 1998)}, volume~25 of {\em CRM Proc. Lecture Notes},
  pages 209--216. Amer. Math. Soc., Providence, RI, 2000.

\bibitem[Isk85]{Is}
V.~A. Iskovskikh.
\newblock Proof of a theorem on relations in the two-dimensional {C}remona
  group.
\newblock {\em Uspekhi Mat. Nauk}, 40(5(245)):255--256, 1985.

\bibitem[Jun42]{Ju}
H.~W.~E. Jung.
\newblock \"{U}ber ganze birationale {T}ransformationen der {E}bene.
\newblock {\em J. Reine Angew. Math.}, 184:161--174, 1942.

\bibitem[Kan95]{Ka}
S.~Kantor.
\newblock {\em {T}heorie der endlichen {G}ruppen von eindeutigen
  {T}ransformationen in der {E}bene}.
\newblock 1895.

\bibitem[Koi88]{Ko}
M.~Koitabashi.
\newblock Automorphism groups of generic rational surfaces.
\newblock {\em J. Algebra}, 116(1):130--142, 1988.

\bibitem[Lam01]{La}
S.~Lamy.
\newblock L'alternative de {T}its pour {${\rm Aut}[\mathbb{C}^2]$}.
\newblock {\em J. Algebra}, 239(2):413--437, 2001.

\bibitem[Loo81]{Lo}
E.~Looijenga.
\newblock Rational surfaces with an anticanonical cycle.
\newblock {\em Ann. of Math. (2)}, 114(2):267--322, 1981.

\bibitem[Man86]{Manin}
Y.~I. Manin.
\newblock {\em Cubic forms}, volume~4 of {\em North-Holland Mathematical
  Library}.
\newblock North-Holland Publishing Co., Amsterdam, second edition, 1986.
\newblock Algebra, geometry, arithmetic, Translated from the Russian by M.
  Hazewinkel.

\bibitem[McM02]{Mc2}
C.~T. McMullen.
\newblock Dynamics on {$K3$} surfaces: {S}alem numbers and {S}iegel disks.
\newblock {\em J. Reine Angew. Math.}, 545:201--233, 2002.

\bibitem[McM07]{Mc}
C.~T. McMullen.
\newblock Dynamics on blowups of the projective plane.
\newblock {\em Publ. Math. Inst. Hautes \'Etudes Sci.}, (105):49--89, 2007.

\bibitem[Nag60]{Na}
M.~Nagata.
\newblock On rational surfaces. {I}. {I}rreducible curves of arithmetic genus
  {$0$}\ or {$1$}.
\newblock {\em Mem. Coll. Sci. Univ. Kyoto Ser. A Math.}, 32:351--370, 1960.

\bibitem[Nag61]{Na2}
M.~Nagata.
\newblock On rational surfaces. {II}.
\newblock {\em Mem. Coll. Sci. Univ. Kyoto Ser. A Math.}, 33:271--293,
  1960/1961.

\bibitem[Nik87]{Ni}
V.~V. Nikulin.
\newblock Discrete reflection groups in {L}obachevsky spaces and algebraic
  surfaces.
\newblock In {\em Proceedings of the {I}nternational {C}ongress of
  {M}athematicians, {V}ol. 1, 2 ({B}erkeley, {C}alif., 1986)}, pages 654--671,
  Providence, RI, 1987. Amer. Math. Soc.

\bibitem[Noe69]{No}
M.~Noether.
\newblock Ueber die auf {E}benen eindeutig abbildbaren algebraischen
  {F}l\"achen.
\newblock {\em G\"ottigen Nachr.}, pages 1--6, 1869.

\bibitem[Noe70]{No2}
M.~Noether.
\newblock Ueber {F}l\"achen, welche {S}chaaren rationaler {C}urven besitzen.
\newblock {\em Math. Ann.}, 3(2):161--227, 1870.

\bibitem[Noe72]{No3}
M.~Noether.
\newblock Zur {T}heorie der eindentigen {E}benentrasformationen.
\newblock {\em Math. Ann.}, 5(4):635--639, 1872.

\bibitem[Ogu]{Og}
K.~Oguiso.
\newblock {\em The third smallest Salem number in automorphisms of K3 surfaces,
  {\tt arXiv: 0905.2396}, 2009}.

\bibitem[PS81]{PS}
R.~Penrose and C.~A.~B. Smith.
\newblock A quadratic mapping with invariant cubic curve.
\newblock {\em Math. Proc. Cambridge Philos. Soc.}, 89(1):89--105, 1981.

\bibitem[Ser77]{Se}
J.-P. Serre.
\newblock {\em Arbres, amalgames, {${\rm SL}\sb{2}$}}.
\newblock Soci\'et\'e Math\'ematique de France, Paris, 1977.
\newblock Avec un sommaire anglais, R\'edig\'e avec la collaboration de Hyman
  Bass, Ast\'erisque, No. 46.

\bibitem[Sil91]{Si}
J.~H. Silverman.
\newblock Rational points on {$K3$} surfaces: a new canonical height.
\newblock {\em Invent. Math.}, 105(2):347--373, 1991.

\bibitem[Tak01a]{Ta1}
T.~Takenawa.
\newblock Algebraic entropy and the space of initial values for discrete
  dynamical systems.
\newblock {\em J. Phys. A}, 34(48):10533--10545, 2001.
\newblock Symmetries and integrability of difference equations (Tokyo, 2000).

\bibitem[Tak01b]{Ta2}
T.~Takenawa.
\newblock Discrete dynamical systems associated with root systems of indefinite
  type.
\newblock {\em Comm. Math. Phys.}, 224(3):657--681, 2001.

\bibitem[Tak01c]{Ta3}
T.~Takenawa.
\newblock A geometric approach to singularity confinement and algebraic
  entropy.
\newblock {\em J. Phys. A}, 34(10):L95--L102, 2001.

\bibitem[Wan95]{Wa}
L.~Wang.
\newblock Rational points and canonical heights on {$K3$}-surfaces in {${\bf
  P}^1\times{\bf P}^1\times{\bf P}^1$}.
\newblock In {\em Recent developments in the inverse {G}alois problem
  ({S}eattle, {WA}, 1993)}, volume 186 of {\em Contemp. Math.}, pages 273--289.
  Amer. Math. Soc., Providence, RI, 1995.

\bibitem[Wri92]{Wr}
D.~Wright.
\newblock Two-dimensional {C}remona groups acting on simplicial complexes.
\newblock {\em Trans. Amer. Math. Soc.}, 331(1):281--300, 1992.

\bibitem[Yom87]{Yo}
Y.~Yomdin.
\newblock Volume growth and entropy.
\newblock {\em Israel J. Math.}, 57(3):285--300, 1987.

\bibitem[Zha01]{Zh}
D.-Q. Zhang.
\newblock Automorphisms of finite order on rational surfaces.
\newblock {\em J. Algebra}, 238(2):560--589, 2001.
\newblock With an appendix by I. Dolgachev.

\end{thebibliography}
\nocite{}

\end{document}